\documentclass[10pt,a4paper]{article}
\usepackage[a4paper,left=2cm,right=2cm,top=2cm,bottom=2cm]{geometry}
\usepackage{color}
\usepackage[thinlines]{easytable}

\usepackage[backend=bibtex,bibencoding=ascii,giveninits=true,url=false,isbn=false,doi=false,style=numeric,maxnames=5]{biblatex}
\usepackage[utf8]{inputenc}
\usepackage{graphicx}
\usepackage{epstopdf}

\usepackage{amsmath}
\usepackage{amsthm}
\usepackage{amsfonts}
\usepackage{nicefrac}
\usepackage{dsfont}
\usepackage{enumitem}

\usepackage{tikz}
\usetikzlibrary{patterns}

\usepackage[
bookmarks,
bookmarksopen,
bookmarksopenlevel=0,
bookmarksnumbered,
breaklinks,
ocgcolorlinks,
colorlinks=true,
linkcolor=linkcolor,
citecolor=citecolor,
urlcolor=urlcolor,
linktoc=all
]{hyperref}

\hypersetup{
	pdftitle    = {Normal forms of 2D metrics admitting exactly one essential projective vector field},
	pdfauthor   = {Gianni Manno, Andreas Vollmer},
	pdfsubject  = {Lie's Problem}
}

\definecolor{linkcolor}{rgb}{0.5,0.0,0.0}
\definecolor{citecolor}{rgb}{0.0,0.5,0.0}
\definecolor{urlcolor} {rgb}{0.0,0.0,0.5}

\theoremstyle{plain}

\newtheorem{lemma}{Lemma}
\newtheorem{theorem}{Theorem}
\newtheorem{proposition}{Proposition}
\newtheorem{corollary}{Corollary}

\newtheorem{problem}{Problem}

\newtheorem{remark}{Remark}

\theoremstyle{definition}
\newtheorem{definition}{Definition}
\newtheorem*{notation}{Notation}

\newcommand{\cc}[1]{\ensuremath{\overline{#1}}}

\newcommand{\pe}[1]{\ensuremath{\check{#1}}}

\newcommand{\g}[1]{\ensuremath{g_{#1}}}
\renewcommand{\a}[1]{\ensuremath{a_{#1}}}
\newcommand{\K}[1]{\ensuremath{K_{#1}}}

\newcommand{\projalg}{\ensuremath{\mathfrak{p}}}
\newcommand{\projcl}{\ensuremath{\mathfrak{P}}}
\newcommand{\solSp}{\ensuremath{\mathfrak{A}}}

\newcommand{\lie}{\ensuremath{\mathcal{L}}}

\newcommand{\til}[1]{\ensuremath{\tilde{#1}}}

\newcommand{\h}[2]{\begin{minipage}{3.5cm}\centering \textbf{#1}\\ \textit{#2} \end{minipage}}
\newcommand{\vect}[1]{\ensuremath{\begin{pmatrix}#1\end{pmatrix}}}
\newcommand{\Rnz}{\ensuremath{\mathds{R}\setminus\{0\}}}
\newcommand{\pmo}{\ensuremath{ \{\pm1\} }}
\newcommand{\angl}{\ensuremath{[0,2\pi)}}
\newcommand{\halfangl}{\ensuremath{[0,\pi)}}

\newcommand{\R}{\ensuremath{\mathds{R}}}
\newcommand{\Z}{\ensuremath{\mathds{Z}}}

\DeclareMathOperator{\sgn}{sgn}
\DeclareMathOperator{\const}{const}

\DeclareMathOperator{\vol}{vol}

\DeclareMathOperator{\tr}{tr}

\DeclareMathOperator{\Id}{Id}

\newcommand\dts{\makebox[1em][c]{.\hfil.\hfil.}}

\makeatletter
\newcommand{\allowhy}{\nobreak\hskip\z@skip}
\makeatother

\makeatletter
\newcommand{\pushright}[1]{\ifmeasuring@#1\else\omit\hfill$\displaystyle#1$\fi\ignorespaces}
\newcommand{\pushleft}[1]{\ifmeasuring@#1\else\omit$\displaystyle#1$\hfill\fi\ignorespaces}
\makeatother

\begin{document}

\title{Normal forms of two-dimensional metrics admitting exactly one essential projective vector field}
\author{%
	\textsc{Gianni Manno} (\textsf{giovanni.manno@polito.it})\\
	\footnotesize
    Dipartimento di Scienze Matematiche (DISMA),
    Politecnico di Torino,
    Corso Duca degli Abruzzi, 24,
    10129 Torino, Italy\\[0.5cm]
	\textsc{Andreas Vollmer} (\textsf{andreasdvollmer@gmail.com}; \textsf{a.vollmer@unsw.edu.au})\\
	\footnotesize
    Istituto Nazionale di Alta Matematica -- Dipartimento di Scienze Matematiche (DISMA),\\[-0.2cm]
	\footnotesize
    Politecnico di Torino,
    Corso Duca degli Abruzzi, 24,
    10129 Torino, Italy\\[-0.2cm]
    --- \\[-0.2cm]
	\footnotesize
    School of Mathematics and Statistics,
    University of New South Wales,
    Sydney, NSW 2052, Australia
    }
\maketitle

\begin{abstract}
We give a complete list of mutually non-diffeomorphic normal forms for the two-dimensional metrics that admit one essential (i.e., non-homothetic) projective vector field.
This revises some results in~\cite{matveev_2012} and extends the results of \cite{bryant_2008,matveev_2012}, solving a problem posed by Sophus Lie in 1882~\cite{lie_1882}.
\end{abstract}

\begin{center}
\footnotesize
MSC 2010 classes: 53A20, 53A55, 53B10\smallskip

\end{center}

\section{Introduction}

Let $(M,g)$ be a smooth Riemannian or pseudo-Riemannian manifold of dimension $2$. In the following the abbreviation ``metric'' is used for both Riemannian and pseudo-Riemannian metrics, unless otherwise specified.
\begin{definition}
A \emph{projective transformation} is a (local) diffeomorphism of $M$ sending geodesics into geodesics (where we view geodesics as unparametrized curves).
A vector field on $M$ is called \emph{projective} if its (local) flow acts by projective transformations.
\end{definition}

\noindent The set of projective vector fields of a metric $g$ forms a Lie algebra \cite{lie_1883}, denoted by $\projalg(g)$ in the following. Infinitesimal homotheties, i.e.\ vector fields $w$ such that $\lie_wg=\lambda g$ for some $\lambda\in\mathds{R}$, are examples of projective vector fields (for $\lambda=0$, this includes Killing vector fields).

\begin{remark}
Since no misunderstandings can occur, we say that one (resp., more than one) projective vector field is admitted by a metric, meaning that its projective algebra has dimension~$1$ (resp., higher).
In doing so, we follow the terminology used in~\cite{matveev_2012,bryant_2008}.
\end{remark}

\begin{definition}
A projective vector field that is not an infinitesimal homothety is called \emph{essential}.
\end{definition}
\noindent In the present paper we mainly focus our attention on metrics admitting essential projective vector fields as metrics admitting an infinitesimal homothety $w$ were described in~\cite{lie_1882}: locally, around a generic point (where the homothetic vector field does not vanish), there is a system of coordinates $(x,y)$ such that $w=\partial_x$ and the metric is of the form
\begin{equation*}\label{eqn:homothety}
 g=e^{\lambda x}\,\begin{pmatrix} E(y) & F(y)\\ F(y) & G(y) \end{pmatrix}\,,
\end{equation*}
where the matrix on the right hand side is a non-degenerate matrix, i.e., $EG-F^2\ne0$.

\begin{definition}
Two metrics are called \emph{projectively equivalent} if they have the same geodesics (as unparametrized curves). The collection of all metrics projectively equivalent to a given metric $g$ is called the \emph{projective class} of~$g$. We denote it by $\projcl(g)$.
\end{definition}

\noindent It turns out that metrics that belong to the same projective class admit the same algebra of projective vector fields. A natural question is to characterize the (projective classes of) metrics with a prescribed Lie algebra of projective symmetries. A more thorough objective is to find a complete list of mutually non-diffeomorphic normal forms, i.e.\ coordinate transformations may not link two different normal forms.
In these directions, in his 1882 paper \cite{lie_1882}, Sophus Lie formulated the following problem:
\begin{problem}[Lie, 1882]\label{pbl:lie}
 Determine the metrics that describe surfaces whose geodesic curves admit an infinitesimal transformation\footnote{German original \cite{lie_1882}: \emph{``Es wird verlangt, die Form des Bogenelementes einer jeden Fl\"ache zu bestimmen, deren geod\"atische Curven eine infinitesimale Transformation gestatten.''}}, i.e.\ metrics whose projective algebra has $\dim(\projalg(g))\geq 1$.
\end{problem}

\noindent Fubini referred to this problem as the ``Lie problem'', see \cite{aminova_2003,fubini_1903}. An overview of the history of the problem can be found, for instance, in \cite{aminova_2003,aminova_2006}. Important works in the field, with respect to the considerations in the present paper, are for instance \cite{dini_1869,liouville_1889,bryant_2008,matveev_2012,bryant_2009}.

Problem~\ref{pbl:lie} is also referred to as \emph{Lie's First Problem}, in contrast to the narrower problem when the inequality~$\dim(\projalg(g))>1$ holds strictly. The latter problem is then referred to as \emph{Lie's Second Problem}.
A solution to Lie's (Second) Problem was first claimed in \cite{aminova_2003,aminova_2006}, but the proof contained a gap. A correct solution is given in \cite{bryant_2008}, where mutually non-diffeomorphic normal forms of $2$-dimensional metrics $g$ with $\dim(\projalg(g))\geq 2$ were found around generic points (i.e., the orbit of the projective algebra is of constant non-zero dimension in a neighborhood of these points).
Metrics admitting exactly one, essential projective vector field have been treated in~\cite{matveev_2012}, around generic points, providing an explicit list of projective classes.
However, this list is not a list of mutually non-diffeomorphic normal forms, i.e.\ non-isometric metrics; actually, it is not even sharp as a list of projective classes under projective transformations.
Moreover, reference~\cite{matveev_2012} contains some gaps and some supplementary results are incomplete, see Section~\ref{sec:state} where these issues are discussed in detail, along with a brief outline of the state of art. The main outcomes of the present paper are discussed in Section~\ref{sec:main.results}: Theorems~\ref{thm:theorem.2.revised}, \ref{thm:theorem.3.revised} and~\ref{thm:normal.forms} together with Proposition~\ref{prop:homotheties.and.eigenvectors}, constitute a classification in terms of normal forms up to isometries of metrics that admit exactly one, essential projective vector field. Note that Theorems~\ref{thm:theorem.2.revised} and~\ref{thm:theorem.3.revised} constitute corrections of two results found in~\cite{matveev_2012}, namely Theorems~2 and~3 of this reference.
As a by-product of the proof of Theorem~\ref{thm:normal.forms}, we also obtain a classification of all projective classes that cover metrics with exactly one, projective vector field that is essential (as stated above, ~\cite{matveev_2012} provides only a description of such classes, not a classification, see Section~\ref{sec:state} for more details).

\smallskip\noindent
For the formulation of the main results, we need the following proposition.

\begin{proposition}[\cite{beltrami_1865,bolsinov_2009,dini_1869}]
\label{prop:dini}
	Let $\g1$ and $\g2$ be projectively equivalent, non-proportional metrics. Then, in a neighborhood of almost every point, there are coordinates $(x,y)$ such that the metrics assume one of the following three normal forms:
	\smallskip

	\begin{center}
		\begin{TAB}(r,1cm,1cm)[3pt]{|c|c|c|c|}{|c|c|c|}
		&	\h{A}{Liouville case}	&	\h{B}{complex Liouville case}	&	 \h{C}{Jordan block case}  \\
		\g1		& $(X-Y)(dx^2\pm dy^2)$		&	 $(\cc{h(z)}-h(z))\,(d\cc{z}^2-dz^2)$					 &	 $(1+xY^\prime)dxdy$  \\
		\g2	& $(\frac{1}{X}-\frac{1}{Y})(\frac{dx^2}{X}\pm\frac{dy^2}{Y})$
		& $\left(\frac{1}{\cc{h(z)}}-\frac{1}{h(z)}\right)\, \left(\frac{d\cc{z}^2}{\cc{h(z)}}-\frac{dz^2}{h(z)}\right)$
		& \begin{minipage}{3.5cm}\centering $\frac{1+xY'}{Y^4}\ (-2Y\,dxdy$\\ $+(1+xY')\,dy^2)$\end{minipage} \\
		\end{TAB}
	\end{center}
	Here, $X=X(x)$ and $Y=Y(y)$ are functions of one variable only, and $h=h(z)$ is holomorphic, where we use coordinates $z=x+iy$, $\cc{z}=x-iy$.
\end{proposition}

\begin{remark}
Proposition~\ref{prop:dini} provides normal forms for a pair of non-proportional, projectively equivalent metrics. These have been used in \cite{matveev_2012} to arrive at a description of all projective classes of metrics with exactly one, essential projective vector field\footnote{In the current context, we use Proposition~\ref{prop:dini} to describe a pair of projectively equivalent metrics that serve as ``generators'', via Formula \eqref{eqn:general.metric} below, of their projective class. However, in particular cases, three such ``generators'' are needed.},
organized into 10 cases belonging to 3 different types according to Proposition~\ref{prop:dini}, see Theorem~\ref{thm:matveev} below, which is taken from \cite{matveev_2012}. In Theorem~\ref{thm:normal.forms}, we maintain this organization of the projective classes.
\end{remark}


\paragraph{Structure of the paper.}
In the remainder of the introduction, we present the main notions and terminology used throughout the paper. For brevity of exposition, our focus is exclusively on definitions and facts required for an understanding of the main results of this paper, and their context. Other results and definitions are introduced later, when and as far as they are needed for the individual proofs.
Section~\ref{sec:lie} first outlines the problem formulated by Sophus Lie, and then gives a brief critical discussion of the state of art of its solution. The section is concluded by a summary of the main contributions the present paper makes to the solution of Lie's Problem.
Sections~\ref{sec:proof.thm.2.revised}, \ref{sec:proof.thm.3.revised} and~\ref{sec:proof.main.theorem} contain the proofs of the main theorems of this paper, while Section~\ref{sec:preparations} contains some supplementary results.

\subsection{Metrizable projective connections, degree of mobility and Liouville tensors}\label{sec:results.used}
A metric $g$ given in an explicit system of coordinates $(x,y)$ gives rise, via its Levi-Civita connection, to a second order ordinary differential equation (ODE)
\begin{equation}\label{eqn:projective-connection}
	y'' = -\Gamma^2_{11} +(\Gamma^1_{11}-2\Gamma^2_{12})\,y' -(\Gamma^2_{22}-2\Gamma^1_{12})\,y'^2 +\Gamma^1_{22}\,y'^3,
\end{equation}
where $y=y(x)$ and $\Gamma^i_{jk}$ are the Christoffel symbols of $g$. The ODE \eqref{eqn:projective-connection} is called \emph{the projective connection associated to $g$}. The name is justified by the fact that, for a solution $y(x)$ to \eqref{eqn:projective-connection}, the curve $(x,y(x))$ is a geodesic of $g$ up to reparametrization.
Thus, (local) diffeomorphisms $(x,y)\to \big(u(x,y),v(x,y)\big)$ that preserve \eqref{eqn:projective-connection} (finite point symmetries) are projective transformations of $g$, i.e.\ they send geodesics into geodesics (as unparametrized curves).
Infinitesimal point symmetries of \eqref{eqn:projective-connection} are projective vector fields of $g$ and generate a $1$-parametric family of projective transformations.
For a more detailed study of such transformations, from a general point of view, see e.g.\ \cite{kruglikov_2009,tresse_1896}.
Let us consider a general ($2$-dimensional) projective connection,
\begin{equation}\label{eq.proj.conn.gen}
  y'' = f_0 +f_1\,y' +f_2\,y'^2 +f_3\,y'^3\,,\quad f_i=f_i(x,y)\,.
\end{equation}
A natural question is whether or not the projective connection \eqref{eq.proj.conn.gen} is metrizable, i.e.\ if there exists a $2$-dimensional metric $g$ such that \eqref{eq.proj.conn.gen} coincides with~\eqref{eqn:projective-connection}.\footnote{The problem of metrisability is examined from multiple perspectives in~\cite{bryant_2009}.}
The metrisability condition can be expressed in terms of a system of four partial differential equations (PDE) on the components of~$g$ (solved, for metrics with one or several projective vector fields, in \cite{matveev_2012} and \cite{bryant_2008}, respectively). There is one major simplification: The PDE system becomes linear when the unknowns (i.e., components of~$g$) are replaced by components of a weighted tensor section $a=\psi^{-1}(g)$,
\begin{equation}\label{eqn:Liouville.tensor.g}
	a=\psi^{-1}(g)=\frac{g}{|\det(g)|^{\nicefrac23}}, \qquad\qquad g=\psi(a)=\frac{a}{\det(a)^2}\,.
\end{equation}
We call $a=\psi^{-1}(g)$ the \emph{Liouville tensor associated to~$g$}.
It takes values in $S^2M\otimes(\vol\,M)^{-\nicefrac43}$ ($S^2M$ is the symmetric tensor product of the module of $1$-forms on $M$, $\vol\,M$ is the one-dimensional bundle of volume forms on $M$).
More precisely we have the following proposition.
\begin{proposition}[\cite{bryant_2008,liouville_1889}]
The projective connection associated to the Levi-Civita connection of a metric $g$ is~\eqref{eq.proj.conn.gen}
if and only if the entries $a_{ij}$ of the matrix $a=\frac{g}{|\det(g)|^{\nicefrac23}}$
satisfy the linear system of PDEs
\begin{equation}\label{eqn:linear.system}
\left\{
\begin{array}{l}
a_{11x} - \frac23\,f_1\,a_{11}+2f_0\,a_{12} = 0
\\
a_{11y} + 2a_{12x} -\frac43\,f_2\,a_{11} +\frac23\,f_1\,a_{12}+2f_0\,a_{22} = 0
\\
2a_{12y} + a_{22x} -2f_3\,a_{11} -\frac23\,f_2\,a_{12}+\frac43\,f_1\,a_{22} = 0
\\
a_{22y} - 2f_3\,a_{12}+\frac23\,f_2\,a_{22} = 0
\end{array}
\right.
\end{equation}
where the subscripts $x,y$ denote derivatives.
\end{proposition}
\noindent Of course, solutions to the linear system of PDE \eqref{eqn:linear.system} span a linear space.
Solutions $a=(a_{ij})$ of the linear system~\eqref{eqn:linear.system} correspond to metrics via \eqref{eqn:Liouville.tensor.g} (i.e. to solutions of the initial metrizability problem) if $\det(a)\ne0$.
In Lemma~2 of~\cite{matveev_2012} it is proven that, if $a\not\equiv0$, then the set of the points where $a$ is degenerate is nowhere dense (in the topological sense).
\begin{definition}\label{def:degree.mobility}
The space of non-zero solutions to the system \eqref{eqn:linear.system} for a metric $g$ is denoted by  $\solSp(g)$ and its dimension is called the \emph{degree of mobility} of $g$.
We usually abbreviate $\solSp=\solSp(g)$ when there is no risk of confusion.
\end{definition}
\begin{remark}\label{rmk:g-a-correspondence}
Let $a\in\solSp(g)$ be a solution of \eqref{eqn:linear.system}. The mapping~$\psi$ defined in~\eqref{eqn:Liouville.tensor.g} is a bijection that identifies $a\ne0$ with the metric $g=\psi(a)$ whose projective connection is~\eqref{eq.proj.conn.gen}.
In view of this correspondence, also the spaces $\solSp(g)$ and $\projcl(g)$ are identified, and in particular we have $\projcl(g)=\psi(\solSp(g))$.
\end{remark}

\begin{remark}
The (weighted) tensor section $a$, as defined in~\eqref{eqn:Liouville.tensor.g}, is not the only possible choice to make the system linear. In fact, we could also have chosen the new indeterminates
\begin{equation}\label{eqn:sigma.g}
   \sigma^{ij} = |\det(g)|^{\tfrac13}\,g^{ij}
\end{equation}
where $g^{ij}$ are components of the inverse metric of $g$.
We use Liouville tensors instead, consistent with the use in~\cite{matveev_2012,bryant_2008,bryant_2009}.
However, the choice~\eqref{eqn:sigma.g} is taken up again in Section~\ref{sec:strategy.benenti} when we define Benenti tensors, i.e., special conformal Killing tensors that are made use of in the proof of the main theorem.
\end{remark}

\begin{notation}
From now on we adopt, for brevity of exposition, the following convention: Whenever there is a non-integer real exponent, the base expression is to be understood as the absolute value if it is real-valued, unless explicitly stated otherwise. Similarly, $\det(g)^{\nicefrac23}$ is to be understood as identical to $|\det(g)|^{\nicefrac23}$. This convention shall not be applied if the base expression is complex-valued. For instance, for $z=\rho e^{i\theta}\in\mathds{C}$ with $\rho>0$ and $\theta\in\angl$, we have that $z^\xi = (\rho e^{i\theta})^\xi = \rho^\xi e^{i\theta\,\xi}$.
The sign of real number $h$ is denoted by $\sgn(h)$.
The indefinite integral of a function $f(y)$ is written as $\int^y f(\xi)\,d\xi$. Note that in Theorem~\ref{thm:normal.forms} below the constant of integration is absorbed by a translation of the coordinate $x$.
\smallskip

\noindent Later we are working with the pullback and pushforward of a diffeomorphism $\tau$, i.e., with $\tau^*$ and $\tau_*$, respectively. For mixed tensor fields, e.g., $A=\alpha\otimes X$ for a 1-form $\alpha$ and a vector field $X$ on $M$, we denote
\[
  \tau^*A = \tau^*(\alpha\otimes X) = \tau^{-1*}(\alpha)\otimes\tau_*(X)\,.
\]
\end{notation}

\smallskip\noindent
It is proven in \cite{matveev_2012}  that, if we restrict our attention to metrics with one, essential projective vector field, $\dim(\solSp)$ is either $2$ or $3$. Therefore, if we denote by $\{a_i\}$ a basis of  $\solSp$, any solution $a\in\solSp$ is of the form
$a=\sum_{i=1}^m K_i a_i$, where $m=\dim(\solSp)\in\{2,3\}$, $(\K1,\K2)\in\R^2\setminus\{0\}$ and the corresponding metric, via Formula~\eqref{eqn:Liouville.tensor.g}, is given by
\begin{equation}\label{eqn:general.metric}
g=g[\K{1},\dts,\K{m}]
 =\frac{
 	\sum_{i=1}^m \K{i}\frac{\g{i}}{\det(\g{i})^{\nicefrac23}}
       }{
    \left[\det\big(
       \sum_{i=1}^m \K{i}\frac{\g{i}}{\det(\g{i})^{\nicefrac23}}
    \big)\right]^2}\,,\quad m=\dim(\solSp)\in\{2,3\}\,.
\end{equation}
It can easily be seen that $\projcl(\g1)=\dts=\projcl(\g{m})=\projcl(g[\K{1},\dts,\K{m}])$, with $(\K1,\dts,\K{m})\in\mathds{R}^m\setminus\{0\}$, $m=\dim(\solSp)\in\{2,3\}$.
For later use let us define, as well,
\begin{equation}\label{eqn:space.G}
 \projcl(\g1,\dots,\g{m}) = \left\{
      \frac{
 	     \sum_{i=1}^m \K{i}\frac{\g{i}}{\det(\g{i})^{\nicefrac23}}
      }{
 	     \left[\det\big(
 	     \sum_{i=1}^m \K{i}\frac{\g{i}}{\det(\g{i})^{\nicefrac23}}
 	 \big)\right]^2} \text{ such that } (\K1,\dots,\K{m})\ne0
    \right\}\,.
\end{equation}
for $m\in\{2,3\}$, respectively.

If the degree of mobility is one, any projective vector field is homothetic. On the other hand, if the degree of mobility is at least four, there exists a Killing vector field~\cite{koenigs_1896}.
There is another remarkable result on Killing vector fields and projective equivalence:
\begin{lemma}[\cite{bryant_2008}]\label{la:killing.vectors.and.projective.equivalence}
 If a metric admits a Killing vector field, any metric projectively equivalent to it also admits a Killing vector field.
\end{lemma}

\noindent The lemma implies that, for the purposes of the current paper, the existence of a Killing vector field should be considered a property of the projective class rather than of a specific metric.

Before we proceed, let us make another brief comment on notation. We use subscripts $1,2,3$ to denote the basis elements for the space~$\solSp$, e.g.\ $\a1,\a2,\a3$. Components of (weighted) tensors are denoted by double subscript, e.g., $a_{11}=a(e_1,e_1)$ is a component of $a$ w.r.t.\ a basis $\{e_i\}$ on the (co-)tangent space $T^*M$ resp.\ $TM$. Obviously, since in this paper the space $\solSp$ has dimension at most $3$, no misunderstandings can occur. Nontheless, whenever it seems advisable, a clarifying comment shall be made.

\subsection{Action of the projective algebra}\label{sec:lie.action}
The possible local infinitesimal symmetry algebras $\projalg$ of~\eqref{eqn:projective-connection} have been classified by Lie \cite{lie_1882}, proven more recently, e.g., in~\cite{romanovskii_1996}. These algebras are isomorphic to one among the following:
\[
 (i)\ \{0\}\,,\quad
 (ii)\ \mathds{R}\,,\quad
 (iii)\ \mathfrak{s}\,,\quad
 (iv)\ \mathfrak{sl}(2,\mathds{R})\,,\quad
 (v)\ \mathfrak{sl}(3,\mathds{R})\,.
\]
The algebra $\mathfrak{s}$ is the non-commutative, 2-dimensional Lie algebra spanned by elements $e_1,e_2$ that satisfy $[e_1,e_2]=e_1$. In the current paper we are in case $(ii)$ of this classification.

The projective algebra acts on the space $\solSp$ of solutions to~\eqref{eqn:linear.system} via the Lie derivative $\lie_w$ along a projective vector field~$w$.
In fact, $\solSp$ is invariant under this action, since the system~\eqref{eqn:linear.system} is projectively invariant, i.e., depends only on the projective connection, see also Section~2.1 in~\cite{matveev_2012} for more details.
We have that
\begin{equation}\label{eqn:Lwa}
\lie_wa=\det(g)^{-\frac23} \cdot \lie_wg - \frac{2}{3}\det(g)^{-\frac23}
\tr_g(\lie_wg) \cdot g
\end{equation}
Consider the Lie derivative $\lie_w|_\solSp:\solSp\to\solSp$. There exists a basis of $\solSp$ such that $\lie_w|_\solSp$ is represented by a matrix in Jordan normal form. Since $\dim(\solSp)\in\{2,3\}$, there exists a two-dimensional, $\lie_w$-invariant subspace $\hat{\solSp}\subset\solSp$ (for $m=2$, $\hat{\solSp}=\solSp$). In a suitable basis, the (restricted) Lie derivative $\lie_w|_{\hat{\solSp}}$ is represented by one of the following three matrices (note that we can rescale the matrices by a constant factor since the projective vector field is defined up to a constant factor only, cf.\ also~\cite{matveev_2012}):
\begin{equation}\label{eqn:normal-matrices}
 \text{(I)}\quad \begin{pmatrix} \lambda & 0\\ 0 & 1 \end{pmatrix}\quad\text{with $|\lambda|\geq1$}\,;\qquad\qquad
 \text{(II)}\quad \begin{pmatrix} 1 & 1\\ 0 & 1 \end{pmatrix}\,;\qquad\qquad
 \text{(III)}\quad \begin{pmatrix} \lambda & -1\\ 1 & \lambda \end{pmatrix} \quad\text{with $\lambda\in\mathds{R}$}.
\end{equation}
To simplify notation, the restricted Lie derivative $\lie_w|_\solSp$ is denoted simply by $\lie_w$, as we are going to do from now on unless otherwise specified.
We shall see in Section~\ref{sec:optimizing.lambda.xi} that, for our purposes, $\lambda\ne1$ in case~(I) and, w.l.o.g., $\lambda\geq0$ in case~(III).
In case (I), it is convenient to use, instead of~$\lambda$,
\begin{equation}\label{eqn:xi.via.lambda}
  \xi=2\,\frac{\lambda-1}{2\lambda+1}\,.
\end{equation}

\section{Lie's Problem}\label{sec:lie}
This paper is concerned with the solution of Problem~\ref{pbl:lie}, known as Lie's Problem. Our goal is a classification of metrics that admit a 1-dimensional algebra of projective vector fields such that these vector fields are essential, i.e.\ non-homothetic.
In Section~\ref{sec:state}, we outline the main results known on the subject. We do not claim to give an exhaustive discussion of the topic, but rather focus only on those results that are crucial in the context of the question under consideration. In particular, some issues in~\cite{matveev_2012} are identified.
Section~\ref{sec:main.results} is a synopsis of major outcomes of the present paper, which fix the issues raised in Section~\ref{sec:state} and moreover establish normal forms of metrics with exactly one, essential projective vector field, and a classification of their projective classes.

\subsection{State of the Art}\label{sec:state}
Normal forms of 2-dimensional metrics that admit at least 2 linearly independent projective vector fields were obtained in \cite{bryant_2008}, around almost every point.
The full classification can be found in the reference, but since we do not need it explicitly for our purposes here, we just cite the following statement:
\begin{proposition}[\cite{bryant_2008}]
Let $p$ be a point on the 2-dimensional manifold with metric~$g$. If~$g$ admits at least 2 projective vector fields that are linearly independent in~$p$, then, in a neighborhood of~$p$, there exists a Killing vector field of~$g$.
\end{proposition}

\noindent The case of 2-dimensional metrics that admit exactly one, projective vector field has been studied in \cite{matveev_2012}, in a neighborhood where such a projective vector field does not vanish.
The major outcome of \cite{matveev_2012} is a list of cases (projective classes) that cover metrics with exactly one, essential projective vector field.
The list in Theorem~\ref{thm:matveev} is divided along the cases of Proposition~\ref{prop:dini}, and subdivided according to the cases in~\eqref{eqn:normal-matrices}.
\begin{theorem}[\cite{matveev_2012}]\label{thm:matveev}
Let $g$ be a $2$-dimensional metric on $M$ such that $\dim(\projalg(g))=1$. Moreover, let us assume that $g|_U$ admits no homothetic vector field in any neighborhood $U\subset M$. Then, in a neighborhood of almost every point there exists a coordinate system
$(x, y)$ such that $g$ is of the form \eqref{eqn:general.metric}, where $g_i$, $i\in\{1,\dts,m\}$, $m\in\{2,3\}$, are described below.
\begin{enumerate}
	\setlength\itemsep{0em}
	\setlist{nolistsep}
	\item[A)] Liouville-type metrics
			  \begin{equation}\label{eq.metrics.Liouville}
            \g1=(X(x)-Y(y))(X_1(x)\,dx^2+Y_1(y)\,dy^2)\,,\quad \g2=\left(\frac{1}{X}-\frac{1}{Y}\right)
              \left(\frac{X_1}{X}\,dx^2+\frac{Y_1}{Y}\,dy^2\right)
              \end{equation}
	For $h\in\Rnz$, $\xi\in(0,1)\cup(1,4]$, $\varepsilon\in\pmo$, $\lambda\in\mathds{R}$:
	\begin{enumerate}
		\item[(I)] $X(x)=e^{\xi x}$, $Y(y)=he^{\xi y}$, $X_1(x)=e^{2x}$, $Y_1(y)=\varepsilon\,e^{2y}$ (if $\xi=2$: $h\ne-\varepsilon$)
		\item[(II)] $X(x)=\frac{1}{x}$, $Y(y)=\frac{1}{y}$, $X_1(x)=\frac{e^{-3x}}{x}$, $Y_1(y)=h\,\frac{e^{-3y}}{y}$
		\item[(III)] $X(x)=\tan(x)$, $Y(y)=\tan(y)$, \smash{$X_1(x)=\frac{e^{-3\lambda x}}{\cos(x)}$, $Y_1(y)=h\,\frac{e^{-3\lambda y}}{\cos(y)}$} (if $\lambda=0$: $h\ne\pm1$)
	\end{enumerate}
	The projective vector field is
	\begin{equation}\label{eqn:proj.vector.field.A}
	  w = \partial_x + \partial_y\,.
	\end{equation}
	\item[B)] Complex Liouville metrics
\begin{equation}\label{eq.metrics.complex.Liouville}
\g1=(h(z)-\cc{h(z)})\,(h_1(z)\,dz^2-\cc{h_1(z)}\,d\cc{z}^2)\,,
	\quad \g2=\left(\frac{1}{h(z)}-\frac{1}{\cc{h(z)}}\right)\,\left(\frac{h_1(z)}{h(z)}\,dz^2-\frac{\cc{h_1(z)}}{\cc{h(z)}}\,d\cc{z}^2\right)
\end{equation}
For $C\in\mathds{C}, |C|=1$, $\xi\in(0,1)\cup(1,4]$, $\lambda\in\mathds{R}$:
	\begin{enumerate}
		\item[(I)] $h(z)=C\,e^{\xi z}$, $h_1(z)=e^{2z}$ (if $\xi=2$: $C\ne\pm1$)
		\item[(II)] $h(z)=\frac{1}{z}$, $h_1(z)=C\,\frac{e^{-3z}}{z}$
		\item[(III)] $h(z)=\tan(z)$, \smash{$h_1(z)=C\,\frac{e^{-3\lambda z}}{\cos(z)}$} (if $\lambda=0$: $C\ne\pm1$)
	\end{enumerate}
	The projective vector field is
	\begin{equation}\label{eqn:proj.vector.field.B}
	  w = \partial_x = \partial_z + \partial_{\cc{z}}\,.
	\end{equation}
	\item[C)] Jordan-block metrics
	\begin{equation}\label{eq.metrics.JB}
\g1=(Y(y)+x)\,dxdy\,,
	\quad
	 \g2=-\frac{2\,(Y(y)+x)}{y^3}\,dxdy+\frac{(Y(y)+x)^2}{y^4}\,dy^2
\end{equation}
For $\lambda\in\mathds{R}$ and $\xi\in(0,\frac12)\cup(\frac12,1)\cup(1,4]$:
	\begin{enumerate}
		\item[(Ia)] $Y(y)=y^2$ ($y>0$), and there is a third metric,
		\begin{equation}\label{eq.metrics.g3.extra}
        \g3=\frac{y^2+x}{(3x-y^2)^6}\,\left(
				9(y^2+x)\,dx^2 -4y\,(9x+y^2)\,dxdy +12x\,(y^2+x)\,dy^2
			\right)\,.
        \end{equation}
       	The projective vector field is
        \begin{equation}\label{eqn:proj.vector.field.CIa}
          w = 2x\partial_x + y\partial_y\,.
        \end{equation}
		\item[(Ib)] $Y(y)=y^{\nicefrac{1}{\xi}}$ ($y>0$) and
       	the projective vector field is
		\begin{equation}\label{eqn:proj.vector.field.CIb}
		  w = x\partial_x + \xi y\partial_y\,.
		\end{equation}
		\item[(II)] $Y(y)=e^{\nicefrac{3}{2y}}\,\frac{\sqrt{|y|}}{y-3} +\int^y e^{\nicefrac{3}{2s}}\,\frac{\sqrt{|s|}}{(s-3)^2}\,ds$
		and the projective vector field is
		\begin{equation}\label{eqn:proj.vector.field.CII}
		  w = \frac{y-3}{2}\,
		      \left(
		   			x + \int^y e^{\nicefrac{3}{2s}}\,
		   						\frac{\sqrt{|s|}}{(s-3)^2}\,ds
		      \right)\partial_x + y^2\partial_y\,.
		\end{equation}
		\item[(III)] $Y(y)=e^{-\frac32\lambda\arctan(y)}\frac{\sqrt[4]{y^2+1}}{y-3\lambda}
		+\int^y e^{-\frac32\lambda\arctan(s)}\frac{\sqrt[4]{s^2+1}}{(s-3\lambda)^2}\,ds$
		and the projective vector field is
		\begin{equation}\label{eqn:proj.vector.field.CIII}
		w = \frac{y-3\lambda}{2}\,
			\left(
				  x + \int^y e^{-\frac32\lambda\arctan(s)}\,
				  		\frac{\sqrt[4]{s^2+1}}{(s-3\lambda)^2}\,ds
			\right)\partial_x + (y^2+1)\partial_y\,.
		\end{equation}
	\end{enumerate}
\end{enumerate}

\end{theorem}

\noindent This statement and its implications deserve a detailed discussion. Theorem~\ref{thm:matveev} says that, given a metric with exactly one, essential projective vector field, its projective class is defined by one of the metrics in the list.
However, in general, the converse is not true, and in this sense the list of Theorem~\ref{thm:matveev} is not optimal.
More specifically, we can identify the following shortcomings of Theorem~\ref{thm:matveev}:
\begin{enumerate}
 \item Theorem~\ref{thm:matveev} contains the projective classes of metrics that admit one, essential projective vector field. It does not provide a classification because the list in Theorem~\ref{thm:matveev} is not sharp, i.e., some projective classes overlap.
 Also, it does not provide normal forms for metrics within these projective classes.
 \item Metrics that admit one, essential projective vector field~$w$ can be projectively equivalent to metrics for which the projective vector field is homothetic (this caveat is, in fact, already stated in \cite{matveev_2012}).
 In Proposition~\ref{prop:homotheties.and.eigenvectors}, we fix this issue under the assumption that no non-trivial Killing vector field exists.
 \item However, some metrics in Theorem~\ref{thm:matveev} do admit a non-trivial Killing vector field in addition to the essential projective vector field (and thus, by Lemma~\ref{la:killing.vectors.and.projective.equivalence}, all metrics projectively equivalent to them do).
 This contradicts the claim of Theorem~3 of~\cite{matveev_2012}.\footnote{Theorem~3 of~\cite{matveev_2012} claims: None of the metrics from Theorem~\ref{thm:matveev} admits a non-trivial Killing vector field.}
 Theorem~\ref{thm:theorem.3.revised} below closes this gap.
 \item Not all metrics in Theorem~\ref{thm:matveev} with 1-dimensional symmetry algebra have degree of mobility~$2$. Theorem~2 of~\cite{matveev_2012} claims that only the case C(Ia) has degree of mobility~3.\footnote{Theorem~2 of~\cite{matveev_2012} claims: The projective class of every metric~$g$ contained in Theorem~\ref{thm:matveev}, except for those of case C(Ia), coincides with $\projcl(\g1,\g2)$ of~\eqref{eqn:space.G}. The projective class of a metric $g$ from the case C(Ia) of Theorem~\ref{thm:matveev} coincides with $\projcl(\g1,\g2,\g3)$.}
 But other metrics of Theorem~\ref{thm:matveev} admit degree of mobility~3, too, e.g.
 $g = (e^{3x}-e^{3y})(e^{2x}dx^2-e^{2y}dy^2)$
 is an explicit counterexample to Theorem~2 of~\cite{matveev_2012}, as it is inside case A(I) of Theorem~\ref{thm:matveev}, but has degree of mobility~$3$. We close this gap in Theorem~\ref{thm:theorem.2.revised}.
\end{enumerate}

\subsection{Main results of this paper}\label{sec:main.results}
The issues outlined in the above list are going to be fixed by several theorems, which are proven (separately) in the remaining sections of the paper.
First, let us consider the gap concerning the degree of mobility of metrics, see item~4 in the list in the end of Section~\ref{sec:state}, which is interconnected with item~3 of that list.
It appears most transparent to proceed in two steps, stated in Theorems~\ref{thm:theorem.2.revised} and~\ref{thm:theorem.3.revised} (which correspond to Theorems~2 and~3 of~\cite{matveev_2012}).
We first identify cases with degree of mobility~3; cases with a non-trivial Killing vector field are accounted for by Theorem~\ref{thm:theorem.3.revised}.
\begin{theorem}\label{thm:theorem.2.revised}
Let $g$ be a metric of Theorem~\ref{thm:matveev}.
Then $g$ is in exactly one of the following cases:
\begin{enumerate}
	\item It has degree of mobility~2 and a 1-dimensional projective (not Killing) algebra.
	\item It admits a non-trivial Killing vector field. In this case its projective algebra has dimension at least 2.
	(These metrics are fully considered in Theorem~\ref{thm:theorem.3.revised}).
	\item The metric~$g$ is projectively equivalent to at least one metric in the following list:
	\begin{itemize}
		\item[] \emph{A(I)} with the parameters satisfying $(\xi,h,\varepsilon)\in\{(3,\pm1,-1), (2,\pm4,\mp1), (2,\pm\nicefrac14,\mp1) \}$
		\item[] \emph{B(I)} with $(\xi,C)=(3,\pm1)$ or $(\xi,C)=(3,\pm i)$
		\item[] \emph{C(Ia)}
	\end{itemize}
	In these cases, the metric~$g$ has degree of mobility~3 and a 1-dimensional projective (not Killing) algebra.
\end{enumerate}
\end{theorem}

\noindent Let us now consider the cases with a Killing vector field. As indicated in item~3 of the list after Theorem~\ref{thm:matveev}, metrics in Theorem~\ref{thm:matveev} might admit a Killing vector field, contrary to the claim in~\cite{matveev_2012}. This gap is closed by the following theorem (recall also Lemma~\ref{la:killing.vectors.and.projective.equivalence}).
\begin{theorem}\label{thm:theorem.3.revised}
The metrics in the list of Theorem~\ref{thm:matveev} do not admit a non-trivial Killing vector field, except for the following ones
\begin{enumerate}[label=(\roman*)]
 \item Metrics projectively equivalent to metrics A(I) with $\xi=4$ and $h=1$, i.e.\ to
 \begin{equation}\label{eqn:metric.nu.4.h.1}
   g = (e^{4x}-e^{4y})(e^{2x}dx^2+\varepsilon\,e^{2y}dy^2)\,,\quad
   \varepsilon\in\{-1,1\}\,.
 \end{equation}
 The metric $g$ admits the Killing vector field
 \[
   v = \frac{e^{y-x}}{e^{2x}+\varepsilon e^{2y}}\,\partial_x
       +\varepsilon\,\frac{e^{x-y}}{e^{2x}+\varepsilon e^{2y}}\,\partial_y
 \]
 in addition to the projective vector field $\partial_x+\partial_y$.
 \item Similarly to the previous case, metrics projectively equivalent to metrics B(I) with $\xi=4$ and $C=1$ admit a Killing vector field.
 \item Metrics projectively equivalent to metrics B(III) with $\lambda=0$ and $C=\pm i$, i.e.\ to
 \begin{equation}\label{eqn:metric.lambda.0.C.im}
  g = i\,(\tan(z)-\tan(\cc{z}))\left( \frac{dz^2}{\cos(z)} + \frac{d\cc{z}^2}{\cos(\cc{z})} \right)\,.
 \end{equation}
 The metric~$g$ admits the Killing vector field
 \[
  v =
  \frac{\sin(z)\big(\sin(z)\sin(\cc{z})+\cos(z)\cos(\cc{z})-\cos(\cc{z})+\cos(z)-1\big)\
	\bigg(\cos(z)\partial_z+\cos(\cc{z})\partial_{\cc{z}}\bigg)}%
      {(\sqrt{\cos(z)+1}\sqrt{\cos(\cc{z})+1}(\cos(z)\sin(z)\cos(\cc{z})-\cos(z)^2\sin(\cc{z})-\sin(z)\cos(\cc{z})+\cos(z)\sin(\cc{z}))}
 \]
 in addition to the projective vector field $\partial_x=\partial_z+\partial_{\cc{z}}$.
\end{enumerate}
\end{theorem}

\noindent A further simplification is obtained in the case that the degree of mobility is exactly~3.
\begin{proposition}\label{prop:dom3.proj.equiv}
 Metrics in Theorem~\ref{thm:matveev} with exactly one, essential projective vector field and degree of mobility~$3$ are projectively equivalent to $g=(x+y^2)\,dxdy$ on some part of its domain, i.e.\ to a metric of type C(Ia).
\end{proposition}
\noindent The proof can be found in Section~\ref{sec:dom3.proj.equiv}.

In Theorem~\ref{thm:theorem.3.revised}, we have identified metrics in Theorem~\ref{thm:matveev} that admit a Killing vector field. Their projective algebra has dimension larger than one.
Under the assumption that the projective vector field~$w$ is unique (up to rescaling), this vector field is therefore either essential or truly homothetic (i.e., not Killing).
A distinction of these two cases is possible by the following theorem (this addresses item~2 of the list after Theorem~\ref{thm:matveev}):
\begin{proposition}\label{prop:homotheties.and.eigenvectors}
Let $g$ be a metric projectively equivalent to a metric of Theorem~\ref{thm:matveev},
but%
\footnote{The metrics in Theorem~\ref{thm:theorem.3.revised} have a projective algebra of dimension $\dim(\projalg(g))>1$, but for the statement $\dim(\projalg(g))=1$ is needed. On the other hand, any metric with $\dim(\projalg(g))>1$ admits a Killing vector field, see~\cite{bryant_2008}. Theorem~\ref{thm:theorem.3.revised} makes sure these are the only exceptions (up to projective equivalence, see Lemma~\ref{la:killing.vectors.and.projective.equivalence}).}
not in the projective class of any metric in Theorem~\ref{thm:theorem.3.revised}.
Then the projective vector field~$w$ is homothetic if and only if the Liouville tensor $a=\psi^{-1}(g)$ is an eigenvector of $\lie_w|_\solSp$.
\end{proposition}

\noindent The proof of Proposition~\ref{prop:homotheties.and.eigenvectors} can be found in Section~\ref{sec:proof.homotheties.and.eigenvectors}.

Finally, and most importantly, the following theorem provides a list of normal forms, up to coordinate transformations, of metrics that admit exactly one essential projective vector field, indicating the type of the metric according to Proposition~\ref{prop:dini}.
Theorem~\ref{thm:normal.forms} below resolves the first issue in item~1 after Theorem~\ref{thm:matveev}.

\begin{theorem}\label{thm:normal.forms}
Let $g$ be a (pseudo-)Riemannian metric on a $2$-dimensional manifold $M$. Let us assume that $\dim(\projalg(g))=1$ in any neighborhood of $M$ and that $\projalg(g)$ is generated by an essential projective vector field. Then, in a neighborhood of almost every point there exists a local coordinate system $(x,y)$ such that $g$ assumes one of the following mutually non-diffeomorphic normal forms (organized according to their type, see Proposition \ref{prop:dini}).

\noindent Note: The list below contains some ``special cases'' for the choice of the parameter values. The origin of these exceptions is going to become clear during the proof, see Section~\ref{sec:proof.main.theorem}.

\paragraph{(A) Liouville case}

\begin{enumerate}
 \setlength\itemsep{0em}
 \setlist{nolistsep}
	\item\label{item:A.2r}
		The normal forms are
		\[
		  g=\kappa\left( \frac{(e^{\xi x}-he^{\xi y})\,e^{2x}}{(1+\varrho he^{\xi y})(1+\varrho e^{\xi x})^2}\,dx^2 +\varepsilon\frac{(e^{\xi x}-he^{\xi y})\,e^{2y}}{(1+\varrho he^{\xi y})^2(1+\varrho e^{\xi x})}\,dy^2 \right)
		\]
		where $\xi=2\,\frac{\lambda-1}{2\lambda+1}\in(0,1)\cup(1,4]$, $\varepsilon\in\pmo$ and $|h|\geq1$ (if $\xi=2$: $h\ne-\varepsilon$ and $h\ne-4\varepsilon$; if $\xi=3$: $|h|\ne1$ when $\varepsilon=-1$; if $\xi=4$: $h\ne1$) as well as $\kappa\in\Rnz$, $\varrho\in\pmo$.\smallskip

		\emph{Special cases:} If $h=-1$, we require $\varrho=1$. If $h=+1$ and $\varepsilon=1$, the parameter $\kappa$ may only assume positive values, $\kappa>0$.
	\item\label{item:A.1r}
		The normal forms are
		\[
		  g=\kappa\left(\frac{(y-x)e^{-3x}}{x^2y}dx^2 + \frac{h (y-x)e^{-3y}}{xy^2}dy^2\right)
		\]
		where $|h|\geq1$ and $\kappa\in\Rnz$.\smallskip

		\emph{Special case:} In case $h=1$, $\kappa>0$ is required.
	\item\label{item:A.2c}
		The normal forms are
		\begin{align*}
		 g&=\frac{\sin(y-x)}{\sin(y+\theta)\,\sin(x+\theta)}\,
		\left(
		\frac{e^{-3\lambda\,x}}{\sin(x+\theta)}\,dx^2
		+\frac{h\,e^{-3\lambda\,y}}{\sin(y+\theta)}\,dy^2
		\right)\,, && \text{for $\lambda>0$}\,, \\
		g&=\kappa\,\frac{\sin(y-x)}{\sin(y)\,\sin(x)}\,
		\left(
		\frac{dx^2}{\sin(x)}
		+\frac{h\,dy^2}{\sin(y)}
		\right)\,,
		\quad (x,y \mod 2\pi)
		&& \text{case ``$\lambda=0$''}\,,
		\end{align*}
		where $|h|\leq e^{-3\lambda\pi}$ (if $\lambda=0$: $h\ne\pm1$) as well as $\theta\in\angl$ and $\kappa>0$\smallskip

		\emph{Special case:} If $\lambda>0$ and $|h|=e^{-3\lambda\pi}$, we require $\theta\in\halfangl$.
\end{enumerate}
The projective vector field in (A.1)--(A.3) is \eqref{eqn:proj.vector.field.A}.

\paragraph{(B) Complex--Liouville case}

	\begin{enumerate}
		\setlength\itemsep{0em}
		\setlist{nolistsep}
		\setcounter{enumi}{3}
		\item\label{item:B.2r}
			The normal forms are
			\[
			  g=\kappa\,\left(
				 \frac{Cz^\xi-\cc{C}\cc{z}^\xi}{(1+\cc{C}\cc{z}^\xi)(1+Cz^\xi)^2}\,dz^2
				 -\frac{Cz^\xi-\cc{C}\cc{z}^\xi}{(1+\cc{C}\cc{z}^\xi)^2(1+Cz^\xi)}\,d\cc{z}^2
				\right)
			\]
			where $\xi=2\,\frac{\lambda-1}{2\lambda+1}\in(0,1)\cup(1,4]$ and $C=e^{i\varphi}$ with $\varphi\in\halfangl$ (if $\xi=2$: $C\ne\pm1$; if $\xi=3$: $C^2\ne\pm1$; if $\xi=4$: $C\ne1$) as well as $\kappa\in\Rnz$.
		\item\label{item:B.1r}
			The normal forms are
			\[
			  g=\kappa\,(\cc{z}-z)\,
				\left(
				\frac{C\,e^{-3z}}{z^2\cc{z}}\,dz^2
				 -\frac{\cc{C}\,e^{-3\cc{z}}}{z\cc{z}^2}\,d\cc{z}^2
				\right)
			\]
			where $C=e^{i\varphi}$, $\varphi\in\halfangl$ and $\kappa\in\Rnz$.
		\item\label{item:B.2c}
			The normal forms are
			\begin{align*}
			g&=\frac{\sin(\cc{z}-z)}{\sin(\cc{z}+\theta)\,\sin(z+\theta)}\,
			\left(
			\frac{C\,e^{-3\lambda z}}{\sin(z+\theta)}\,dz^2
			-\frac{\cc{C}\,e^{-3\lambda\cc{z}}}{\sin(\cc{z}+\theta)}\,d\cc{z}^2
			\right)\,, && \text{for $\lambda>0$}\,, \\
			g&=\kappa\,\frac{\sin(\cc{z}-z)}{\sin(\cc{z})\,\sin(z)}\,
			\left(
			\frac{C\,dz^2}{\sin(z)}
			-\frac{\cc{C}\,d\cc{z}^2}{\sin(\cc{z})}
			\right)\,,
			\quad (z \mod 2\pi)
			&& \text{case ``$\lambda=0$''}\,,
			\end{align*}
			where $C=e^{i\varphi}$ (if $\lambda>0$: $\varphi\in\halfangl$; if $\lambda=0$, $\varphi\in(0,\pi)$) as well as	$\theta\in\angl$ and $\kappa>0$.
	\end{enumerate}
	The projective vector field in (B.4)--(B.6) is \eqref{eqn:proj.vector.field.B}.

\paragraph{(C) Jordan block case}
	\begin{enumerate}
		\setlength\itemsep{0em}
		\setlist{nolistsep}
		\setcounter{enumi}{6}
		\item\label{item:C.2r.m2}
			The normal forms are ($y>0$)
			\[\displaystyle
			  g=\kappa\left(-2\frac{y^{\frac{1}{\xi}}+x}{(y-\varrho)^3}dxdy
			+ \frac{\left(y^{\frac{1}{\xi}}+x\right)^2}{(y-\varrho)^4}dy^2\right)
			\]
			where $\xi=2\,\frac{\lambda-1}{2\lambda+1}\in(0,\frac12)\cup(\frac12,1)\cup(1,4]$ as well as $\varrho\in\pmo$ and $\kappa\in\Rnz$.
		\item\label{item:C.1r}
			The normal forms are
			\[
		 		g=\kappa\,(Y(y)+x)\,dxdy
         		\qquad
		 		\text{where}
		 		\quad
		 		Y(y)=\int^y \frac{e^{\nicefrac{3}{2s}}}{|s|^{\nicefrac32}}\,ds
		 		\quad\text{and}\quad
		 		\kappa\in\Rnz\,.
			\]
		\item\label{item:C.2c}
			The normal forms are
			\[
		  		g	= \kappa\,(Y(\lambda,y)+x)\,dxdy
          		\qquad
          		\text{where}
          		\quad
          		Y(\lambda,y)
		   		= \int^y \frac{e^{-\frac32\lambda\arctan(s)}}{(s^2+1)^{\nicefrac34}}\,ds
			\]
			and $\lambda\geq0$. For $\lambda=0$, we require $\kappa>0$. For $\lambda>0$, $\kappa=e^{\lambda\theta}$ with $\theta\in\angl$.
		\item\label{item:C.2r.m3}
		\emph{Superintegrable case}:
		These metrics are parameterized by points on a sphere. Their normal forms $g[\theta,\varphi]$ are obtained via~\eqref{eqn:Liouville.tensor.g} from a solution~$\sigma[\theta,\varphi]$ of~\eqref{eqn:linear.system}, given by
		\begin{equation}\label{eqn:supint.sphere}
		\sigma[\theta,\varphi]
		=\sin\theta\,\cos\varphi\,\a1
		+\sin\theta\,\sin\varphi\,\a2
		+\cos\theta\,\a3\,,\quad
		\theta\in(0,\pi)\,,\,\,\varphi\in[0,2\pi)\,,
		\end{equation}
		where\footnote{The missing points of the sphere (intersections with the axes) are the metrics with a homothetic vector field.}
		$\varphi\notin\{0,\frac{\pi}{2},\pi,\frac{3\pi}{2}\}$ if $\theta=\frac{\pi}{2}$.
		The $\a{i}$ are given via~\eqref{eqn:Liouville.tensor.g} from the projectively equivalent metrics ($y>0$)
		\begin{subequations}\label{eqn:supint.generators}
		\begin{align}
		  \g1 &= (y^2+x)\,dxdy \\
		  \g2 &= -2\,\tfrac{y^2+x}{y^3}\,dxdy+\tfrac{(y^2+x)^2}{y^4}\,dy^2 \\
		  \g3 &= \tfrac{y^2+x}{(3x-y^2)^6}\,
			(9\,(y^2+x)\,dx^2-4y\,(9x+y^2)\,dxdy+12x\,(y^2+x)\,dy^2)
		\end{align}
		\end{subequations}
	\end{enumerate}
	The projective vector field in the case (C.7) is \eqref{eqn:proj.vector.field.CIb}, and in the case (C.10) it is \eqref{eqn:proj.vector.field.CIa}.
	In the case (C.8), instead of \eqref{eqn:proj.vector.field.CII}, we have
	\begin{equation}\label{eqn:proj.vector.field.C.8}
		w = \frac{y-3}{2}\left( x+ 2\int^y \frac{e^{\nicefrac{3}{2y}}\sqrt{|y|}}{(y-3)^2}\right) \partial_x +y^2\partial_y
	\end{equation}
	and in the case (C.9), instead of \eqref{eqn:proj.vector.field.CIII}, we have
	\begin{equation}\label{eqn:proj.vector.field.C.9}
	   w = \frac{y-3\lambda}{2}\left( x+2\int^y \frac{ e^{-\frac32\lambda\arctan(y)}\sqrt[4]{y^2+1} }{(y-3\lambda)^2} \right)\partial_x + (y^2+1)\,\partial_y
	\end{equation}
\end{theorem}
\noindent Alternatively to~\eqref{eqn:supint.sphere}, we derive explicit coordinate normal forms -- not identical to those of~\eqref{eqn:supint.sphere}~-- of type-10 metrics in Theorem~\ref{thm:normal.forms}. These normal forms are summarized in Proposition~\ref{prop:normal.forms.CIa} on page~\pageref{prop:normal.forms.CIa}.
Furthermore, as a by-product of Theorem~\ref{thm:normal.forms}, the following corollary is obtained. It provides a classification of the projective classes that contain metrics with one, essential projective vector field, in the sense of providing mutually projectively inequivalent normal forms. This fixes the second issue of item~1 in the list after Theorem~\ref{thm:matveev}.

For a more detailed discussion of the superintegrable case (C.10) and the normal forms (C.8) and (C.9) see also~\cite{manno_2019}, which is based on the present paper.
\begin{corollary}\label{cor:projective.classes}
The parameters $C,h,\varepsilon$ and $\lambda$ (or $\xi=\xi(\lambda)$ via~\eqref{eqn:xi.via.lambda}, respectively) describe all projective classes that contain metrics with exactly one, essential projective vector field.
In particular, two different parameter configurations imply that the projective classes are different.
\end{corollary}
\begin{proof}
The statement follows from the proof of Theorem~\ref{thm:normal.forms}, see Section~\ref{sec:proof.projective.classes}.
\end{proof}

\begin{remark}
 The cases (C.8) and (C.9) in Theorem~\ref{thm:normal.forms} involve indefinite integrals, which are determined up to an integration constant only. Note that this inherent freedom is absorbed by simple translations in $x$.
 In these cases, the non-diffeomorphic property is satisfied after fixing the integration constant.
\end{remark}

\begin{remark}
 Note that the normal forms in Theorem~\ref{thm:normal.forms} are local normal forms, where locality is to be understood as in Proposition~\ref{prop:dini}.
 More precisely, for almost every point $p$, one can find a neighborhood $\mathcal{U}\ni p$, and a coordinate system $(x,y)$  on $\mathcal{U}$, in which the metric is described by one of the forms of the Theorem~\ref{thm:normal.forms}. We underline that the coordinate system is generally not centered in $p$.
 The normal forms themselves are not canonical, as some restrictions may be modified. For instance, in case (C.7) we may relax the requirement $y>0$, allowing negative~$y$, but at the expense of imposing $\rho=1$, which can be achieved after a transformation $y\to-y$.
\end{remark}

\section{Proof of Theorem \ref{thm:theorem.2.revised}}\label{sec:proof.thm.2.revised}
Theorem~2 of \cite{matveev_2012} claims that in almost all cases -- all except C(Ia) of Theorem~\ref{thm:matveev} -- the space of solutions is 2-dimensional. Unfortunately, the proof contains a gap which gives rise to counterexamples.
Take a metric~$g$ from Theorem~\ref{thm:matveev}. As it turns out, the following situations can happen:
(1) First, the metric~$g$ could admit a Killing vector field in addition. This implies that the projective algebra is higher-dimensional \cite{bryant_2008}.
(2) If no additional Killing vector field exists, the projective algebra is 1-dimensional. In this situation, the degree of mobility can be~2 or~3, see~\cite{matveev_2012}.

The current section is devoted to identifying the metrics in Theorem~\ref{thm:matveev} that have a 1-dimensional projective algebra but whose degree of mobility is equal to~3.

\subsection{Method}\label{sec:method.thm2}
We aim to determine the metrics in Theorem~\ref{thm:matveev} which have degree of mobility~3 and do not admit a non-trivial Killing vector field.
Such a metric~$g$ corresponds to a solution of the system~\eqref{eqn:linear.system}, and also satisfies a system $\lie_w(\a{i}) = \sum_j A_{ij}\a{j}$, where $i,j$ range from~1 to~3 (with $A_{ij}\in\mathds{R}$).
Note that $\a1,\a2,\a3$ are Liouville tensors, and a certain linear combination of them corresponds to the metric~$g$, via Equation~\eqref{eqn:general.metric}.
Recall that there is always a two-dimensional subspace of $\solSp$ that is invariant under $\lie_w$, see also~\cite{matveev_2012}.
Let $(\a2,\a3)$ be a basis of this subspace. As it is invariant under $\lie_w$, we obtain
\[
  \qquad
  \lie_w\begin{pmatrix} \a2 \\ \a3 \end{pmatrix}
  =A\begin{pmatrix} \a2 \\ \a3 \end{pmatrix}\,,\quad
  \text{where $A$ is a matrix as in Equation~\eqref{eqn:normal-matrices}}\,.
\]
Explicit solutions for $\a2,\a3$ are obtained in~\cite{matveev_2012}.
We rely on the same approach as~\cite{matveev_2012}, which is based on an argument by contradiction. Thus, let us assume that the degree of mobility is~3. We can find a basis $(\a1,\a2,\a3)$ of the space $\solSp$ of solutions to~\eqref{eqn:linear.system} such that $\lie_w$ takes one of the following forms:

\begin{equation}\label{eqn:dom.3.Lw}
 \begin{pmatrix} 1 & 1 & \\ & 1 & 1 \\ & & 1 \end{pmatrix}\,,\quad
 \begin{pmatrix} \mu & & \\ & 1 & 1 \\ & & 1 \end{pmatrix}\,,\quad
 \begin{pmatrix} \mu & & \\ & \lambda' & \\ & & 1 \end{pmatrix}\,,\quad
 \begin{pmatrix} \mu & & \\ & \lambda & -1 \\ & 1 & \lambda \end{pmatrix}\,,\quad
\end{equation}
where $|\lambda'|\geq1$, $\lambda\in\mathds{R}$ and $\mu\ne0$. Each of these cases admits a 2-dimensional subspace invariant under $\lie_w$, see~\eqref{eqn:normal-matrices}, where we can use Proposition~\ref{prop:dini}.
Now, assume the degree of mobility is indeed~3, and~$\a1$ is transversal to the invariant subspace spanned by $(\a2,\a3)$.
Then, following~\cite{matveev_2012}, the components of $\a1$ satisfy a certain linear system.
Generically there are no solutions to this system, and~\cite{matveev_2012} claims that a solution only exists in case C(Ia) in Theorem~\ref{thm:matveev}.
However, contrary to this claim, we actually find additional solutions to this linear system in other cases.
We will also use the following lemma which is implicitly proven in~\cite{matveev_2012}.
\begin{lemma}\label{la:homothetic.subspace}
 Assume $\solSp$ (the space of solutions to~\eqref{eqn:linear.system}) admits a 2-dimensional subspace, invariant under~$\lie_w$, such that on this subspace the projective vector field is homothetic.
 Then any metric in the projective class represented by $\solSp$ admits a non-trivial Killing vector field.
\end{lemma}
\noindent Before proving the lemma, recall Lemma~\ref{la:killing.vectors.and.projective.equivalence}: If a metric admits a non-trivial Killing vector field, then all metrics in its projective class admit a non-trivial Killing vector field, too.
\begin{proof}
 Let us assume the projective vector field is truly homothetic on the subspace, i.e.\ not Killing (otherwise the statement is trivial).
 We recite the proof from~\cite{matveev_2012}. However, in order to ensure self-containedness as well as brevity, we only briefly outline the reasoning.
 Let the invariant subspace be spanned by a pair $\a1,\a2$.
 Next, observe that the assumption implies, maybe after rescaling $w$,
 \begin{equation}\label{eqn:Lw.for.homothetic.class}
  \lie_w\begin{pmatrix} \a2 \\ \a3 \end{pmatrix}
  =\begin{pmatrix} 1 & \\ & 1 \end{pmatrix}
  \begin{pmatrix} \a2 \\ \a3 \end{pmatrix}\,.
 \end{equation}
 The pair $(\a2,\a3)$ can be assumed to be in Liouville, complex Liouville or Jordan form, according to Proposition~\ref{prop:dini}.
 Let us first consider the Liouville case and $w=(w^1,w^2)$. Then~\eqref{eqn:Lw.for.homothetic.class} is equivalent to
 \[
  \frac{\partial w^1}{\partial x} = -\frac32 = \frac{\partial w^2}{\partial y}\,,\quad
  \frac{\partial w^1}{\partial y} = \frac{\partial w^2}{\partial x} = 0\,,\quad
  w^1X' = 0 = w^2Y'\,,
 \]
 as derived in~\cite{matveev_2012}.
 This means, since $w\ne0$, that the system is incompatible unless the functions $X,Y$ are constant, which implies the existence of a non-trivial Killing vector field.
 For the complex Liouville case, the proof is similar.
 Finally, in the Jordan case, Equation~\eqref{eqn:Lw.for.homothetic.class} is equivalent to
 \begin{equation*}
  (w^1,w^2) = (-\tfrac32x+c,0)\,,\quad
  3Y(y)+2c=0\,,
 \end{equation*}
 where $c$ is an integration constant (which, w.l.o.g., can be chosen $c=0$, see~\cite{matveev_2012}). Thus, we obtain the metric $(Y(y)+x)\,dxdy = x\,dxdy$, which admits the Killing vector field $\partial_y$ (it is even constant curvature).
\end{proof}

\noindent The proof of Theorem~\ref{thm:theorem.2.revised} is performed in Sections~\ref{sec:Thm2.A.I}---\ref{sec:Thm2.C}.
It is carried out along the following general directions:
\begin{enumerate}
	\item Perform a change of coordinates such that in the new coordinates $(x,y)$ the projective vector field becomes $w=\partial_x$. Thus, the metrizability conditions~\eqref{eqn:linear.system} will depend only on~$y$.
	\item Assume the degree of mobility is 3. Let $(\a2,\a3)$ be a basis of the $\lie_w$-invariant subspace of~$\solSp$, such that $\lie_w$ takes the form of one of the cases in~\eqref{eqn:normal-matrices}. Then two principal cases can occur (cf.~\eqref{eqn:dom.3.Lw}):
	\begin{enumerate}
		\item The Liouville tensor~$\a1$ is an eigendirection associated to the eigenvalue $\mu$ of $\lie_w$.
		The metrizability conditions~\eqref{eqn:linear.system} thus become
		\begin{subequations}\label{eqn:linear.system.dom3.mu}
			\begin{align}
			\mu\,a_{11}-\frac23\,f_1\,a_{11}+2f_0\,a_{12} &= 0
			\label{eqn:linear.system.dom3.mu.first} \\
			a_{11,y}+2\mu\,a_{12}-\frac43\,f_2\,a_{11}+\frac23\,f_1\,a_{12}+2f_0\,a_{22} &= 0
			\label{eqn:linear.system.dom3.mu.second} \\
			2a_{12,y}+\mu\,a_{22}-2f_3\,a_{11}-\frac23\,f_2\,a_{12}+\frac43\,f_1\,a_{22} &= 0
			\label{eqn:linear.system.dom3.mu.third} \\
			a_{22,y}-2\,f_3\,a_{12}+\frac23\,f_2\,a_{22} &= 0\,.
			\label{eqn:linear.system.dom3.mu.fourth}
			\end{align}
		\end{subequations}
		This is a homogeneous linear system of four PDE for the components of~$a=\a1$.
		\item The second possibility, which actually only occurs once in~\eqref{eqn:dom.3.Lw}, is that $\lie_w$ on the 3-dimensional space $\solSp$ of solutions to~\eqref{eqn:linear.system} admits only one eigenvalue, but with algebraic multiplicity~3. It this case, a system of PDE similar to~\eqref{eqn:linear.system.dom3.mu} is obtained, which however is inhomogeneous.
	\end{enumerate}
	\item The Equations~\eqref{eqn:linear.system.dom3.mu.second}---\eqref{eqn:linear.system.dom3.mu.fourth} can be solved for derivatives of components of $a=\a1$. The remaining Equation~\eqref{eqn:linear.system.dom3.mu.first} is an algebraic condition on~$\a1$. Differentiating~\eqref{eqn:linear.system.dom3.mu.first} twice and replacing derivatives, we obtain two more algebraic conditions on~$\a1$. Actually, the equations are linear algebraic and, for the case~(a) of point~2, homogeneous. In the case~(b) of point~2,~\eqref{eqn:linear.system.dom3.mu} is replaced by an inhomogeneous system, and the linear algebraic system becomes inhomogeneous, too.
	Writing the linear algebraic system in the form of a matrix equation,
	\begin{equation}\label{eqn:Mv=b}
	  \underbrace{
	  	\begin{pmatrix}
	  	  \mu-\tfrac23 f_1 & 2f_0 & 0 \\
	  	  * & * & * \\
	  	  * & * & *
	  	\end{pmatrix}
	  }_{M}\,
  	  \underbrace{
  	  	\begin{pmatrix} a_{11} \\ a_{12} \\ a_{22} \end{pmatrix}
  	  }_{v}
      = b\,.
	\end{equation}
	The entries of $M$ marked by~$*$ depend only on $\mu$ and the functions $f_0,\dots,f_3$.
	In case (a), $b=0$. In case (b), $b\ne0$.

	If $\det(M)\ne0$, system~\eqref{eqn:Mv=b} admits only the trivial solution, which would contradict the hypothesis. Therefore, $\det(M)=0$ must hold, in order that the degree of mobility can be~3 (or higher).

	The requirement $\det(M)=0$ can be solved and yields possible parameter configurations that might admit degree of mobility~3. In order to get the precise answer, we need to substitute these parameter configurations back into~\eqref{eqn:Mv=b}. Solving this restricted problem yields relations on the functions $a_{11}$, $a_{12}$ and $a_{22}$ that permit to express two of these functions in terms of the other function. Reinserting into \eqref{eqn:linear.system.dom3.mu.second}--\eqref{eqn:linear.system.dom3.mu.fourth}, a reduced ODE problem is obtained, whose compatibility is straightforwardly confirmed or disconfirmed.
\end{enumerate}

\noindent The proof is carried out along the structure suggested by Theorem~\ref{thm:matveev}. In order to maintain brevity, only the computations for metrics of types A(I---III) are to be shown in details. For the other cases, only a summary is given as, in fact, the computations are conceptually analogous and would therefore be of little surplus value for the reader.

\subsection{Metrics A(I) of Theorem~\ref{thm:matveev}}\label{sec:Thm2.A.I}
The metrics under consideration in this section are those of type A(I) in Theorem~\ref{thm:matveev}.
We perform the change of coordinates $(x_\text{new},y_\text{new})=\left(x_\text{old}+y_\text{old},y_\text{old}-x_\text{old}\right)$, such that the projective vector field~$w$ assumes the form~$\partial_{x_\text{new}}$.
The projective connection, in the new coordinates, reads
\begin{align*}
y''
=\frac{\xi\,\varepsilon\,f_\xi^+(y) + f_{\xi-4}^+(y)}%
      {4\varepsilon\,f_\xi^-(y)}
&-\frac{3\xi\,f_{\xi-4}^-(y) - (\xi-4)\varepsilon\,f_\xi^-(y)}%
      {4\varepsilon\,f_\xi^-(y)}\,y'
\\
&+\frac{3\,f_{\xi-4}^+(y) - \varepsilon\,f_\xi^+(y)}%
      {4\varepsilon\,f_\xi^-(y)}\,y'^2
-\frac{\xi\,f_{\xi-4}^-(y) + (\xi-4)\varepsilon\,f_\xi^-(y)}%
      {4\varepsilon\,f_\xi^-(y)}\,y'^3
\end{align*}
where we define the shorthand notation $f^\pm_r(y)=e^{ry}\pm h^2 e^{-ry}$.
The metrizability equations are~\eqref{eqn:linear.system.dom3.mu}, as the Lie derivative on the entire 3-dimensional space $\solSp$ of solutions to~\eqref{eqn:linear.system} can only take the form
\smash{\tiny $\left(\begin{smallmatrix} \mu & & \\ & \lambda' & \\ & & 1 \end{smallmatrix}\right)$}.
Note that $\xi=\xi(\lambda')$ according to~\eqref{eqn:xi.via.lambda}.
As outlined in Section~\ref{sec:method.thm2}, a matrix equation $Mv=0$ can be derived from system~\eqref{eqn:linear.system.dom3.mu}, assuming $\lie_w\a1=\mu\a1$ for the eigenvector of $\lie_w$ transversal to the 2-dimensional invariant subspace.
Explicitly, we find that the determinant of the matrix $M$ evaluates to
\begin{equation}\label{eqn:detM.1c}
\det(M) =
-256\,e^{-8y} (3\mu+\xi+2)(3\mu-2\xi+2)(e^{2\xi\,y}-h)^4\,P\,,
\end{equation}
where $P$ is a polynomial in
$
e^{8y}\,,\,
e^{12y}\,,\,
e^{16y}\,,\,
e^{2(\xi+4)\,y}\,,\,
e^{4\xi\,y}\,,\,
e^{4(\xi+1)\,y}\,,\,
e^{4(\xi+2)\,y}\,,\,
e^{2(\xi+6)\,y}\,,\,
e^{2(\xi+2)\,y}
$
with constant coefficients.
These are linearly independent functions if
\begin{equation}\label{eqn:detM.1c.generic.condition}
\xi^5\,e^{18(\xi+4)y}\,(\xi-1)^2(\xi-3)^2(\xi-6)^2(\xi+2)^2(\xi-4)^5(\xi-2)^9
\ne0
\end{equation}
(note that not all zeros of~\eqref{eqn:detM.1c.generic.condition} are admissible parameter values in Theorem~\ref{thm:matveev}).
Let us first assume that~\eqref{eqn:detM.1c.generic.condition} is non-zero.
In this generic situation, the vanishing of~\eqref{eqn:detM.1c} requires that any coefficient w.r.t.\ the independent functions vanishes. In particular, the coefficient of $e^{16y}$ has to vanish, i.e.
\[
 81\,h^2\,\xi^2\,(\xi+2)(\xi-1) = 0
\]
must hold, which is impossible for any admissible value of $h$ or $\xi$, see Theorem~\ref{thm:matveev}.
Let us now see what happens if~\eqref{eqn:detM.1c.generic.condition} is not satisfied. In view of the parameter restrictions on $\xi$ in Theorem~\ref{thm:matveev}, c.f.~\cite{matveev_2012}, only the following cases need to be considered:
$\xi\in\left\{ 2,\, 3,\, 4\right\}$.

In case $\xi=3$, the condition that~\eqref{eqn:detM.1c} must vanish can be solved straightforwardly, using that $(e^{Ny})_{N\in\Z}$ are linearly independent.
Solving the system of equations obtained by equating the coefficients of these linearly independent functions in~\eqref{eqn:detM.1c} to zero, we find
$(h,\varepsilon,\mu)=(\pm1,-1,-\tfrac23)$.
We reinsert this into~\eqref{eqn:Mv=b} and~\eqref{eqn:linear.system.dom3.mu}, and verify that this system is solvable. We also confirm that it does not admit a Killing vector field.
In the case $\xi=2$ we impose the additional assumption $\mu\ne0$, which ensures that no non-trivial Killing vector field exists\footnote{This excludes only the case $(\xi,h,\varepsilon)=(2,\pm1,\mp1)$, which is forbidden in Theorem~\ref{thm:matveev} anyway.}.
Requiring~\eqref{eqn:detM.1c} to vanish, we find $h\in\{\pm4,\pm\tfrac14,\pm1\}$ and $\varepsilon=\pm1$ with the constraint $\sgn(h)=-\varepsilon$.
Resubstituting this into~\eqref{eqn:Mv=b} and~\eqref{eqn:linear.system.dom3.mu} we verify that these parameter configurations actually admit the additional solution~$\a3$. For $h=-\varepsilon$, the metrics admit a Killing vector field, but we notice that these configurations are forbidden in Theorem~\ref{thm:matveev}.
Finally, if $\xi=4$, requiring~\eqref{eqn:detM.1c} to vanish implies $\mu=0$. Such metrics therefore would admit a Killing vector field, and can thus be ignored.

\subsection{Metrics A(II) of Theorem~\ref{thm:matveev}}
For metrics of type A(II) the Lie derivative on the 2-dimensional invariant subspace is of the form (II) in~\eqref{eqn:normal-matrices}. It is a 2-dimensional Jordan block for an eigenvalue of $\lie_w$ that has algebraic multiplicity (at least)~2.
This configuration allows for two possibilities on the entire, 3-dimensional space~$\solSp$ of solutions to~\eqref{eqn:linear.system}, namely the first and second of the cases in~\eqref{eqn:dom.3.Lw}, as they admit eigenvalues with algebraic multiplicity larger than~$1$.
Therefore,
\begin{equation}\label{eqn:A.II.cases}
 \lie_w\begin{pmatrix} \a1 \\ \a2 \\ \a3 \end{pmatrix}
 = \begin{pmatrix} \mu & & \\ & 1 & 1 \\ & & 1 \end{pmatrix}
 \begin{pmatrix} \a1 \\ \a2 \\ \a3 \end{pmatrix}
 \quad\text{or}\quad
 \lie_w\begin{pmatrix} \a1 \\ \a2 \\ \a3 \end{pmatrix}
 = \begin{pmatrix} 1 & 1 & \\ & 1 & 1 \\ & & 1 \end{pmatrix}
 \begin{pmatrix} \a1 \\ \a2 \\ \a3 \end{pmatrix}\,.
\end{equation}
Let us briefly comment on these cases.
In the first case, represented by the matrix
\smash{\tiny $\left(\begin{smallmatrix} \mu & & \\ & 1 & 1 \\ & & 1 \end{smallmatrix}\right)$},
$\lie_w$ admits two real eigenvalues with, respectively, algebraic multiplicity one and two.
In the second case, represented by the matrix
\smash{\tiny $\left(\begin{smallmatrix} 1 & 1 & \\ & 1 & 1 \\ & & 1 \end{smallmatrix}\right)$},
$\lie_w$ admits only one real eigenvalue, which is of algebraic multiplicity~3.

Let us begin with the first case of~\eqref{eqn:A.II.cases}.
%
%
Following~\cite{matveev_2012}, we perform the change of coordinates
$x=\frac{x_\text{old}+y_\text{old}}{2}$,
$y=\frac{x_\text{old}-y_\text{old}}{2}$\,,
leading to the projective connection for metrics of type A(II),
\begin{equation}
y'' = \frac{2h+f^+_6(y)}{8h\,y}
+ \frac{12h\,y-3f^-_6(y)}{8h\,y}\,y'
+ \frac{3f^+_6(y)-2h}{8h\,y}\,y'^2
- \frac{12hy+f^-_6(y)}{8h\,y}\,y'^3
\end{equation}
where $f^\pm_r(y)=e^{r\lambda y}\pm h^2 e^{-r\lambda y}$.
The determinant of the matrix $M$, obtained as explained in Section~\ref{sec:method.thm2}, is
\begin{multline}
\det(M) =
55296\,h^4\,y^4\,(\mu-1)^2\,e^{6y}\,\bigg(
9\,e^{24y}
+6h\,(5\mu+4)(2\mu+1)\,y\,e^{18y}
-30h\mu\,(\mu+2)\,e^{18y}
\\
-8h^2\,(\mu+2)(\mu-1)(2\mu+1)^2\,y^2\,e^{12y}
-6h^2(10\mu^2+20\mu+3)\,e^{12y}
\\
-6h^3\,(5\mu+4)(2\mu+1)\,y\,e^{6y}
-30h^3\mu\,(\mu+2)\,e^{6y}
+9\,h^4
\bigg)\,.
\end{multline}
Obviously, $\det(M)=0$ if $\mu=1$. Let us postpone this possibility for the moment and assume $\mu\ne1$. In that case, the determinant vanishes (identically) if and only if the big bracket does. However, the functions
\[
1, e^{24y}, ye^{18y}, e^{18y}, y^2e^{12y}, e^{12y}, ye^{6y}, e^{6y}
\]
are linearly independent, and thus, in order that the bracket vanishes, each coefficient of these functions has to vanish independently.
But this cannot occur since for instance the coefficient of $e^{24y}$ is always non-zero.
Next, let us turn back to the case $\mu=1$; for this case $\lie_w$ admits a 2-dimensional subspace $\solSp'\subset\solSp$ with
$\lie_wa=a$ for any $a\in\solSp'$,
and such a projective class would admit a non-trivial Killing vector field, see Lemma~\ref{la:homothetic.subspace}.

Finally, we need to consider the second case of~\eqref{eqn:A.II.cases},
which is the particular case, mentioned in Section~\ref{sec:method.thm2}, when we obtain an inhomogeneous matrix system.
It is obtained as follows. Firstly we obtain, in coordinates such that the projective vector field $w=\partial_x$,
\begin{equation*}
 \lie_w\begin{pmatrix} \a2 \\ \a3 \end{pmatrix} = \begin{pmatrix} \a2+\a3 \\ \a3 \end{pmatrix}
  \qquad\text{and}\qquad
 \partial_x \a1 = \a1 + \a2\,.
\end{equation*}
The first subsystem has, in fact, already been solved in~\cite{matveev_2012}: Its solution is among the cases of Theorem~\ref{thm:matveev}. We are therefore left with the (inhomogeneous) subsystem $\partial_x \a1 = \a1 + \a2$.
Before we continue, let us write $\a1 = e^x\tilde{a}+a_\mathrm{part}$, where we introduce $\tilde{a}$, which is (for the purposes here) a matrix that depends on $y$ only and which solves the homogeneous problem; on the other hand, $a_\mathrm{part}$ is a particular solution of the inhomogeneous problem.
A particular solution has, in fact, already been obtained in~\cite{matveev_2012},
\begin{equation}
a_\mathrm{part} = -\frac{xe^x}{8y^{\tfrac13}h^{\tfrac23}}\,
\begin{pmatrix}
(x-2y)he^{-3y}+(x+2y)e^{3y} & (x-2y)he^{-3y}-(x+2y)e^{3y} \\
(x-2y)he^{-3y}-(x+2y)e^{3y} & (x-2y)he^{-3y}+(x+2y)e^{3y}
\end{pmatrix}\,.
\end{equation}
Inserting this into~\eqref{eqn:linear.system}, we end up with the system of PDE
\begin{subequations}
\begin{align}
0 &= \tilde{a}_{11}-\frac23\,f_1\,\tilde{a}_{11}+2f_0\,\tilde{a}_{12} -\frac{y^{\frac23}}{4c^{\frac23}}\,\left( e^{3y}-he^{-3y} \right)
\label{eqn:linear.system.dom3.inhomogeneous.first} \\
\tilde{a}_{11,y} &= -2\,\tilde{a}_{12}+\frac43\,f_2\,\tilde{a}_{11}-\frac23\,f_1\,\tilde{a}_{12}-2f_0\,a_{22} -\frac{y^{\frac23}}{4c^{\frac23}}\,\left( e^{3y}+he^{-3y} \right)
\label{eqn:linear.system.dom3.inhomogeneous.second} \\
\tilde{a}_{12,y} &= -\frac12\,\tilde{a}_{22}+f_3\,\tilde{a}_{11}+\frac13\,f_2\,\tilde{a}_{12}-\frac23\,f_1\,\tilde{a}_{22} +\frac{y^{\frac23}}{4c^{\frac23}}\,\left( e^{3y}-he^{-3y} \right)
\label{eqn:linear.system.dom3.inhomogeneous.third} \\
\tilde{a}_{22,y} &= 2\,f_3\,\tilde{a}_{12}-\frac23\,f_2\,\tilde{a}_{22}\,.
\label{eqn:linear.system.dom3.inhomogeneous.fourth}
\end{align}
\end{subequations}
Although this is an inhomogeneous system, we can argue as in the homogeneous cases:
We differentiate~\eqref{eqn:linear.system.dom3.inhomogeneous.first} twice w.r.t.\ $y$, and use~\eqref{eqn:linear.system.dom3.inhomogeneous.second}---\eqref{eqn:linear.system.dom3.inhomogeneous.fourth} to replace derivatives.
As a result, we arrive at a linear system $Mv=b$, $v=(\tilde{a}_{11},\tilde{a}_{12},\tilde{a}_{22})^T$, for which it is easily confirmed that $\det(M)=0$.
We aim now to derive another condition which is incompatible with this.
To this end, let $M=(m_1,m_2,m_3)$ where $m_i$ denote the columns of $M$. Consider the matrix $B=(m_2,m_3,b)$.
Its determinant is
\begin{multline}
\det(B) =
7776h^{\tfrac{11}{3}}\,y^{\tfrac{11}{3}}\,\bigg(
7\,e^{39y}
+65h\,e^{33y}
+3h^2\,(96y+1)\,e^{27y}
+3h^3\,(96y-89)\,e^{21y} \\
-3h^4\,(96y+89)\,e^{15y}
-3h^5\,(96y-1)\,e^{9y}
+65h^6\,e^{3y}
+7h^7\,e^{-3y}
\bigg)\,,
\end{multline}
and indeed does not vanish for $h\ne0$.
Therefore, the linear system $Mv=b$ does not admit solutions,
according to the Rouch\'e-Capelli theorem.

\subsection{Metrics A(III) of Theorem~\ref{thm:matveev}}
Any metric in the projective class is given in the form
\[
  g = (\tan(x)-\tan(y))\,\left(h\,\frac{e^{-3\lambda\,x}}{\cos(x)}\,dx^2 + \frac{e^{-3\lambda\,y}}{\cos(y)}\,dx^2\right)\,.
\]
Let us assume its degree of mobility is $3$. The Lie derivative then takes the form
\[
 \lie_w = \begin{pmatrix} \mu & & \\ & \lambda & -1 \\ & 1 & \lambda\end{pmatrix}\,.
\]
For $\lambda=0$, the steps outlined in Section~\ref{sec:method.thm2} are easily carried out. We find that $\mu=0$ must hold for a non-trivial solution of~\eqref{eqn:Mv=b} to exist. This implies the existence of a Killing vector field.
Therefore, let us assume $\lambda\ne0$. This case is a little more involved.
In new coordinates $x=\frac{x_\text{old}+y_\text{old}}{2}$ and $y=\frac{x_\text{old}-y_\text{old}}{2}$, the projective connection reads
\begin{equation}
 y'' =
 \frac{f^+_6(y)+2h\cos(2y)}{4h\,\sin(2y)}
 +3\,\frac{-f^-_6(y)+2h\sin(2y)}{4h\,\sin(2y)}\,y'
 +\frac{3f^+_6(y)-2h\cos(2y)}{4h\,\sin(2y)}\,y'^2
 -\frac{f^-_6(y)+6\lambda h\,\sin(2y)}{4h\,\sin(2y)}\,y'^3
\end{equation}
where $f^\pm_6(y)=e^{6\lambda y}\pm h^2 e^{-6\lambda y}$.
From the $3\times3$ matrix system obtained as before, we compute the determinant to be
\begin{multline}\label{eqn:Thm2.1b.detM}
 \det(M)
 = 55296h^4\,(\lambda^2-2\lambda\mu+\mu^2+1)\,\cos^4(y)\sin^4(y)\,\bigg(
   9\,(\lambda^2+1)\,f^+_{12}(y)
   \\
   -30h\,\mu\,(2\lambda+\mu)\,\cos(2y)\,f^+_6(y)
   +3h\,(4\lambda^3+13\lambda^2\mu+10\lambda\mu^2+4\lambda-7\mu)\,\sin(2y)\,f^-_6(y)
   \\
   +2h^2\,(\lambda-\mu)(2\lambda+\mu)\,(\lambda^2+4\lambda\mu+4\mu^2+1)\,\sin(2y)^2
   -6h^2\,(3\lambda^2+20\lambda\mu+10\mu^2+3)
 \bigg)
\end{multline}
where we introduce $f^\pm_{12}(y)=e^{12\lambda y}\pm h^4 e^{-12\lambda y}$.
The functions $f^+_{12}(y),\cos(2y)f^+_6(y),\sin(2y)f^-_6(y),\sin(2y)^2,1$ are linearly independent, and
$\lambda^2-2\lambda\mu+\mu^2+1 = 0$
does not have real solutions.
Therefore each (constant real) coefficient in the bracket of~\eqref{eqn:Thm2.1b.detM} has to vanish independently, and in particular this is impossible for the coefficient of $f^+_{12}(y)$.

\subsection{Metrics B(I--III) of Theorem~\ref{thm:matveev}}
The metrics of types B(I), B(II) and B(III) are, written in real coordinates $x=\frac{z+\cc{z}}{2}$, $y=\frac{z-\cc{z}}{2}$, already such that $w=\partial_x$.
Computing a determinant is an algebraic task. Formally, it is therefore sufficient to replace~$h$ by a complex number~$H=\frac{C}{\cc{C}}=e^{2i\varphi}$ ($C=e^{i\varphi}$), and set $\varepsilon=-1$.
We can thus follow exactly the same steps as in, respectively, the cases A(I--III), but have to take into account complex solutions for $H$.
For brevity, we just summarize the results.

\paragraph{Metrics B(I)}
The exceptional cases are $\xi\in\{2,3,4\}$.
Firstly, if~$\xi=2$, we find only real solutions for~$H$, namely~$H=1$, since $|H|=1$. This implies $C=\cc{C}$ and thus $C=\pm1$. But this configuration is forbidden in Theorem~\ref{thm:matveev}.
Next, if~$\xi=3$, we find $(H,\mu)=(\pm1,-\tfrac23)$. This implies $C\in\{\pm1,\pm i\}$, but in fact these configurations correspond only to two metrics (since we can always rescale by a constant global factor),
\[
  g_{+} = (e^{3z}-e^{3\cc{z}})(e^{2z}dz^2-e^{2\cc{z}}d\cc{z}^2)
  \quad\text{and}\quad
  g_{-} = (e^{3z}+e^{3\cc{z}})(e^{2z}dz^2-e^{2\cc{z}}d\cc{z}^2)\,.
\]
Substituting this back into~\eqref{eqn:Mv=b} and~\eqref{eqn:linear.system.dom3.mu}, we find that it is compatible and admits a solution, but no Killing vector field. Therefore, the degree of mobility of these metrics is~3.
%
Finally, if~$\xi=4$, we obtain~$\mu=0$. This corresponds to a solution that admits a Killing vector field, see Section~\ref{sec:Thm2.A.I}.

\paragraph{Metrics B(II)}
The proof in the real Liouville case only uses $h\ne0$, and thus the same conclusion holds also in the complex Liouville case.

\paragraph{Metrics B(III)}
The proof is analogous to the real Liouville case. If $\lambda=0$, we quickly see that $\mu=0$, implying the existence of a Killing vector field.
For $\lambda\ne0$, we can proceed as in the case A(III) and find that no cases with degree of mobility equal to~3 do exist.

\subsection{Metrics C(I--III) of Theorem~\ref{thm:matveev}}\label{sec:Thm2.C}
For metrics of type C in Theorem~\ref{thm:matveev}, the explicit computations can be found in~\cite{matveev_2012}.

\section{Proof of Theorem~\ref{thm:theorem.3.revised}}\label{sec:proof.thm.3.revised}
\subsection{Method}
We prove Theorem~\ref{thm:theorem.3.revised} by following the approach outlined in~\cite{matveev_2012}.
It is based on a technique already known to Darboux~\cite[\S\S~688--689]{darboux_1897} and Eisenhart~\cite[pp.~323--325]{eisenhart_1909}, and more recently a similar approach is pursued in~\cite{kruglikov_2009}.
Essentially, this technique is based on the following fact:
If a metric~$g$ admits a Killing vector field, then the scalar curvature~$R$, the square of its length $\ell=g(dR,dR)$, and its Laplacian $\Delta R$ are functionally dependent. In particular, this means that the differentials $dR$ and $d\ell$ need to be linearly dependent.
%
Then, in any system of coordinates, the determinant
\begin{equation}\label{eqn:E.general}
  E = \det\begin{pmatrix}
        R_x & \ell_x \\
        R_y & \ell_y
      \end{pmatrix}
      = 0
\end{equation}
necessarily is zero if~$g$ admits a non-trivial Killing vector field.
The converse does not hold in general. But, for our purposes, it turns out to be already strong enough a restriction. Indeed, it is feasible to check all metrics that satisfy~\eqref{eqn:E.general} individually for the existence of Killing vector fields.
According to Theorem~3 of~\cite{matveev_2012} none of the metrics in Theorem~\ref{thm:theorem.3.revised} do admit a non-trivial Killing vector field, but we actually find a few counterexamples to this statement. These have to be removed from the ultimate list of normal forms, as we do in Theorem~\ref{thm:normal.forms}.

In Sections~\ref{sec:proof.thm.3.A.I}---\ref{sec:proof.thm.3.C}, we study Equation~\eqref{eqn:E.general} for each case of Theorem~\ref{thm:matveev}. We find a number of parameter values $h,C$, $\varepsilon$ and $\lambda$ (resp., $\xi$), for which the corresponding metrics satisfy condition~\eqref{eqn:E.general} identically.
More precisely, our strategy is as follows: We can formulate~\eqref{eqn:E.general} as a polynomial equation whose variables are certain linearly independent functions. All coefficients w.r.t.\ these independent functions have to vanish identically. Thus, we obtain a system of equations on the parameter values. Finally, as~\eqref{eqn:E.general} is only a necessary criterion, we need to check if the resulting metrics, specified by these parameter values, actually admit a non-trivial Killing vector field. For explicit metrics, this can be done easily using standard computer algebra routines.
We were using Maple\textsuperscript{TM} for this paper.
Alternatively, one could also resort to criteria as in~\cite{kruglikov_2009}.

Linear independence of functions can be checked with their Wronskian. Note that the vanishing of the Wronskian is only a necessary criterion for linear dependence of functions. But, for the purposes here, it is sufficient for reaching a conclusion.
We work with two independent coordinates $x,y$. Nontheless instead of generalized Wronskians only usual Wronskians are going to be used.

\subsection{Metrics A(I) of Theorem~\ref{thm:matveev}}\label{sec:proof.thm.3.A.I}
We can express the determinant $E$ in~\eqref{eqn:E.general} as a polynomial
\begin{equation}\label{eqn:E.A.I}
  E = \sum_k b_k e^{f_k(x,y,\xi)}
\end{equation}
where $f_k(x,y;\xi)$ are certain polynomial functions $f_k(x,y;\xi)=(c^k_{11}\xi+c^k_{12})x+(c^k_{21}\xi+c^k_{22})y$ for integers $c^k_{ij}\in\Z$.
A straightforward computation yields the Wronskian w.r.t.~$x$, of the functions $e^{f_k(x,y;\xi)}$, as
\[
  W = P(\xi)\,\prod_k e^{f_k(x,y;\xi)}\,,
\]
where $P(\xi)$ is a polynomial in $\xi$ and does not depend on $x,y$.
Solving $P(\xi)=0$ confirms that the functions $f_k(x,y;\xi)$ are linearly independent if $\xi$ satisfies
\begin{equation}\label{eqn:thm3.1c.nu}
\xi \not\in \left\{ 1,2,3,4,\frac12,\frac23,\frac25,\frac32,\frac43,\frac45,\frac65,\frac83,\frac85 \right\}
\end{equation}
In the generic case~\eqref{eqn:thm3.1c.nu}, Equation~\eqref{eqn:E.general} implies that every coefficient $b_k$ in~\eqref{eqn:E.A.I} vanishes.
In particular the coefficient of $e^{2(\xi-1)(3x+y)}$ in~\eqref{eqn:E.A.I} has to vanish, i.e.
\[
-2h\varepsilon\,(\xi-1)(\xi+2)(\xi^2-3\xi+30) = 0\,.
\]
Obviously, provided~\eqref{eqn:thm3.1c.nu}, this cannot occur for real, positive~$\xi$ (see also the constraints on $\xi$ stated in Theorem~\ref{thm:matveev}).
We conclude that metrics of type A(I) do not admit a non-trivial Killing vector field, except maybe when~\eqref{eqn:thm3.1c.nu} does not hold.
Let us examine these special cases:
For the fractional values in~\eqref{eqn:thm3.1c.nu}, and also if $\xi=3$, it is easy to see that~\eqref{eqn:E.general} implies $h=0$, but this is forbidden in Theorem~\ref{thm:matveev}.
Next, if~$\xi=1$, Equation~\eqref{eqn:E.general} is always satisfied, regardless of the value of $h$ or $\varepsilon$. However, $\xi=1$ is forbidden in Theorem~\ref{thm:matveev} already, so no further restriction applies.
Assuming~$\xi=2$, we find that~\eqref{eqn:E.general} implies $h=-\varepsilon$, i.e.\ $(h,\varepsilon)=(1,-1)$ or $(h,\varepsilon)=(-1,1)$. These cases indeed admit a non-trivial Killing vector field. However, again, these cases are already forbidden in Theorem~\ref{thm:matveev}.
Finally, let us consider~$\xi=4$. Taking the coefficient of $e^{22x+2y}$, we infer from Equation~\eqref{eqn:E.general} that
\[
 -24\varepsilon\,(h-1) = 0\,.
\]
This implies $(\xi,h,\varepsilon)=(4,1,\pm1)$, and we indeed confirm that the corresponding metrics~\eqref{eqn:metric.nu.4.h.1} admit a non-trivial Killing vector field, see page~\pageref{thm:theorem.3.revised} for the explicit formula.
This, indeed, is a counterexample to the claim of Theorem~3 of~\cite{matveev_2012}.

\subsection{Metrics A(II) of Theorem~\ref{thm:matveev}}\label{sec:thm.3.A.II}
By a straightforward computation, using the explicit coordinate expressions of Theorem~\ref{thm:matveev}, we obtain~\eqref{eqn:E.general} as a polynomial equation
\begin{equation}\label{eqn:E.A.II}
  E = \sum_{k=1}^5 \sum_{i,j} c_{ijk} x^i y^j f_k(x,y) = 0\,,
\end{equation}
w.r.t.\ the functions
\begin{multline*}
 f_1 = e^{9x+3y}+h^2\,e^{9y+3x} \,,\,\,
 f_2 = e^{9x+3y}-h^2\,e^{9y+3x} \,,\,\,
 f_3 = e^{12x}+h^4\,e^{12y}\,,\,\,
 f_4 = e^{12x}-h^4\,e^{12y} \,,\,\,
 f_5 = e^{6x+6y} \,.
\end{multline*}
Alternatively, we may write~\eqref{eqn:E.A.II} as a polynomial of the form
\begin{equation}\label{eqn:E.A.II.2}
 E = \sum_{i=1}^m C_{i} x^{\alpha_i} y^{\beta_i} e^{\nu_i x} e^{\rho_i y} = 0\,,\quad
 \text{where $\alpha_i, \beta_i, \nu_i, \rho_i\in\Z$}\,.
\end{equation}
This is a convenient form in view of the following condition for linear independence.
\begin{lemma}\label{la:linear.independence.polynomial}
 The set $\{x^{\alpha_i}e^{\beta_i x}\}_{i=1,\dots,n}$, where $\alpha_i,\beta_i\in\Z$, $n\in\mathds{N}$, is linearly independent if
 \[
   \prod_{i<j} (\alpha_i-\alpha_j) \ne0
   \quad\text{or}\quad
   \prod_{i<j} (\beta_i-\beta_j) \ne0
 \]
\end{lemma}
\noindent Applying the lemma twice, for $x$ and $y$, to~\eqref{eqn:E.A.II.2}, we find that if~\eqref{eqn:E.A.II} holds, $C_{i}=0$ for each~$i$.
This, however, is easily disconfirmed. Therefore, $E=0$ does not admit solutions, and thus no non-trivial Killing vector field exists.

\begin{proof}[Proof of Lemma~\ref{la:linear.independence.polynomial}]
The Wronskian of the functions $\{x^{\alpha_i}e^{\beta_ix}\}_{i=1\dots n}$ is equal to
\[
P = \frac{(-1)^n}{x^{N}}\, \tilde{P}\cdot \prod_{i=1}^n x^{\alpha_i}e^{\beta_ix}
\]
where $N=\frac{(n-1)n}{2}$ and $\tilde{P}$ is a polynomial of degree (at most) $N$.
The component of highest degree of~$P$ is $\hat{P}=\prod_{i<j}(\beta_i-\beta_j)$,
and the component of degree zero is equal to $\check{P}=\prod_{i<j}(\alpha_i-\alpha_j)$.
Therefore, if either $\hat{P}\ne0$ or $\check{P}\ne0$, the Wronskian is not (identically) zero.
\end{proof}

\subsection{Metrics A(III) of Theorem~\ref{thm:matveev}}
In this case, similarly to Section~\ref{sec:thm.3.A.II}, we obtain~\eqref{eqn:E.general} in the form
\begin{equation}\label{eqn:E.A.III}
  E = \sum_{i=1}^m C_i\,\cos(x)^{\alpha_i}\cos(y)^{\beta_i}\sin(x)^{\epsilon_i}\sin(y)^{\eta_i}\,e^{n_i\lambda x}\,e^{m_i\lambda y} = 0\,,
\end{equation}
where $C_i$ are constants depending on $h$ and $\lambda$, and where $\epsilon_i,\eta_i\in\{0,1\}$, $0\leq\alpha_i,\beta_i\leq23$, and $\alpha_i,\beta_i,n_i,m_i\in\mathds{Z}$.
Analogously to Lemma~\ref{la:linear.independence.polynomial}, we prove
\begin{lemma}\label{la:linear.independence.A.III}
 The set $\{\cos(x)^{\alpha_i}e^{\beta_i\lambda x},\sin(x)\cos(x)^{\rho_i}e^{\eta_i\lambda x}\}_{i=1,\dots,N}$, where $\alpha_i,\beta_i,\rho_i,\eta_i\in\Z$, $N\in\mathds{N}$, is linearly independent.
\end{lemma}
\begin{proof}
Consider the linear combination (with $a_{rm},b_{rm}\in\mathds{R}$; $n$ large enough)
\begin{equation}\label{eqn:lin.ind.A.III}
 0 = \sum_{r,m=0}^n\left(
	a_{rm}\,\cos(x)^{r}
	+ b_{rm}\,\sin(x)\cos(x)^{r}
    \right)\,e^{m\lambda\,x}\,.
\end{equation}
Evaluating~\eqref{eqn:lin.ind.A.III} at $x=x_0+2k\pi$ for $k\in\Z$ and $x_0\in\R$, we obtain a family of equations
\[
 0 = \sum_{r,m} \left(
          a_{rm}\,\cos(x_0)^{r}+b_{rm}\,\sin(x_0)\cos(x_0)^{r}
		\right)
		e^{m\lambda\,x_0}e^{2m\pi\lambda\,k}\,.
\]
Letting $B_m=e^{2m\lambda\pi}$ and
\[
  c_m(x_0) = e^{2m\lambda\pi\,x_0}
    \sum_r \left(
          a_{rm}\,\cos(x_0)^{r}+b_{rm}\,\sin(x_0)\cos(x_0)^{r}
        \right)\,,
\]
we therefore obtain a Vandermonde-like matrix equation (letting $k$ run from $1$ to $n$)
\begin{equation}\label{eqn:vandermonde}
 \begin{pmatrix}
  B_1 & \cdots & B_n \\
  \vdots & \ddots & \vdots \\
  B_1^n & \cdots & B_n^n
 \end{pmatrix}\,
 \begin{pmatrix}
  c_1(x_0)\\ \vdots\\ c_n(x_0)
 \end{pmatrix}
 =0
\end{equation}
with $B_i\ne B_j$ for $i\ne j$.
The matrix~\eqref{eqn:vandermonde} is invertible and therefore $c_1(x_0)=\dots=c_m(x_0)=0$ for any value of $x_0$.
This implies
\[
  \sum_r \left( a_{rm}\,\cos(x)^{r}+b_{rm}\,\sin(x)\cos(x)^{r} \right) = 0\qquad\forall\,m\,,
\]
whose symmetric respectively antisymmetric parts (w.r.t.~$x$) have to vanish independently,
\[
  \sum_r a_{rm}\,\cos(x)^{r} = 0
  \quad\text{and}\quad
  \sum_r b_{rm}\,\cos(x)^{r} = 0\,,\qquad\forall\,m\,.
\]
Thus $a_{rm}=0$ and $b_{rm}=0$ for all $0\leq r,m\leq n$.
This proves the assertion.
\end{proof}
\smallskip

\noindent Applying the lemma twice, for~$x$ as well as~$y$, we have to solve the system $\{C_i=0\}$ for $\lambda$ and $h$.
We find that the only solutions are $(\lambda=0,h=\pm1)$, and we notice that in Theorem~\ref{thm:matveev} this parameter configuration is forbidden.
Thus, we have confirmed that metrics of type A(III) in Theorem~\ref{thm:matveev} do not admit a non-trivial Killing vector field.

\subsection{Metrics B(I) of Theorem~\ref{thm:matveev}}
This case is analogous to the case of metrics A(I).
The determinant~$E$ in~\eqref{eqn:E.general} can be expressed as a polynomial
\begin{equation}\label{eqn:E.B.I}
E = \sum_k b_k e^{f_k(z,\bar{z},\xi)}
\end{equation}
where $f_k(z,\bar{z};\xi)$ are polynomial functions of the form $f_k(z,\bar{z};\xi)=(c^k_{11}\xi+c^k_{12})z+(c^k_{11}\xi+c^k_{12})\bar{z}$ with $c^k_{ij}\in\Z$.
The Wronskian w.r.t.~$z$, of the functions $e^{f_k(z,\bar{z};\xi)}$ is of the form
\[
W = P(\xi)\,\prod_k e^{f_k(z,\bar{z};\xi)}\,,
\]
where $P(\xi)$ is a polynomial in $\xi$ and does not depend on $z,\bar{z}$.
Solving $P(\xi)=0$ confirms that the functions $f_k(z,\bar{z};\xi)$ are linearly independent if $\xi$ satisfies
\begin{equation}\label{eqn:thm3.2c.nu}
\xi\not\in\left\{ 1, 2, 3, 4, \frac12, \frac23,
\frac25, \frac32, \frac43, \frac45,
\frac65, \frac83, \frac85 \right\}\,.
\end{equation}
In the generic case~\eqref{eqn:thm3.2c.nu}, Equation~\eqref{eqn:E.general} implies that every coefficient $b_k$ in~\eqref{eqn:E.B.I} vanishes.
In particular the coefficient of $e^{2\xi z+4\xi\cc{z}+2z+6\cc{z}}$ has to vanish, i.e.
\[
  (7\xi^2-27\xi+30)(\xi-1)(\xi^2-4)\,C^2\,\cc{C}^4 = 0\,.
\]
The only way to realize this equality (for $\xi\in\R$, $C\in\mathds{S}^1\subset\mathds{C}$) is $C=0$, but this value is forbidden by Theorem~\ref{thm:matveev}.
It remains to check the cases when~\eqref{eqn:thm3.2c.nu} is not satisfied.
Following the procedure outlined for case A(I), we find:
For~$\xi=2$, we obtain an equation that is polynomial in the (linearly independent) functions $e^{2z+6\cc{z}}$, $e^{6z+2\cc{z}}$. The coefficients vanish simultaneously whenever $c_2=\Im(C)=0$, i.e.~$C=\pm1$. However, this is forbidden in Theorem~\ref{thm:matveev}.
Second, if $\xi=3$, $c_2=\Im(C)=0$. This leads to metrics that admit a non-trivial Killing vector field, analogously to the case A(I). This is another counterexample to Theorem~3 of~\cite{matveev_2012}.
Next, we do not have to check the case~$\xi=1$, since it is completely forbidden in Theorem~\ref{thm:matveev}, but indeed any choice of $C\ne0$ results in a metric that admits a non-trivial Killing vector field.
Finally, any other exception to~\eqref{eqn:thm3.2c.nu} implies $C=0$, which is forbidden in Theorem~\ref{thm:matveev}.

\subsection{Metrics B(II) of Theorem~\ref{thm:matveev}}
Analogously to case A(II), we can express
\begin{equation}
E = \sum_{k=1}^5 \sum_{i,j} c_{ijk} z^i \cc{z}^j f_k(z,\cc{z})\,,
\end{equation}
where we define (note $C\in\mathds{S}^1\subset\mathds{C}$ does not appear, in contrast to case~A(II) where~$h$ appears)
\begin{equation*}
f_1 = e^{9z+3\cc{z}}+e^{9\cc{z}+3z}\,,\,\,
f_2 = e^{9z+3\cc{z}}-e^{9\cc{z}+3z} \,,\,\,
f_3 = e^{12z}+e^{12\cc{z}} \,,\,\,
f_4 = e^{12z}-e^{12\cc{z}} \,,\,\,
f_5 = e^{6z+6\cc{z}} \,.
\end{equation*}
We proceed as in the case A(II), and confirm that metrics of type B(II) admit no non-trivial Killing vector field.

\subsection{Metrics B(III) of Theorem~\ref{thm:matveev}}
Arguing as in the case A(III), we consider the polynomial equation
\begin{equation}\label{eqn:E.B.III}
  E = \sum_{i=1}^m K_i\,\cos(z)^{\alpha_i}\cos(\bar{z})^{\beta_i}\sin(z)^{\epsilon_i}\sin(\bar{z})^{\eta_i}\,e^{n_i\lambda z}\,e^{m_i\lambda \bar{z}} = 0\,,
\end{equation}
where $\epsilon_i,\eta_i\in\{0,1\}$, $0\leq\alpha_i,\beta_i\leq23$, and $\alpha_i,\beta_i,n_i,m_i\in\mathds{Z}$.
However, the constants $K_i$ do now depend on $\lambda\in\R$ and $H=\frac{C}{\bar{C}}$ where $C\in\mathds{S}^1\subset\mathds{C}$ is the parameter from Theorem~\ref{thm:matveev}.

By the same reasoning as in the case A(III), we arrive at the conclusion that~\eqref{eqn:E.B.III} can only hold if $\lambda=0$ and $C=\pm\bar{C}$. Since $|C|=1$, this implies $C\in\{\,\pm1,\,\pm i\,\}$. We discard $C=\pm1$ as this is forbidden in Theorem~\ref{thm:matveev}.
For $C=\pm i$, we actually find that the metric~\eqref{eqn:metric.lambda.0.C.im}
admits a non-trivial Killing vector field, contrary to the statement in Theorem~3 of~\cite{matveev_2012}.
The explicit form of the Killing vector field can be found in the statement of Theorem~\ref{thm:theorem.3.revised}.

\subsection{Metrics C(I--III) of Theorem~\ref{thm:matveev}}\label{sec:proof.thm.3.C}
Consider a metric of the form $g = (x+Y(y))\,dxdy$.
Its scalar curvature is given by $R = \frac{4\,Y'(y)}{(x+Y(y))^3}$
and, by a straightforward computation, we obtain
\begin{equation}\label{eqn:E.C}
 E = \frac{768\,Y'(y)}{(x+Y(y))^{12}}\,\left( 3Y'''(y)Y'(y)-5Y''(y)^2 \right)\,.
\end{equation}
For the metrics of types C(I--III), obviously $Y'(y)\ne0$, and therefore Equation $E=0$ with $E$ as in~\eqref{eqn:E.C} implies
\[
  3Y'''(y)Y'(y)-5Y''(y)^2 = 0\,.
\]
It is easily verified that this ODE is not satisfied by the functions $Y(y)$ for the metrics under consideration.
Thus, we conclude that metrics of the types C(I--III) do not admit a non-trivial Killing vector field.
Of course, this statement is also contained in~\cite{matveev_2012}, where a different reasoning is employed, see Section~4.5 of~\cite{matveev_2012}.

\section{Proof of Propositions~\ref{prop:dom3.proj.equiv} and~\ref{prop:homotheties.and.eigenvectors}, and parameter refinements}\label{sec:preparations}
Before we commence with the actual proof of Theorem~\ref{thm:normal.forms}, we first prove some preparative statements. In particular, we prove Proposition~\ref{prop:homotheties.and.eigenvectors} which responds to one of the shortcomings of reference~\cite{matveev_2012} as outlined in the list after Theorem~\ref{thm:matveev}.
Next, we prove Proposition~\ref{prop:dom3.proj.equiv}, which constitues a huge simplification for the case of degree of mobility~3.
Finally, we remove some obvious ambiguity from the parameters in Theorem~\ref{thm:matveev}. In the cited theorem they are not optimal in the sense that two sets of parameters can represent the same projective class. A priori there can be even more ambiguity hidden in the parameters and it is not before the end of the proof of Theorem~\ref{thm:normal.forms} that we can confirm that actually all ambiguity has been removed.

\subsection{Proof of Proposition~\ref{prop:dom3.proj.equiv}}\label{sec:dom3.proj.equiv}
We prove that the metrics in Theorem~\ref{thm:matveev} that have degree of mobility~3 and admit exactly one projective vector field, are projectively equivalent on some subset of their domain.
For the proof, we first choose suitable metrics in each projective class, and then prove that these metrics are isometric on a certain part of their domain.

\paragraph{Representatives for the projective classes.}
Theorem~\ref{thm:theorem.2.revised} tells us which metrics in Theorem~\ref{thm:matveev} have degree of mobility~3 and admit exactly one projective vector field.
On the other hand, in Theorem~\ref{thm:matveev}, projective classes are described, via Formula~\eqref{eqn:space.G}, by metrics that have the form specified by Proposition~\ref{prop:dini}.
For our purposes, suitable metrics are:
\begin{itemize}
	\item Metrics A(I) with $(\xi,h,\varepsilon)=(3,1,-1)$ are projectively equivalent to
	\begin{align}\label{eqn:metric.g1a}
	\g1^a &= (e^{3u}-e^{3v})(e^{2u}du^2-e^{2v}dv^2)\,,
	\end{align}
	The projective vector field is $w=\partial_u+\partial_v$ and $\lie_w\g1=(\xi+2)\g1$.
	\item Metrics A(I) with $(\xi,h)\in\{(2,4), (2,\nicefrac14)\}$, $\varepsilon=-1$, are projectively equivalent to
	\begin{equation}\label{eqn:metric.g2b}
	 \g2^b = (e^{-2y}-4e^{-2x})(4dx^2-dy^2)
	\end{equation}
	The projective vector field is $w=\partial_x+\partial_y$ and $\lie_w\g1=(\xi+2)\g1$.
	\item Metrics of type B(I) with $(\xi,C)=(3,\pm1)$ are projectively equivalent to
	\begin{equation}\label{eqn:metric.g1B+}
	\g1^{B+} = (e^{3z}-e^{3\cc{z}})(e^{2z}dz^2-e^{2\cc{z}}d\cc{z}^2)\,.
	\end{equation}
	The projective vector field is $w=\partial_z+\partial_{\cc{z}}$ and $\lie_w\g1=5\g1$.
	\item Metrics of type B(I) with $(\xi,C)=(3,\pm i)$ are projectively equivalent to
	\begin{equation}\label{eqn:metric.g1B-}
	\g1^{B-} = (e^{3z}+e^{3\cc{z}})(e^{2z}dz^2-e^{2\cc{z}}d\cc{z}^2)\,.
	\end{equation}
	The projective vector field is $w=\partial_z+\partial_{\cc{z}}$ and $\lie_w\g1=5\g1$.
	\item Metrics of type C(Ia) are projectively equivalent to
	\begin{equation}\label{eqn:metric.g1C}
	\g1^C = (x+y^2)dxdy\,.
	\end{equation}
	The projective vector field is $w=2x\partial_x+y\partial_y$ with $\lie_w\g1=5\g1$.
\end{itemize}

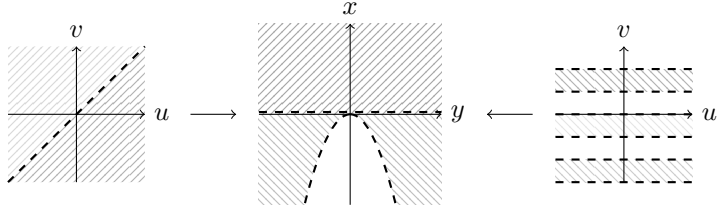
\begin{figure}
\begin{center}
\begin{tikzpicture}[scale=0.3]
\fill[pattern=north east lines,pattern color=gray!30] (-15,-3) -- (-9,3) -- (-15,3) -- cycle;
\fill[pattern=north east lines,pattern color=gray!60] (-9,-3) -- (-15,-3) -- (-9,3) -- cycle;
\draw[dashed, thick] (-15,-3) -- (-9,3);
\draw[->] (-15,0) -- (-9,0) node[right] {$u$};
\draw[->] (-12,-3) -- (-12,3) node[above] {$v$};
\draw[->] (-7,0) -- (-5,0);
\fill[pattern=north west lines,pattern color=gray!50] (0,0) parabola (2,-4) |- (0,0);
\fill[pattern=north west lines,pattern color=gray!50] (0,0) parabola (-2,-4) |- (0,0);
\fill[pattern=north west lines,pattern color=gray!50] (2,-4) rectangle (4,0);
\fill[pattern=north west lines,pattern color=gray!50] (-2,-4) rectangle (-4,0);
\fill[pattern=north east lines,pattern color=gray!70] (-4,0) rectangle (4,4);
\draw[dashed, thick] (0,0) parabola (2,-4);
\draw[dashed, thick] (0,0) parabola (-2,-4);
\draw[dashed, thick] (-4,0.1) -- (4,0.1);
\draw[->] (-4,0) -- (4,0) node[right] {$y$};
\draw[->] (0,-4) -- (0,4) node[above] {$x$};
\draw[->] (8,0) -- (6,0);
\fill[pattern=north west lines,pattern color=gray!20] (9,-1) rectangle (15,0);
\fill[pattern=north west lines,pattern color=gray!40] (9,-3) rectangle (15,-2);
\fill[pattern=north west lines,pattern color=gray!60] (9,1) rectangle (15,2);
\draw[dashed, thick] (9,0) -- (15,0);
\draw[dashed, thick] (9,-1) -- (15,-1);
\draw[dashed, thick] (9,1) -- (15,1);
\draw[dashed, thick] (9,2) -- (15,2);
\draw[dashed, thick] (9,-2) -- (15,-2);
\draw[dashed, thick] (9,-3) -- (15,-3);
\draw[->] (9,0) -- (15,0) node[right] {$u$};
\draw[->] (12,-3) -- (12,3) node[above] {$v$};
\end{tikzpicture}
\caption{Sketch of the patches mapped onto one another by coordinate transformations mediating between the metrics $g_1^a$, $g_1^C$ and $g_1^{B\pm}$. In particular the coordinate diagonal of~\eqref{eqn:metric.g1a} is mapped onto the $y$-axis of~\eqref{eqn:metric.g1C}, and a family of parallel lines for~\eqref{eqn:metric.g1B+} and~\eqref{eqn:metric.g1B-}, see the left, middle or right panel of the figure, respectively.}
\label{fig:prop4.coordinate.trafos}
\end{center}
\end{figure}

\noindent We emphasize again that everything here is local, i.e.\ we are do not consider `internal' coordinate transformations that map a neighborhood of any of the metrics~\eqref{eqn:metric.g1a}--\eqref{eqn:metric.g1C} onto another neighborhood of the same metric.
We start with showing that the metric $g_1^a$ is isometric to $g_1^C$.
\begin{lemma}\label{la:g1a.g1C}
Let $\mathcal{U}$ be a neighborhood in the domain of~\eqref{eqn:metric.g1a}.
There is a coordinate transformation $\tau$ such that
$\tau^*(g|_\mathcal{U})=\pe{g}|_{\pe{\mathcal{U}}}$
where $\pe{\mathcal{U}}$ is a neighborhood in the domain of~\eqref{eqn:metric.g1C} with $x>0$.
\end{lemma}
\begin{proof}
Obviously, the metric~\eqref{eqn:metric.g1a} is not defined for $u=v$. On the other hand, the metric~\eqref{eqn:metric.g1C} is not defined on the parabola $x=-y^2$ (compare the dashed lines in Figure~\ref{fig:prop4.coordinate.trafos}).
Consider the transformation
\begin{equation}
 \tau(u,v) = \begin{pmatrix} x \\ y \end{pmatrix} =
 \begin{pmatrix}
   \frac13\,k^2\,(e^u-e^v)^2 \\
   k\,(e^u+e^v)
 \end{pmatrix}
\end{equation}
with $k=\sqrt[5]{\nicefrac98}$.
Its Jacobian is non-zero whenever $u\ne v$.
Under this transformation, we have
\[
 u=v \Leftrightarrow x=0\,,
\]
and $x>0$, $y>0$ for all values of $u,v$.
It is then easily confirmed that
\[
  \{\tau(u,v)\,|\,u>v\} = \{\tau(u,v)\,|\,u<v\} = \{(x,y)\,|\,x>0,\,y>0\}\,.
\]
Similarly, if we use the transformation (still $k=\sqrt[5]{\nicefrac98}$)
\begin{equation}
\tau(u,v) = \begin{pmatrix} x \\ y \end{pmatrix} =
\begin{pmatrix}
  \frac13\,k^2\,(e^u-e^v)^2 \\
  -k\,(e^u+e^v)
\end{pmatrix}\,,
\end{equation}
then we obtain
\[
  \{\tau(u,v)\,|\,u>v\} = \{\tau(u,v)\,|\,u<v\} = \{(x,y)\,|\,x>0,\,y<0\}\,.
\]
This concludes the proof.
\end{proof}

\noindent Next, we show that the metrics $g_1^{B\pm}$ are isometric to $g_1^C$.
\begin{lemma}\label{la:g1B.g1C}
Let $\mathcal{U}$ be a neighborhood in the domain of~\eqref{eqn:metric.g1B+}.
Then there is a coordinate transformation $\tau$ such that
$\tau^*(g|_\mathcal{U})=\pe{g}|_{\pe{\mathcal{U}}}$
where $\pe{\mathcal{U}}$ is a neighborhood in the domain of~\eqref{eqn:metric.g1C} with $x<0$.
Similarly, let $\mathcal{U}$ be a neighborhood in the domain of~\eqref{eqn:metric.g1B-}.
Then there is a coordinate transformation $\tau$ such that
$\tau^*(g|_\mathcal{U})=\pe{g}|_{\pe{\mathcal{U}}}$
where $\pe{\mathcal{U}}$ is a neighborhood in the domain of~\eqref{eqn:metric.g1C} with $x>0$.
\end{lemma}
\begin{proof}
The proof is similar to that of Lemma~\ref{la:g1a.g1C}.
However, the metric $g_1^{B+}$ is not defined for
\[
 \sin(3v)=0\quad\Rightarrow\quad
 v=\frac{N}{3}\,\pi\quad (N\in\mathds{Z})\,.
\]
Consider the coordinate transformation
\begin{equation}\label{eqn:trafo.g1B+}
 \tau(u,v) = \begin{pmatrix} x \\ y \end{pmatrix}
  = \begin{pmatrix}
     \frac13\,k^2\,(e^z-e^{\bar z})^2 \\
     \pm k\,(e^z+e^{\bar z})
    \end{pmatrix}\,,\qquad
 k= \sqrt[5]{\nicefrac98}\,.
\end{equation}
We see that $x<0$ for all values of $z,\bar{z}$. This transformation has non-zero Jacobian if
\[
 \sin(v)=0\quad\Rightarrow\quad
 v=N\pi\quad (N\in\mathds{Z})\,.
\]
The situation is illustrated in the middle and right panel of Figure~\ref{fig:prop4.coordinate.trafos}.
Using the coordinate transformation, we obtain
\[
  \{(x,y)\,|\,x+y^2=0\} = \{\tau(u,v)\,|\,v=\tfrac{N}{3}\,\pi\ \ N\in\mathds{Z}\}
\]
and
\[
  \{ (x,y)\,|\,x=0\} = \{ \tau(u,v)\,|\,v=N\pi, N\in\mathds{Z}\}\,.
\]
It is then easily seen that
\begin{align*}
  -y(u,v)^2 < x(u,v) < 0
  &\quad\text{for }
  \frac{N\pi}{3}<v<\frac{(N+1)\pi}{3}\text{ and $N$ odd}\,, \\
  x(u,v) < -y(u,v)^2
  &\quad\text{for }
  \frac{N\pi}{3}<v<\frac{(N+1)\pi}{3}\text{ and $N$ even}\,.
\end{align*}
This concludes the first part of the proof.
For metric~\eqref{eqn:metric.g1B-}, the conclusion can be drawn from an analogous reasoning, using the transformation
\begin{equation}\label{eqn:trafo.g1B-}
 \tau(u,v) = \begin{pmatrix} x \\ y \end{pmatrix}
  = \begin{pmatrix}
     \frac13\,k^2\,(e^z+e^{\bar z})^2 \\
     \pm k\,(e^z-e^{\bar z})
    \end{pmatrix}\,,\qquad
  k=\sqrt[5]{\nicefrac98}\,.
\end{equation}
\end{proof}

\noindent Finally we check that the projective class of $g_2^b$ coincides with that of $g_1^C$. For the proof, however, we choose representatives different from~\eqref{eqn:metric.g1C}.
\begin{lemma}\label{la:g1b.g1C}
Let $\mathcal{U}$ be a neighborhood in the domain of~\eqref{eqn:metric.g2b}.
Then there exists a coordinate transformation $\tau$ such that
$\tau^*(g|_\mathcal{U})=\pe{g}|_{\pe{\mathcal{U}}}$
where $\pe{\mathcal{U}}$ is a neighborhood in the domain of a metric~$\pe{g}$ that is (on $\pe{\mathcal{U}}$) projectively equivalent to~\eqref{eqn:metric.g1C}.
\end{lemma}
\begin{proof}
The proof is similar to those of Lemmas~\ref{la:g1a.g1C} and~\ref{la:g1B.g1C}.
The following metric is clearly projectively equivalent to~\eqref{eqn:metric.g1C}:
\[
  \hat{g} = -2\,\frac{y^2+x}{y^3}dxdy + \frac{(y^2+x)^2}{y^4}dy^2\,.
\]
It is not defined on the parabola $x=-y^2$ as well as for $y=0$.
Consider the transformation
\[
 \tau(u,v) = \begin{pmatrix} x \\ y \end{pmatrix} =
 \begin{pmatrix}
   \frac13\,\left( -e^{-2u}+e^{2v-4u} \right) \\
   \pm e^{v-2u}
 \end{pmatrix}
\]
Its Jacobian is non-zero for any value of $u,v$.
Furthermore, we see
\[
  x+y^2=0 \Leftrightarrow v=u-\ln(2)\,.
\]
Finally, it is confirmed that
\[
  x(u,v) > -y(u,v)^2 \text{  if $v>u-\ln(2)$\quad and\quad }
  x(u,v) < -y(u,v)^2 \text{  if $v<u-\ln(2)$}\,.
\]
This concludes the proof.
\end{proof}

\subsection{Proof of Proposition~\ref{prop:homotheties.and.eigenvectors}}\label{sec:proof.homotheties.and.eigenvectors}
The parametrical description of the projective classes given in Theorem~\ref{thm:matveev} is not optimal.
Before we commence on the proof of Theorem~\ref{thm:normal.forms}, we therefore remove some of the parametrical ambiguity in order to simplify the body of the proof.
This parametrical simplification is achieved in two ways:
First, in the present Section~\ref{sec:proof.homotheties.and.eigenvectors}, we prove Proposition~\ref{prop:homotheties.and.eigenvectors}, which allows us to address item~2 of the list after Theorem~\ref{thm:matveev} in Section~\ref{sec:state}.
Second, in the remainder of Section~\ref{sec:preparations}, we determine refined parameter ranges for the metrics \eqref{eq.metrics.Liouville}--\eqref{eq.metrics.JB}.
Specifically, for this latter improvement, we identify sets of parameters that give rise to the same projective class, and then restrict the parameter ranges accordingly in the metrics \eqref{eq.metrics.Liouville}--\eqref{eq.metrics.JB}.

\paragraph{Proof of Proposition~\ref{prop:homotheties.and.eigenvectors}.} Proposition~\ref{prop:homotheties.and.eigenvectors} characterizes metrics that do not admit a homothetic vector field\footnote{We need to impose that the projective algebra is 1-dimensional, but this is straightforward by Theorem~\ref{thm:theorem.3.revised}.}.
In~\cite{matveev_2012}, it has already been remarked (with examples) that projective classes of metrics with exactly one, essential projective vector field can still contain metrics for which this vector field becomes homothetic.
We prove that such metrics are linked to the eigenvalues of $\lie_w$, which allows us to easily remove these metrics from our list.\medskip

\noindent First, assume there is a homothetic vector field~$w$. Then, since $\dim(\projalg(g))=1$, we have $\lie_wg=\eta\,g$.
Furthermore, according to Equation~\eqref{eqn:Lwa},
\[
 \lie_wa
 =\det(g)^{-\frac23} \cdot \lie_wg
 - \frac23\,\det(g)^{-\frac23}\tr_g(\lie_wg) \cdot g
 =\eta\,a - \frac43\,\eta\,\cdot a
 =: \mu\,a\,.
\]
On the other hand, if $\lie_wa=\mu\,a$, then compute the components of $\lie_wg$, i.e.\
\begin{align*}
\lie_wg_{ij}
&=w^k\partial_k\left(\frac{a_{ij}}{\det(a)^2}\right)
=\frac{\lie_wa_{ij}}{\det(a)^2}-\frac{\beta\partial_kw^k}{\det(a)^2}
-2\det(a)^{-3}\,w^k\partial_k(\det(a))\,a_{ij} \\
&=\frac{\lie_wa_{ij}}{\det(a)^2}
-\frac{\beta\partial_kw^k}{\det(a)^2}\,a_{ij}
-2\,\frac{w^k\tr_g(\partial_ka)}{\det(a)^2}\,g_{ij}
\stackrel{\text{($\ast$)}}{=}
3\mu\,g_{ij}
+(4+3\beta)\partial_kw^k\,g_{ij}
=-3\mu\,g_{ij}
\end{align*}
where $\beta=-\frac43$ is the weight of the Liouville tensor $a$ and where we use $\lie_wa=\mu\,a$ at ($\ast$). This confirms $\lie_wg=\eta\,g$ with $\eta\in\mathds{R}$, implying $w$ is a homothetic vector field.
This concludes the proof.

\subsection{Optimizing the parameters~\texorpdfstring{$\lambda$ and $\xi$}{that represent eigenvalues of the Lie derivative} in Theorem~\ref{thm:matveev}}\label{sec:optimizing.lambda.xi}
The parameters $\lambda$ and $\xi$ (see also~\eqref{eqn:xi.via.lambda}) in Theorem~\ref{thm:matveev} are not sharp, in the sense that we could pick two cases with different values of those parameters, but whose projective classes are identical.
The reason is that in~\eqref{eqn:normal-matrices} the parameter $\lambda$ is not optimal.
The following lemmas fix this ambiguity.
\begin{lemma}\label{la:lambda.different.classes}
 (a) For metrics of cases~(I) in Theorem~\ref{thm:matveev}, two different values of~$\xi=\xi(\lambda)$, see~\eqref{eqn:xi.via.lambda}, give rise to different projective classes if $|\lambda|\ne1$ (i.e., $\xi\ne0$ and $\xi\ne4$).

 (b) In cases~(III), two different parameter values~$\lambda$, $\lambda'$, with $|\lambda'|\ne|\lambda|$, give rise to different projective classes. Furthermore, w.l.o.g., the parameter value can be chosen non-negative.
\end{lemma}
\begin{proof}
 Recall the considerations in Section~\ref{sec:lie.action}. For the projective vector field~$w$, the eigenvalues of~$\lie_w$ are fixed, up to the freedom of rescaling~$w$. On the other hand, we can reorder eigenvalues. These are the only liberties we have in rephrazing the eigenvalue problem. In~\eqref{eqn:normal-matrices}, these two ambiguities have already been used to impose a certain normalization and ordering of the eigenvalues, however the freedom has only been exploited partially. This is where we start from in this proof.

 In case~(I) of~\eqref{eqn:normal-matrices}, the eigenvalues can be reordered and rescaled such that they become $\{\lambda,1\}$ with $|\lambda|\geq1$. Therefore, the ratio of their absolute values is fixed. The only case when such a choice is not unique is when $|\lambda|=1$, i.e.\ we deal with the situation of two eigenvalues that are equal up to their sign.

 Similarly, in case~(III), the (pair of complex) eigenvalues given in~\eqref{eqn:normal-matrices} is $( \lambda+i\,, \lambda-i )$. Obviously, to achieve this, only a rescaling of~$w$ is needed. If we also use our freedom to reorder the eigenvalues, we find that $( \lambda-i\,, \lambda+i )$ is indistinguishable from $( \lambda'+i, \lambda'-i )$, where $\lambda'=-\lambda$. Indeed, rescaling $( \lambda-i\,, \lambda+i )$ by $-1$ yields $( \lambda'+i, \lambda'-i )$, with $\lambda'=-\lambda$. We have thus established that $\lambda$ is unique up to its sign, i.e.\ the absolute value of $\lambda$ cannot be different for different projective classes.
\end{proof}
Let us now consider the case (I) when $|\lambda|=1$.
\begin{lemma}\label{la:lambda.1.homothetic}
 If, in case (I) of~\eqref{eqn:normal-matrices}, we have $\lambda=1$ and a unique (up to rescaling) projective vector field, then this vector field is homothetic.
\end{lemma}
\begin{proof}
 Obviously, for any solution~$a$ to~\eqref{eqn:linear.system}, we have $\lie_wa=a$. The claim then follows from Proposition~\ref{prop:homotheties.and.eigenvectors}.
\end{proof}

\subsection{Optimizing the parameters \texorpdfstring{$h$}{h} and \texorpdfstring{$C$}{C} in Theorem~\ref{thm:matveev}}\label{sec:optimizing.h.C}
The real Liouville metrics in Theorem~\ref{thm:matveev} depend on a (real) parameter~$h$,
and in the complex Liouville cases we similarly have a complex parameter~$C$, which w.l.o.g.\ can be assumed to have absolute value $|C|=1$. Examining the Liouville type metrics in Proposition~\ref{prop:dini} and Theorem~\ref{thm:matveev}, we notice a structural feature of these metrics: they are roughly symmetric under the exchange of the roles of~$x$ and~$y$.
Let us investigate how swapping $x\leftrightarrow y$ can actually produce restrictions on the parameters.
Consider metrics of type A(I) in Theorem~\ref{thm:matveev},
\begin{equation}
 \g1 =\g1[\xi,h,\varepsilon]
  = (x^\xi-hy^\xi)(dx^2+\varepsilon dy^2)\,,\quad
 \g2 =\g2[\xi,h,\varepsilon]
  = \left(\frac{1}{x^\xi}-\frac{1}{hy^\xi}\right)\left(\frac{dx^2}{x^\xi}+\varepsilon\frac{dy^2}{hy^\xi}\right)
\end{equation}
(note the change of coordinates, and the restriction $x>0$, $y>0$).
Let us first rewrite this as
\begin{equation}
 \g1 = -h\varepsilon\,\left(y^\xi-\frac{1}{h}\,x^\xi\right)(dy^2+\varepsilon dx^2)\,,\quad
 \g2 = -\frac{\varepsilon}{h^2}\,\left(\frac{1}{y^\xi}-\frac{1}{hx^\xi}\right)\left(\frac{dy^2}{y^\xi}+\varepsilon\frac{dy^2}{\tfrac{1}{h}y^\xi}\right)\,.
\end{equation}
Next, let us consider a linear combination
\begin{equation}
 g 
   = \K1 \g1[h,\xi,\varepsilon](x,y) + \K2 \g2[h,\xi,\varepsilon](x,y)
   = -\frac{\varepsilon}{h}\,\K1\,\g1[\tfrac{1}{h},\xi,\varepsilon](y,x) -\frac{\varepsilon}{h^2}\,\K2\,\g1[\tfrac{1}{h},\xi,\varepsilon](y,x)
\end{equation}
This means that by exchanging the variables $x\leftrightarrow y$, we are able to map the spaces $\solSp(\g1[\xi,h,\varepsilon])$ and $\solSp(\g1[\xi,\frac{1}{h},\varepsilon])$ onto each other.
Therefore, in the proof of Theorem~\ref{thm:normal.forms} we may restrict to $h\in\Rnz$ with $|h|\geq1$ (or, equivalently, $|h|\leq1$).
The same applies in the cases A(II) and A(III), with the proof identical to that for case A(I) outlined above.
An analogous reasoning works in the cases B(I), B(II) and B(III), with the following implication: We may in those cases exchange $z\leftrightarrow\cc{z}$ and, therefore, map the space spanned by
$(\g1[\xi,C],\g2[\xi,C])$, $C=e^{-i\varphi}$ with $\varphi\in\angl$, onto the one spanned by $(\g1[\xi,C],\g2[\xi,C])$, $C=e^{-i\varphi}$ with $\varphi\in\halfangl$.
Note that here we applied also the freedom to rescale~$C$ in order to achieve $|C|=1$. In addition, we have found that we may restrict the phase of~$C$ to $\varphi\in\halfangl$.

\section{Proof of Theorem~\ref{thm:normal.forms}}\label{sec:proof.main.theorem}
This section is devoted to the proof of the main theorem, i.e.\ the list of mutually non-diffeomorphic normal forms given in Theorem~\ref{thm:normal.forms}.
It is organized as follows:
\begin{enumerate}
 \item First, we outline the techniques applied for the main part of the proof of Theorem~\ref{thm:normal.forms}.
 Roughly speaking, we consider the metrics contained in Theorem~\ref{thm:matveev} in explicit coordinate choices (these metrics depend not only on the parameters appearing in~\eqref{eq.metrics.Liouville}--\eqref{eq.metrics.JB}, but also on those parameters involved in~\eqref{eqn:general.metric}).
 We then examine all possible coordinate transformations between them and use this freedom to derive the normal forms.
 \item In line with Theorem~\ref{thm:theorem.2.revised}, metrics with degree of mobility~3 can be treated separately, and we present two ways for arriving at non-diffeomorphic normal forms.
 \item The derivation of the actual normal forms is performed, for degree of mobility~2, in a case-by-case procedure. Each case is characterized by the number of (real) eigenvalues of~$\lie_w$ (c.f.\ the Roman numbers in Theorem~\ref{thm:matveev}) and of the Benenti tensor (c.f.\ letters in Theorem~\ref{thm:matveev}).
 \item As a by-product to the classification in terms of non-diffeomorphic normal forms, we obtain also an optimal list of the projective classes containing metrics with exactly one, essential projective vector field.
\end{enumerate}

\noindent Together with Theorems~\ref{thm:theorem.2.revised} and \ref{thm:theorem.3.revised}, and Proposition~\ref{prop:homotheties.and.eigenvectors}, the steps 1--3 constitute the proof of Theorem~\ref{thm:normal.forms}. Recall two observations: firstly, metrics with degree of mobility~1 need not be considered as the (unique up to rescaling) projective vector field is homothetic for such metrics. Secondly, metrics with degree of mobility at least~4 admit a non-trivial Killing vector field \cite{koenigs_1896} and thus are not of interest here.

\subsection{Techniques and Strategies}\label{sec:techniques.strategies}
Theorem~\ref{thm:normal.forms} takes account of the liberty to apply arbitrary coordinate transformations to a metric. This freedom is ``removed'' from the normal forms as follows: Theorem~\ref{thm:matveev} already fixes the form of the metric in the sense that the coordinate expressions in each case have a certain, parametric form.
Next, we have established in Theorems~\ref{thm:theorem.2.revised} and~\ref{thm:theorem.3.revised} which metrics~$g$ of Theorem~\ref{thm:matveev} have degree of mobility~2, or 3, and $\dim\projalg(g)=1$. For simplicity let us assume degree of mobility~2 (for degree of mobility~3 a similar reasoning holds, see Section~\ref{sec:proof.degree.of.mobility.3}).
We obtain the normal forms of Theorem~\ref{thm:normal.forms} by a thorough investigation of the possible coordinate transformations between two metrics of a particular, parametric form as specified by each case in Theorem~\ref{thm:matveev}. Equivalently, we can state this as follows: We fix a projective class contained in Theorem~\ref{thm:matveev}, requiring degree of mobility~2. Thus, the metric can be taken in a certain parametric form, according to the cases A(I)--C(III) of Theorem~\ref{thm:matveev}.
The preoccupation of the current section is to find all coordinate transformations that mediate between two parametric representations of a given metric, i.e.\ to find the parameter configurations in each of the cases A(I)--C(III) that represent the same metric.

\begin{lemma}\label{la:different.cases}
 Metrics with degree of mobility~2 and 1-dimensional projective algebra cannot be projectively equivalent if they belong to different cases in Theorem~\ref{thm:matveev}.
\end{lemma}
\begin{proof}
 The cases in Theorem~\ref{thm:matveev} are organized according to
 (i) \emph{letters A--C,} i.e.\ the type of eigenvalues of the Benenti tensor, to be discussed in detail in Section~\ref{sec:strategy.benenti}. See also~\cite{bolsinov_2009} where Proposition~\ref{prop:dini} is proven.
 (ii) \emph{Roman numbers I--III:} type of eigenvalues of~$\lie_w$ where, as always, $w$ is the projective vector field (see~\cite{bolsinov_2009}).
 More precisely, the labels mean: 2 real eigenvalues (labels A resp.\ I), 2 complex conjugate eigenvalues (B resp.\ III) and 1 real eigenvalue with algebraic multiplicity~2 (C resp.\ II).
 As one easily confirms, these properties are characteristic to the projective class, i.e.\ invariant under projective transformations. This proves the statement.
\end{proof}

\noindent The current section presents the techniques and strategies applied later, in Sections~\ref{sec:proof.normal.forms} and~\ref{sec:proof.degree.of.mobility.3}.
It is organized as follows:
First we discuss the flow of the projective vector field in the space~$\solSp$ of solutions to~\eqref{eqn:linear.system}. In fact, metrics along its orbits are isometric, as the exponential map of~$w$ provides an explicit coordinate transformation.
Second, we discuss how this coordinate transformations act on the length of the projective vector field~$w$, viewed mainly from the viewpoint of explicit coordinate expressions.
Next, Section~\ref{sec:strategy.benenti} introduces Benenti tensors and discusses how their properties can be exploited for our purposes. Two explicit examples are computed for later reference in Section~\ref{sec:proof.normal.forms}.
Finally, in Section~\ref{sec:reparametrisation} we discuss an alternative parametrisation of metrics in a projective class for which $\lie_w$ does not admit real eigenvalues.

For the remainder of this section, let us fix a system of coordinates $(x,y)$ as in Theorem~\ref{thm:matveev}, and assume degree of mobility~2 (unless stated otherwise). The assumption of degree of mobility~2 holds if we exclude metrics mentioned in Theorem~\ref{thm:theorem.2.revised}. Also, we exclude metrics covered by Theorem~\ref{thm:theorem.3.revised} and metrics that correspond to eigenvectors of $\lie_w$, see Proposition~\ref{prop:homotheties.and.eigenvectors}, as they do not comply with the requirements we are interested in.

\subsubsection{Projective orbits}
Consider a specific projective class among A(I)--C(III) of Theorem~\ref{thm:matveev}, excluding metrics covered by Theorems~\ref{thm:theorem.2.revised} and~\ref{thm:theorem.3.revised}.
Metrics belonging to this class will be of the form~\eqref{eqn:general.metric} and admit, up to a non-vanishing constant factor, one and only one projective vector field: let us fix this vector field and denote it by~$w$, while we denote its local flow by $\phi_t$.
\begin{definition}[Projective orbits]
Let~$w$ be a projective vector field and $\phi_t$ be its local flow.
An orbit of $\phi_t$ in~$\solSp$ that passes through a point~$a$ (i.e.\ a Liouville tensor~\eqref{eqn:Liouville.tensor.g}) is called the \emph{projective orbit through $a$}.
In view of Remark~\ref{rmk:g-a-correspondence},
the term of projective orbit will refer either to the orbit in the space of Liouville tensors or in the space of metrics, depending on the context.
\end{definition}

\noindent The transformation $\phi_t$ sends a metric of the form \eqref{eqn:general.metric} into a metric of the same form. Thus, in view of~\eqref{eqn:space.G}, it induces a map $\mathds{R}^m\setminus\{\textrm{eigenspaces of }\lie_w\}\to \mathds{R}^m\setminus\{\textrm{eigenspaces of }\lie_w\}$, $m=\dim \solSp\in\{2,3\}$. Note that eigenspaces of~$\lie_w$ need to be removed in view of Proposition~\ref{prop:homotheties.and.eigenvectors}. Due to~\eqref{eqn:space.G}, the eigenspaces are essentially axes in~$\mathds{R}^m$, and the parameters~$\K{i}$ are cartesian coordinates. Let us discuss how the projective vector field acts in this cartesian setting.
To this end, let $g\in\projcl(\g1,\dts,\g{m})$, $m\in\{2,3\}$, as in~\eqref{eqn:space.G}. Furthermore, let $a=\psi^{-1}(g)$ be its associated Liouville tensor~\eqref{eqn:Liouville.tensor.g}, and let $\a{i}=\psi^{-1}(\g{i})$ for each value of~$i$.
Thus, the metric is represented by a cartesian expression $a=\sum_i \K{i}\a{i}$ for some $K_i\in\mathds{R}$.
We have already seen in Section~\ref{sec:results.used} how the Lie derivative~$\lie_w$ acts on the space~$\solSp$ of Liouville tensors,
\begin{equation}\label{eqn:bo}
  \phi^*_t\left(\sum_{i=1}^m \K{i}\a{i}\right)
   =\sum_{i=1}^m \K{i}\,\exp(tw)^\ast\a{i}
   =\sum_{i=1}^m \K{i}'(t)\a{i}\,.
\end{equation}
Note that the pullback with the exponential mapping $\exp(tw)$ is represented, in the local basis $(\a{i})$, by a matrix that can be computed explicitly. Indeed, it is the matrix exponential of a matrix from~\eqref{eqn:normal-matrices}.
The projective orbit through a point $(\K1,\dots,\K{m})$ is thus visualised by a curve $(\K1'(t),\dots,\K{m}'(t))$.
Metrics on the same projective orbit are isometric, as they are linked by $\phi^*_t$.
A complication that can occur is that not all projective transformations are of type $\phi_t$ for some $t$, so that metrics belonging to different $\phi_t$-orbits could be isometric.
In order to see if this happens, we need additional techniques.

\subsubsection{Coordinate transformations as local isometries}\label{sec:poles.isometries}
Locally, any isometry is a coordinate transformation, i.e.\ a local diffeomorphism.
Thus, let $\tau$ be a local isometry between two metrics~$g',g$ with $g=\tau^*(g')$, with degree of mobility~2 or~3. Actually, since we work with explicit coordinate representations of a certain form, $\tau$ in fact is a coordinate transformation that preserves the parametrical form of the respective case of Theorem~\ref{thm:matveev}.
Recalling the action of~$\lie_w$, and Formula~\eqref{eqn:Lwa}, fix the projective vector field~$w$ by the condition
\begin{equation}\label{eqn:fix.proj.vector.field}
  \lie_w\begin{pmatrix} a\\ \pe{a} \end{pmatrix}
  = \textsf{M}\begin{pmatrix} a\\ \pe{a} \end{pmatrix},
\end{equation}
where $\textsf{M}$ is a $2\times2$-matrix as in~\eqref{eqn:normal-matrices}. This depends only on the directions of $a$ and $\pe{a}$ in the 2-dimensional invariant subspace~$\langle\a1,\a2\rangle=:\hat{\solSp}\subset\solSp$.
Since~$\tau$ preserves eigenspaces, condition~\eqref{eqn:fix.proj.vector.field} fixes $w$, in the sense that $\tau_*(w)=w$, keeping in mind the caveats from Section~\ref{sec:preparations}.
Denoting the local flow of~$w$ by~$\phi_t$, $t\to\tau(\phi_t(p_0))$ is an integral curve of~$w$ starting at $p'_0=\tau(p_0)$ for a point $p_0$ on the manifold.
Moreover, we have that
\begin{equation}\label{eqn:isometries.length.preservation}
g_{\phi_t(p_0)}(w_{\phi_t(p_0)},w_{\phi_t(p_0)})
=g'_{\tau(\phi_t(p_0))}(w_{\tau(\phi_t(p_0))},w_{\tau(\phi_t(p_0))})
=g'_{\phi_t(p'_0)}(w_{\phi_t(p'_0)},w_{\phi_t(p'_0)})\,.
\end{equation}
The poles and zeros of each side of Equation~\eqref{eqn:isometries.length.preservation}, regarded as functions in $t$, coincide for isometric metrics, entailing extra conditions on $K_i$ and $K_i'$.
This is going to be exploited in Section~\ref{sec:proof.normal.forms}.
In addition, we make use of the properties of Benenti tensors in order to identify possible coordinate transformations between metrics in Theorem~\ref{thm:matveev}.

\subsubsection{Benenti tensors and eigenspaces of the Lie derivative\texorpdfstring{ $\lie_w$}{}}\label{sec:strategy.benenti}
Benenti tensors are well-known objects that have successfully been employed in a number of fields, e.g., in the context of orthogonal separation coordinates, Killing-St\"ackel spaces, the proof of the Lichnerowicz-Obata conjecture, splitting-gluing constructions, and others. Here, we use Benenti tensors merely as a tool to facilitate the computations in a conceptually concise way, without going deeper into the theory of these interesting and useful objects, more on which can be found in the references given below.
For the purposes in this present paper, Benenti tensors will be defined as follows:
\begin{definition}\label{def:benenti.g}
 Let $(g,\pe{g})$ be a pair of projectively equivalent metrics. Then
 \begin{equation}\label{eqn:benenti.tensor}
   L(g,\pe{g}) := \left(\frac{\det(\pe{g})}{\det(g)}\right)^{\frac{1}{n+1}}\,\pe{g}^{-1}g
   = \pe{\sigma}\sigma^{-1} = \left|\frac{\det(a)}{\det(\pe{a})}\right|\,\pe{a}^{-1}a
 \end{equation}
 is called the \emph{Benenti tensor} of the pair $(g,\pe{g})$.
\end{definition}

\noindent Precisely speaking, Definition~\ref{def:benenti.g} is not the actual definition of a Benenti tensor as given initially by Benenti, see~\cite{benenti_1992}.
Rather, a Benenti tensor (for a metric $g$, not a pair of metrics) is a special conformal Killing tensor with certain additional properties, cf.~\cite{benenti_1992,ibort_2000,bolsinov_2003}.
For our context, the definitions are equivalent.
Readers interested in the precise definition and properties of Benenti tensors, may like to refer to~\cite{ibort_2000,crampin_2008,benenti_1992,bolsinov_2003}, as well as other references, for instance \cite{benenti_2016,benenti_2008,kiosak_2010,manno_2017ben,topalov_2003,blaszak_2005}.

\begin{remark}\label{rmk:eigenspaces.are.1d}
In dimension 2, any non-zero solution $a_{ij}$ of \eqref{eqn:linear.system} has, in almost every point, full rank ($a_{ij}$ viewed as a matrix), see~\cite{matveev_2012}. In higher dimension degenerate solutions are possible and linked to eigenvalues of Benenti tensors, see~\cite{manno_2017ben}.
If we consider a metric~$g$ with exactly one, essential projective vector field~$w$, the Lie derivative~$\lie_w$ on~$\solSp(g)$ has only $1$-dimensional eigenspaces (if any). For instance, assume~$g$ has degree of mobility 2, and an eigenvalue of algebraic multiplicity $2$. If its geometric multiplicity were $2$, the vector field~$w$ would be homothetic for all metrics in the projective class, including $g$, which would contradict the hypothesis. Thus, the geometric multiplicity is at most~$1$ (this is a general fact). Actually, it can be shown (regardless of the degree of mobility of $g$) that if $\lie_w$ admits a 2-dimensional eigenspace, it admits a non-trivial Killing vector, see Lemma~\ref{la:homothetic.subspace} and~\cite{matveev_2012} for details.
\end{remark}

\noindent The following lemma turns out to be a very useful tool for finding possible coordinate transformations between metrics in Theorem~\ref{thm:matveev}, see also the discussion in the beginning of Section~\ref{sec:techniques.strategies}.

\begin{lemma}\label{la:benenti.transformations}
Let $g$ be a 2-dimensional metric with exactly one projective vector field~$w$ (up to rescaling), and degree of mobility~2.
Recall the space of solutions to the metrizability equations~\eqref{eqn:linear.system}, which is denoted by~$\solSp=\solSp(g)$.
Let~$\pe{g}$ be a metric projectively equivalent, but non-proportional to~$g$, and let~$L$ be the Benenti tensor for the pair $(g,\pe{g})$, given by~\eqref{eqn:benenti.tensor}.
Assume that~$w$ is homothetic for~$g$. Then there is a local diffeomorphism~$\tau$ that maps the family
\begin{equation}\label{eqn:benenti.1r}
  \{ KL+c\Id\,,\  K,c\in\R\,,\ K\ne0 \}\,,
\end{equation}
onto itself.
In case~$w$ is homothetic for~$\pe{g}$, too, the local diffeomorphism~$\tau$ can be chosen such that it maps also the family
\begin{equation}\label{eqn:benenti.2r}
  \{ KL\,,\  K\ne0 \}\,,
\end{equation}
onto itself, i.e.\ $L$ is preserved up to multiplication by a non-zero constant.
\end{lemma}
\begin{proof}
Let~$a$ be an eigenvector of $\lie_w$ and $\tau^*:\solSp(g)\to\solSp(g)$ for a (local) isometry $\tau$. Then $\tau^*(a)=ka$ for $k\in\Rnz$ (where we think of $a$ as a matrix in explicit coordinates). At the same time, the solution $\pe{a}$ (corresponding to $\pe{g}$ via~\eqref{eqn:Liouville.tensor.g}) is mapped to another element of $\solSp$.
Thus, $\tau^*(\pe{a})=K_1a+K_2\pe{a}$ for some constants $\K1,\K2$. Altogether, we obtain
\[
 \tau^*(L(a,\pe{a})) = \frac{1}{k}L(a,K_1a+K_2\pe{a})
 = \frac{\K1}{k}\,\Id+\frac{\K2}{k}\,L(a,\pe{a})
\]
and thus follows the statement.
If~$\pe{a}$ also is an eigenvector of~$\lie_w$, then~$\K1=0$ and~$\tau^*(L)$ is a scalar multiple of~$L$.
\end{proof}

\noindent Broadly speaking, Lemma~\ref{la:benenti.transformations} allows us to obtain constraints to coordinate expressions of isometries, i.e.\ for the local isometry~$\tau:(x,y)\to (u(x,y),v(x,y))$.
In fact the condition $\tau^*(L)$ being in the span of $L$ and $\Id$ yields four differential equations on the components of the Jacobian matrix of $\tau$, providing restrictions on the possible coordinate transformations.
Below, we can obtain explicit solutions for $u(x,y)$ and $v(x,y)$ from these equations, however depending on some free constants. These constants can be constricted further, using explicit metrics corresponding to eigenvectors of~$\lie_w$, and the form-preserving requirement.
We illustrate this by two examples that are going to be referred to in later computations in Section~\ref{sec:proof.normal.forms}.

\paragraph{Example 1: Diagonal Benenti tensor.}
Let us assume the Benenti tensor $L$ has the diagonal form
\begin{equation}\label{eqn:Benenti.diagonal}
L = \begin{pmatrix} X & 0 \\ 0 & Y \end{pmatrix}\,,\quad X=X(x)\,,\,\, Y=Y(y)\,,
\end{equation}
w.r.t.\ coordinates $(x,y)$.
We discuss the possible coordinate transformations preserving the tensor \eqref{eqn:Benenti.diagonal} up to a constant conformal factor.
This addresses the case when we have two distinct real eigenvalues for $\lie_w$; the case when there is only one eigenvalue can be treated similarly, and we shall discuss this at the end of this subsection.
In local coordinates, we may write $L = X\,dx\otimes\partial_x +Y\,dy\otimes\partial_y$.
Take a change of coordinates $\tau:(x,y)\to(u(x,y),v(x,y))$ such that, in the new coordinates $(u,v)$, $L$ is still represented by a diagonal matrix.
More precisely,
\begin{equation}\label{eqn:new.diagonal.L}
\tau^*(L) = U\,du\otimes\partial_u +V\,dv\otimes\partial_v\,,
\end{equation}
with $U=U(u)$ and $V=V(v)$ depending on only one of the new coordinates $(u,v)$, respectively.
The off-diagonal entries of~\eqref{eqn:new.diagonal.L} entail the requirements
\begin{equation}\label{eqn:diagonal.trafo}
X\,x_u\,v_x+Y\,y_u\,v_y = 0\quad\text{and}\quad
X\,x_v\,u_x+Y\,y_v\,u_y = 0\,.
\end{equation}
Consider now the transformation into the new coordinates,
\[
\begin{pmatrix} du \\ dv \end{pmatrix} =
\begin{pmatrix} u_x & u_y \\ v_x & v_y \end{pmatrix}\,
\begin{pmatrix} dx \\ dy \end{pmatrix}\,.
\]
The inverse transformation is given by
\begin{equation}\label{eqn:derivatives.coeffs}
\begin{pmatrix} dx \\ dy \end{pmatrix} =
\frac{1}{u_xv_y-u_yv_x}\,
\begin{pmatrix} v_y & -u_y \\ -v_x & u_x \end{pmatrix}\,
\begin{pmatrix} du \\ dv \end{pmatrix} =
\begin{pmatrix} x_u & x_v \\ y_u & y_v \end{pmatrix}\,
\begin{pmatrix} du \\ dv \end{pmatrix}\,.
\end{equation}
Thus, Equations~\eqref{eqn:diagonal.trafo} become, respectively,
\[
(X-Y)\,v_x\,v_y = 0\quad\text{and}\quad
(Y-X)\,u_x\,u_y = 0\,,
\]
and therefore, locally, we have
\begin{equation}\label{eqn:2r.cases}
u=u(x), v=v(y)\quad\text{or}\quad
u=u(y), v=v(x)\,.
\end{equation}
Next, the entries on the diagonal of~\eqref{eqn:new.diagonal.L}, in matrix representation, entail the requirements
\begin{equation}\label{eqn:diagonal.entries.2r}
X\,x_uu_x+Y\,y_uu_y = K\,U\quad\text{and}\quad
X\,x_vv_x+Y\,y_vv_y = K\,V
\end{equation}
where $K\in\Rnz$ and where the $U$ and $V$ are analogs of $X$ and $Y$, respectively.
Combining this with \eqref{eqn:2r.cases}, we find that either
\begin{equation}\label{eqn:possible.transf.2r.diagonal.L}
 X=KU\,, Y=KV\,,\quad\text{or}\quad X=KV\,, Y=KU\,.
\end{equation}

\paragraph{Example 2.}
Let us now quickly look at the case when~$\lie_w$ admits only one real eigenvalue, i.e.\ the first case of Lemma~\ref{la:benenti.transformations}.
The Benenti tensor shall still have the form~\eqref{eqn:Benenti.diagonal}, so
Equation~\eqref{eqn:diagonal.trafo} remains true in this case. However,~\eqref{eqn:diagonal.entries.2r} must be replaced by
\begin{equation*}
X\,x_uu_x+Y\,y_uu_y = K\,U+c\quad\text{and}\quad
X\,x_vv_x+Y\,y_vv_y = K\,V+c
\end{equation*}
with a constant $c\in\R$, and thus we find
\begin{equation}\label{eqn:possible.transf.1r.diagonal.L}
  X=KU+c\,,\ Y=KV+c\,,\quad
  \text{or}\quad
  X=KV+c\,,\ Y=KU+c\,.
\end{equation}

\noindent Of course, Equations~\eqref{eqn:possible.transf.2r.diagonal.L} and~\eqref{eqn:possible.transf.1r.diagonal.L} are the result of a generic reasoning. The details of the computations are going to be discussed in Section~\ref{sec:proof.normal.forms}, where we apply the reasoning outlined above to prove Theorem~\ref{thm:normal.forms}.

\subsubsection{Parametrization of metrics within a given projective class}\label{sec:reparametrisation}
According to Formula~\eqref{eqn:general.metric}, and assuming degree of mobility~2, we can describe any metric~$g$ inside a projective class in terms of two generating metrics $\g1$, $\g2$. Alternatively, using its Liouville tensor~\eqref{eqn:Liouville.tensor.g}, we can characterize~$g$ as a linear combination of the Liouville tensors~$\a1$ and~$\a2$ (associated to $\g1,\g2$),
$a = \K1\a1+\K2\a2$.
Formally, this corresponds to Cartesian coordinates on the space $\langle\a1,\a2\rangle$. Of course, we are free to choose a different parametrisation. Particularly, in cases where $\lie_w$ does not admit a real eigenvalue, we exploit reparametrizations as follows:
let $\lie_w$ have the form~(III) of~\eqref{eqn:normal-matrices}, and let  $(\a1,\a2)$ be a basis of $\solSp$ in which $\lie_w$ assumes such a form. We have that
\begin{equation}\label{eqn:III.expLw}
\exp(t\lie_w)=e^{\lambda t}\,\begin{pmatrix} \cos(t) & -\sin(t)\\ \sin(t) & \quad\cos(t) \end{pmatrix}.
\end{equation}
Consider a general metric $g=\psi(\K1\a1+\K2\a2)$ from the projective class $\projcl(\g1,\g2)$.
By virtue of Equation~\eqref{eqn:III.expLw}, we can compute
\begin{equation}\label{eqn:III.pullback}
\phi_t^*(\K1\a1+\K2\a2)=e^{\lambda t}\left( \K1\cos(t)\a1-\K1\sin(t)\a2+\K2\sin(t)\a1+\K2\cos(t)\a2 \right).
\end{equation}
Thus, the orbits of the projective flow describe logarithmically spiraling curves in \solSp.

\begin{remark}\label{rmk:III.parameter.values}
For what follows, it will be convenient to reformulate the problem in terms of the alternative parameters $(\theta,K)$ in \solSp, defined by the relations
\begin{equation}\label{eqn:III.reparametrization}
 \K1=K e^{\lambda\theta} \sin(\theta)
 \quad\text{and}\quad
 \K2=K e^{\lambda\theta} \cos(\theta)\,.
\end{equation}
For practical purposes, it is advisable to consider the cases of vanishing and non-vanishing~$\lambda$ separately. If $\lambda=0$, the reparametrization~\eqref{eqn:III.reparametrization} leads to polar parameters $(K,\theta)$ with $K>0$ representing the radius and $\theta\in\halfangl$ being the angle parameter. The inverse transformation for~\eqref{eqn:III.reparametrization} reads, in this case,
\begin{equation}
 K=\sqrt{\K1^2+\K2^2}
 \quad\text{and}\quad
 \theta=\arctan\left(\frac{\K1}{\K2}\right)\,.
\end{equation}
On the other hand, if $\lambda>0$, we require $\theta\in\mathds{R}$. For fixed $K>0$, $\theta$ parametrizes a spiraling curve in $\solSp$.
The parameter~$K$ thus has to be restricted to values $K\in[1,e^{2\lambda\pi})$ in order to avoid ambiguity. We define $K=e^{\lambda\alpha}$ and interpret $\alpha$ as an inclination angle, $\alpha\in\angl$.
For $\lambda>0$, the inverse of~\eqref{eqn:III.reparametrization} is not straightforward, but we can construct $\alpha,\theta$ if we first convert $(\K1,\K2)$ into polar coordinates $(R(\K1,\K2),\varphi(\K1,\K2))$, where~$R$ is the radial and~$\varphi$ the angular parameter. Then we have
\[
 \alpha = \left( \frac{\ln(R)}{\lambda}-\varphi \right) \mod\,2\pi\,,
\]
and, setting
$m:=\frac{1}{2\pi}\,\left( \alpha-\frac{\ln(R)}{\lambda}+\varphi \right)$,
we obtain also $\theta = \varphi + 2\pi m$.
\end{remark}
\noindent The parametrization~\eqref{eqn:III.reparametrization} is adjusted to the problem in the sense that $K$ is invariant along orbits.

\begin{lemma}\label{la:projcurve_III}
	The pairs $(K,\theta)$ and $(K',\theta')$ are related by the transformation~\eqref{eqn:III.pullback} if and only if $K' = K$.
\end{lemma}
\begin{proof}
	Consider~\eqref{eqn:III.pullback} from which we infer
	\begin{equation*}
	  K e^{\lambda (t+\theta)} \left( \sin(\theta+t)\,\a1+\cos(\theta+t)\,\a2 \right)
	  =
	  K' e^{\lambda \theta'} \sin(\theta')\,\a1
	  +K' e^{\lambda \theta'} \cos(\theta')\,\a2 \,.
	\end{equation*}
	Since $(\a1,\a2)$ is a basis, we obtain $\tan(\theta')=\tan(\theta+t)$. If $\lambda=0$, this implies $K'=K$ directly. If $\lambda\ne0$, we also need to take into account that $K',K<e^{2\lambda\pi}$.
	The other implication is straightforward.

\end{proof}

\begin{figure}
  \begin{center}
	\scalebox{.45}{\input{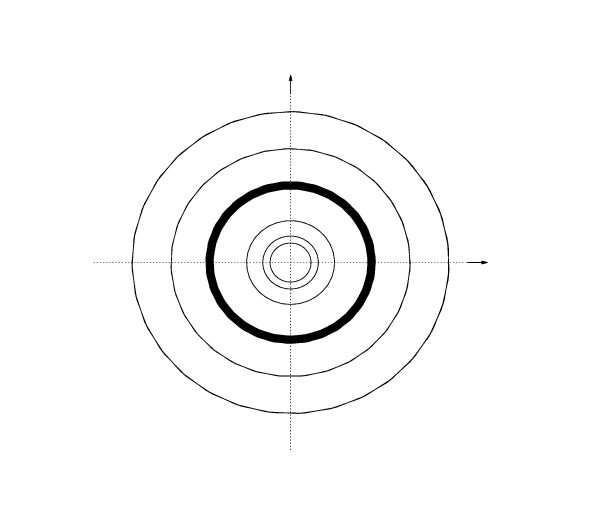}}
	\hspace{1cm}
	\scalebox{.45}{\input{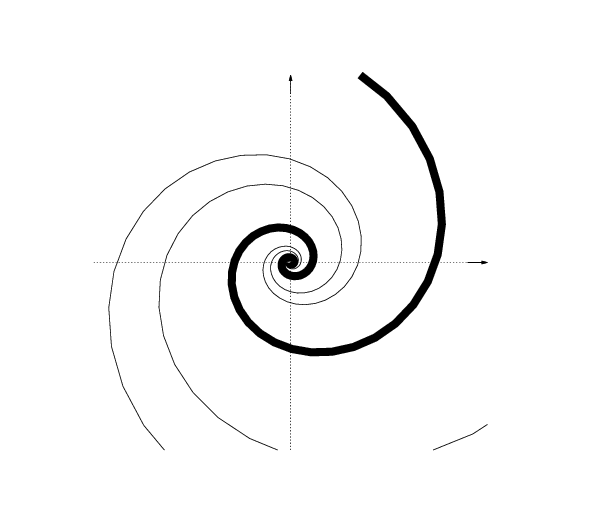}}
	\smallskip
  \end{center}
  \vspace{-1cm}
  \caption{The flow of the projective vector field in \solSp, sketched for different values of $(\K1,\K2)$ and different~$\lambda$.
  The left panel shows orbits for $\lambda=0$, the right panel for $\lambda>0$.
  See Lemma~\ref{la:projcurve_III} for details.}\label{fig:orbit.III}
\end{figure}

\subsection{Metrics with degree of mobility 3}\label{sec:proof.degree.of.mobility.3}
According to Proposition~\ref{prop:dom3.proj.equiv}, metrics with degree of mobility~3 can be transformed, locally, into a metric of the form~\eqref{eqn:general.metric} where the metrics $\g{i}$ are given by Equations~\eqref{eqn:supint.generators}.
The projective vector field, fixed by~\eqref{eqn:fix.proj.vector.field} and case~(I) of~\eqref{eqn:normal-matrices}, is \eqref{eqn:proj.vector.field.CIa}.
In the current section, we establish the normal forms for metrics within the projective class described by~\eqref{eqn:supint.generators} via~\eqref{eqn:space.G}.
We present two ways how this can be achieved. The first is a ``hands-on'' approach that provides us with explicit coordinate normal forms.
The second way is more conceptual and gives us a better understanding of the geometry of the projective class modulo isometries.
Before we proceed with the actual proof, let us make a few useful comments.
\begin{remark}\label{rmk:dom3.eigenvalues}
For the metrics $\g1$, $\g2$ and $\g3$ in~\eqref{eqn:supint.generators}, the projective vector field is homothetic, according to Propositions~\ref{prop:dom3.proj.equiv} and~\ref{prop:homotheties.and.eigenvectors}. The Lie derivatives w.r.t.~\eqref{eqn:proj.vector.field.CIa} are
\[
 \lie_w\g1 = 5\g1\,, \qquad
 \lie_w\g2 = 2\g2\,, \qquad
 \lie_w\g3 = -4\g3\,.
\]
Any isometry is a projective transformation, and moreover any isometry preserves eigenspaces of~$\lie_w$, as eigenspaces are geometric objects.
Therefore, for an isometry $\tau$ we have the identities $\tau^*\g{i}=\mu_i\g{i}$ with $\mu_i\in\Rnz$ for $i\in\{1,2,3\}$.
On the other hand, the metrics~\eqref{eqn:supint.generators} suffice, due to Formula~\eqref{eqn:general.metric}, to describe any metric of the projective class $\projcl(\g1,\g2,\g3)$, using the same system of coordinates $(x,y)$. In order to identify which metrics in $\projcl(\g1,\g2,\g3)$ can be transformed into one another, we therefore need to understand those coordinate transformations that preserve the form of the metric $g\in\projcl(\g1,\g2,\g3)$ obtained via~\eqref{eqn:general.metric}.
\end{remark}

\noindent According to Remark~\ref{rmk:dom3.eigenvalues}, the Lie derivative $\lie_w:\solSp\to\solSp$ admits three real, different eigenvalues.
This restricts heavily the freedom to perform coordinate transformations that preserve the form of the metric.
\begin{lemma}\label{la:dom3.form.preserving.isometry}
Any isometry that preserves the form of a metric $g\in\projcl(\g1,\g2,\g3)$ in the projective class~\eqref{eqn:supint.generators} is, in these coordinates, of the form
\begin{equation}\label{eqn:dom3.transformation}
 (x_\text{new},y_\text{new}) = (k^2x_\text{old},ky_\text{old})\quad
 \text{where $k\in\Rnz$}\,.
\end{equation}
\end{lemma}
\begin{proof}
Consider the metrics \g1\ and \g2\ as in~\eqref{eqn:supint.generators}, and a transformation $(x,y)=(x(u,v),y(u,v))$. For $\g1$, we obtain (at `!' impose the preservation of eigenspaces)
\begin{equation}\label{eqn:dom3.isometry.g1}
 \g1	= (x+y^2)\,(x_uy_u\,du^2+x_vy_v\,dv^2+(x_uy_v+x_vy_u)\,dudv)
		\stackrel{!}{=}\mu_1\,(u+v^2)\,dudv
\end{equation}
Thus, around almost every point, we have one of the following possibilities:
\[
  (i)\ x=x(u),\quad y=y(v)
  \qquad\text{or}\qquad
  (ii)\ x=x(v),\quad y=y(u)\,,
\]
so~\eqref{eqn:dom3.isometry.g1} gives
\begin{equation}\label{eqn:dom3.isometry.cond.1}
  \mu_1\,(u+v^2) = (x+y^2)\,(x_uy_v+x_vy_u).
\end{equation}
Now, do an analogous computation for \g2.
Case (ii) turns out to be incompatible as it would replace the $dy^2$-term by a term in $du^2$ instead of $dv^2$.
For alternative (i), we have
\begin{align}
  \g2 &= -2\,\frac{x+y^2}{y^3}\,x_uy_v\,dudv +\frac{(x+y^2)^2}{y^4}\,y_v^2\,dv^2 \nonumber \\
  &\stackrel{\eqref{eqn:dom3.isometry.cond.1}}{=}
   -2\mu_1\,\frac{u+v^2}{y^3}\,dudv +\frac{\mu_1^2}{x_u^2}\,\frac{(u+v^2)^2}{y^4}\,dv^2
  \stackrel{!}{=}
	  -2\mu_2\,\frac{u+v^2}{v^3}\,dudv +\mu_2\,\frac{(u+v^2)^2}{v^4}\,dv^2.
  \label{eqn:dom3.isometry.cond.2}
\end{align}
The coefficients of, respectively, $dudv$ and $dv^2$ in \eqref{eqn:dom3.isometry.cond.2} yield the equations
\[
  \frac{\mu_1}{y^3} = \frac{\mu_2}{v^3}
  \qquad\text{and}\qquad
  \frac{\mu_1^2}{y^4\,x_u^2} = \frac{\mu_2}{v^4}\,,
\]
from which we can deduce
\begin{equation}\label{eqn:dom3.isometry.cond.3}
  y=\sqrt[3]{ \frac{\mu_1}{\mu_2} }\,v=:\mu\,v
  \qquad\text{and}\qquad
  x_u^2=\frac{\mu_1^2}{\mu_2}\,\sqrt[3]{\frac{\mu_2}{\mu_1}}^4 = \mu_1^{\nicefrac23}\sqrt[3]{\mu_2}=:\eta\,,
\end{equation}
where we let the root take account of the sign and where $\mu,\eta\in\Rnz$ are constants.
In view of~\eqref{eqn:dom3.isometry.cond.2}, we can integrate the second equation of~\eqref{eqn:dom3.isometry.cond.3} and obtain $x=\eta\,u$. Note that $\mu$ and $\eta$ are not independent, but, since their definition is rather complicated, we shall continue with the investigation of the transformation $(x,y)=(\eta u,\mu v)$ without prior study of the relation of $\mu$ and $\eta$. Apply this coordinate transformation to~$\g3$ of~\eqref{eqn:supint.generators}. The result again has to be in the same eigenspace. Thus,
\begin{multline*}
  \g3
  = \frac{\mu^2 v^2+\eta u}{(3\eta u-\mu^2 v^2)^6} \left(
	  9\,(\mu^2 v^2+\eta u)\,\eta^2\,du^2
	  -4\mu\,y\,(9\eta u+\mu^2 v^2)\,\eta\mu\,dudv
	  +12\eta u\,(\mu^2 v^2+\eta u)\,\mu^2\,dv^2
  \right) \\
  \stackrel{!}{=}
  \mu_3\,\frac{v^2+u}{(3u-v^2)^6} \left(
	  9\,(v^2+u)\,du^2
	  -4y\,(9u+v^2)\,dudv
	  +12u\,(v^2+u)\,dv^2
  \right)\,.
\end{multline*}
The coefficients w.r.t.\ $du^2$ and $dv^2$ of this equation give two equations from which we may infer
\[
  \eta^2\,\frac{(\mu^2 v^2+\eta u)^2}{(3\eta u-\mu^2 v^2)^6}
  = \mu_3\frac{(v^2+u)^2}{(3u-v^2)^6}
  = \eta\mu^2\,\frac{(\mu^2 v^2+\eta u)^2}{(3\eta u-\mu^2 v^2)^6}\,.
\]
Therefore, we have verified $\eta=\mu^2$ and thus, in order to preserve the form of the metric, it is necessary that the coordinate transformation takes the form $(x,y)=(\mu^2u,\mu v)$.
\end{proof}

\begin{remark}\label{rmk:normalization.superintegrable.case}
Consider an arbitrary metric $g\in\projcl(\g1,\g2,\g3)$.
Due to Formula~\eqref{eqn:general.metric}, if we apply $(x,y)\to(k^2x,ky)$,
\[
 g=\frac{
	\sum_{i=1}^3 \K{i}k^{-\frac{\alpha_i}{3}}\,\frac{\g{i}}{\det(\g{i})^{\nicefrac23}}
 }{
	\left[\det\big(
	\sum_{i=1}^3 \K{i}k^{-\frac{\alpha_i}{3}}\,\frac{\g{i}}{\det(\g{i})^{\nicefrac23}}
	\big)\right]^2}\,,
\]
where $(\alpha_1,\alpha_2,\alpha_3)=(5,2,-4)$ according to Remark~\ref{rmk:dom3.eigenvalues}.
Thus, by a suitable choice of~$k$, we may normalize the absolute value of one coefficient among $K_1$, $K_2$, $K_3$, leaving only its sign as a free parameter.
\end{remark}

\subsubsection{Way 1: Normal forms as explicit coordinate expressions}
Our starting point is Remark~\ref{rmk:normalization.superintegrable.case}.
Due to Propositions~\ref{prop:dom3.proj.equiv} and~\ref{prop:homotheties.and.eigenvectors}, the metric can be taken to be in the form C(Ia) of Theorem~\ref{thm:matveev}.
Generically, $\K{i}\ne0$ for $i\in\{1,2,3\}$, but we need to admit the possibility~$K_j=0$ for one~$j\in\{1,2,3\}$ (when two of the $K_i$ are zero, the metric is proportional to one of~\eqref{eqn:supint.generators} and thus admits a homothety).

\begin{lemma}\label{la:CIa.K3.zero}
If $\K3=0$, $\K1,\K2\ne0$, then a metric~$g\in\projcl(\g1,\g2,\g3)$ is isometric to
\begin{equation}\label{eqn:superint.K3zero.2}
	 g[\kappa,\varrho]
	 =\kappa\,\left(-2\frac{y^2+x}{(y-\varrho)^3}\,dxdy+\frac{(y^2+x)^2}{(y-\varrho)^4}\,dy^2\right)
\end{equation}
where $\kappa\in\Rnz$ and $\varrho\in\pmo$.
There is no coordinate transformation between two different copies~$g=g[\kappa,\varrho]$ and $g'=g[\kappa',\varrho']$ of~\eqref{eqn:superint.K3zero.2}.
\end{lemma}
\begin{proof}
Equation~\eqref{eqn:superint.K3zero.2} is obtained by a suitable choice of~$\K1$ according to Remark~\ref{rmk:normalization.superintegrable.case}.
In order to prove that two different copies of~\eqref{eqn:superint.K3zero.2} are non-isometric, recall that the flow of the projective vector field~\eqref{eqn:proj.vector.field.CIa} is $\phi_s(x_0,y_0)=(s^2x_0,sy_0)$ with $s=e^t>0$.
Consider Equation~\eqref{eqn:isometries.length.preservation}, i.e.
\begin{equation}\label{eqn:CIa.lengths.equation}
  \kappa\,\left(-2\frac{y_0^2+x_0}{(s\,y_0-\varrho)^3}\,s^5\,x_0y_0
  +\frac{s^6\,(y_0^2+x_0)^2}{(s\,y_0-\varrho)^4}\,y_0^2\right)
  =
  \kappa'\,\left(-2\frac{{y'_0}^2+x'_0}{(s\,y'_0-\varrho')^3}\,s^5\,x'_0y'_0
  +\frac{s^6\,({y'_0}^2+x'_0)^2}{(s\,y'_0-\varrho')^4}\,{y'_0}^2\right)\,.
\end{equation}
In view of Section~\ref{sec:poles.isometries}, we compare the poles of either side and find $\varrho y_0=\varrho' y_0'$.
Resubstituting this into~\eqref{eqn:CIa.lengths.equation}, we obtain a polynomial equation in~$s$,
\begin{equation}\label{eqn:CIa.K3.zero.nonisometry.polynomial.cond}
 \big[ (\kappa-\kappa')y_0^4+(\kappa'{x'_0}^2-\kappa x_0^2)\big]\,y_0^2\,s^6
  +2\varrho\,y_0\,\big[  (\kappa x_0-\kappa'x'_0)y_0^2+(\kappa x_0^2-\kappa'{x'_0}^2) \big]\,s^5
  =0\,,
\end{equation}
whose coefficients yield a system of polynomial equations.
For generic values of $y_0$, it implies $\kappa'=\kappa$ and $x_0'=x_0$.
Substituting these relations back into~\eqref{eqn:CIa.lengths.equation}, we have
\[
  -2\frac{y_0^2+x_0}{(s\,y_0-\varrho)^3}\,x_0y_0
  +\frac{s\,(y_0^2+x_0)^2}{(s\,y_0-\varrho)^4}\,y_0^2
  =
  -2\frac{\varrho'\varrho\,y_0^2+x_0}{(s\,\varrho'\varrho\,y_0-\varrho')^3}\,
			  \varrho'\varrho\,x_0y_0
  +\frac{s\,(\varrho'\varrho\,y_0^2+x'_0)^2}{(s\,\varrho'\varrho\,y_0-\varrho')^4}\,\varrho'\varrho\,y_0^2\,,
\]
from which we finally infer $\varrho'=\varrho$.
\end{proof}

\begin{lemma}\label{la:CIa.K2.zero}
If $\K2=0$, $\K1,\K3\ne0$, then a metric $g\in\projcl(\g1,\g2,\g3)$ is isometric to
\begin{equation}\label{eqn:superint.K2zero}
g[\kappa,\varrho]
=\frac{\kappa}{F(\varrho;x,y)^2}\,\left(
9(y^2+x)^2\,dx^2
-2(y^3+9xy-2\varrho)(y^2+x)\,dxdy
+12x(y^2+x)^2\,dy^2
\right)\,.
\end{equation}
with $\varrho\in\pmo$ and $\kappa\in\Rnz$ and
$F(c;x,y) =
y^6-9xy^4+27x^2y^2-27x^3+4c^2-(36xy+4y^3)\,c$.
There is no coordinate transformation between two different copies~$g=g[\kappa,\varrho]$ and $g'=g[\kappa',\varrho']$ of metrics~\eqref{eqn:superint.K2zero}.
\end{lemma}
\begin{proof}
From Equation~\eqref{eqn:general.metric}, we infer
\begin{equation*}
  g=\frac{\K3^{-3}}{F\left(\frac{\K1}{\K3};x,y\right)^2}\,\left(
  9(y^2+x)^2\,dx^2
  -2\left(y^3+9xy-2\frac{\K1}{\K3}\right)(y^2+x)\,dxdy
  +12x)(y^2+x)^2\,dy^2
  \right)\,,
\end{equation*}
with
$
F(c;x,y) =
y^6-9xy^4+27x^2y^2-27x^3+4c^2-(36xy+4y^3)\,c
$.
In view of Remark~\ref{rmk:normalization.superintegrable.case}, a change of coordinates $(x,y)\to(k^2x,ky)$ with suitable $k$ then yields~\eqref{eqn:superint.K2zero}.
We continue with an argument by contradiction. Thus, assume there is a (non-trivial) coordinate transformation~$\tau$ between two metrics~\eqref{eqn:superint.K2zero}. It must, by virtue of Lemma~\ref{la:dom3.form.preserving.isometry}, be of the form~\eqref{eqn:dom3.transformation}.
Then $\tau^*(g)=g'$ is an equation polynomial in $dx,dy,x,y$ and therefore implies a system of algebraic equations on $\kappa,\varrho,\kappa',\varrho'$ and $k$ only.
Solving this system under the assumption $k\ne0$, we find that the only real solution is~$k=1$ and $\kappa'=\kappa$, $\varrho'=\varrho$. This is the identity, and thus $\tau$ is trivial.
This concludes the proof.
\end{proof}

\begin{lemma}\label{la:CIa.K2.K3.nonzero}
If $\K2,\K3\ne0$, $\K1\in\R$, then a metric $g\in\projcl(\g1,\g2,\g3)$ is isometric to
\begin{equation}\label{eqn:superint.K2K3nonzero}
  g[\kappa;c,\varrho]=\frac{\kappa}{F(\varrho,c;x,y)^2}\,\left(
      9(y^2+x)^2\,dx^2
      -2(y^3+2\varrho\,y+9xy-2c)(y^2+x)\,dxdy
      +4(3x+\varrho)(y^2+x)^2\,dy^2
    \right)
\end{equation}
with $\varrho=\sgn(K_2)\in\pmo$, $\kappa\in\Rnz$, $c\in\mathds{R}$, and where the function $F$ is defined by
\begin{equation}\label{eqn:F}
 F(\zeta,c;x,y) =
 y^6-9xy^4+27x^2y^2-27x^3+4c^2-(36xy+4y^3)\,c
 +(18xy^2-5y^4-9x^2-8cy)\,\zeta+4y^2\,\zeta^2\,.
\end{equation}
There is no coordinate transformation between two different copies~$g=g[\kappa,c,\varrho]$ and $g'=g[\kappa',c',\varrho']$ of~\eqref{eqn:superint.K2K3nonzero}.
\end{lemma}
\begin{proof}
Factor out~$K_3$ in~\eqref{eqn:general.metric}. Then, recalling Remark~\ref{rmk:normalization.superintegrable.case}, perform a change of coordinates $(x,y)\to(k^2x,ky)$
with $k=|\frac{\K2}{\K3}|^{\nicefrac{3}{14}}\ne0$. This yields~\eqref{eqn:superint.K2K3nonzero}.
If there were a non-trivial coordinate transformation~$\tau$ between two such metrics, it must, by virtue of Lemma~\ref{la:dom3.form.preserving.isometry}, be of the form~\eqref{eqn:dom3.transformation}.
Then $\tau^*(g)=g'$ is an equation polynomial in $dx,dy,x,y$ and therefore implies a system of algebraic equations on $\kappa,c,\varrho,\kappa',c',\varrho'$ and $k$ only.
Under the assumption $k\ne0$, the only real solution of this system is $(\kappa',c',\varrho')=(\kappa,c,\varrho)$, $k=1$, a contradiction.
This concludes the proof.
\end{proof}

\noindent Finally, we have to make sure that two metrics of different form~\eqref{eqn:superint.K3zero.2}, \eqref{eqn:superint.K2zero} or~\eqref{eqn:superint.K2K3nonzero} cannot be transformed into one another.

\begin{lemma}\label{la:CIa.no.cross.isometries}
There is no coordinate transformation that can map a metric of the form~\eqref{eqn:superint.K3zero.2}, \eqref{eqn:superint.K2zero} or~\eqref{eqn:superint.K2K3nonzero} into a metric of another of these forms.
\end{lemma}
\begin{proof}
Let $a=\sum\K{i}\a{i}$ and $\pe{a}=\sum\K{i}'\a{i}$ be the Liouville tensors of a pair of metrics $g,\pe{g}$, as in the statement of Lemma~\ref{la:CIa.no.cross.isometries}.
In view of Lemma~\ref{la:dom3.form.preserving.isometry} and Remark~\ref{rmk:normalization.superintegrable.case}, isometries between $g$ and $\pe{g}$ rescale the coefficients $\K1,\K2,\K3$ according to
$\K{i}' = \K{i}\,k^{-\nicefrac{\alpha_i}{3}}$ ($i\in\{1,2,3\}$), with $k\in\Rnz$, $(\alpha_1,\alpha_2,\alpha_3)=(5,2,-4)$.
Therefore, if $\K2=0$, then $\K2'=0$, and if $\K3=0$, then $\K3'=0$. If $\K2,\K3\ne0$, then $\K2',\K3'\ne0$. So, $g$ and $\pe{g}$ cannot be transformed into one another.
\end{proof}

\noindent Combining Lemmas~\ref{la:CIa.K3.zero}, \ref{la:CIa.K2.zero}, \ref{la:CIa.K2.K3.nonzero} and \ref{la:CIa.no.cross.isometries}, we have found the following:

\begin{proposition}\label{prop:normal.forms.CIa}
A metric with degree of mobility~3 that admits exactly one, essential projective vector field can locally, by a coordinate transformation, be mapped onto one and only one of the following metrics:
\begin{enumerate}[label={(\roman*)}]
	\item $g=\kappa\,\left(-2\frac{y^2+x}{(y-\varrho)^3}\,dxdy+\frac{(y^2+x)^2}{(y-\varrho)^4}\,dy^2\right)$\,,

	\item
	$g=\frac{\kappa}{F(0,\varrho;x,y)^2}\,\left(
	9(y^2+x)^2\,dx^2
	-2(y^3+9xy-2\varrho)(y^2+x)\,dxdy
	+12x)(y^2+x)^2\,dy^2
	\right)$\,,

	\item
	$g=\frac{\kappa}{F(\varrho,c;x,y)^2}\,\left(
	9(y^2+x)^2\,dx^2
	-2(y^3+2\varrho\,y+9xy-2c)(y^2+x)\,dxdy
	+4(\varrho+3x)(y^2+x)^2\,dy^2
	\right)$\,,
\end{enumerate}
where $\kappa\in\Rnz$, $c\in\mathds{R}$, $\varrho\in\pmo$, and where $F(\zeta,c;x,y)$ is given by~\eqref{eqn:F}.
\end{proposition}

\subsubsection{Way 2: Parametrisation by points on the unit sphere}
The solution that we found, though it is a perfectly valid answer to the considered question, may not be entirely satisfactory, as we obtain several cases instead of a concise formula. Moreover, the cases do not respect the fact that $\g1,\g2,\g3$ should be considered on an equal level: they can be interchanged since our ordering of the eigenvalues is, basically, arbitrary.
We now show how to uniformly describe the metrics of the projective class (Ia) of Theorem~\ref{thm:matveev}.
These are metrics $g$ with degree of mobility~3 that admit exactly one, essential projective vector field.
Applying the transformation~\eqref{eqn:dom3.transformation} to
\[
 a=\sum_{i=1}^3 \K{i}k^{-\frac{\alpha_i}{3}}\,\a{i}\,,
\]
with $(\alpha_1,\alpha_2,\alpha_3)=(5,2,-4)$.
If we choose~$k$ in a suitable way, we achieve $|K_i|=1$ for one value of~$i$, i.e.\ one among the coefficients $K_1$, $K_2$ or $K_3$ can be normalized.
Let us consider the action of~$\lie_w$ on~$\solSp$, i.e.\ $\lie_wa_i=-\tfrac{\alpha_i}{3}\,a_i$ for $i\in\{1,2,3\}$.
We conclude that
\begin{equation}\label{eqn:proj.orbit.supint}
\phi_t^*\left(\sum\K{i}\a{i}\right) = e^{-\frac53t}\K1\a1+e^{-\frac23t}\K2\a2+e^{\frac43t}\K3\a3\,,
\end{equation}
where $\phi_t$ is the local flow of $w$.
Next, using ellipsoidal coordinates, we reparametrize
\[
\begin{pmatrix} \K1 \\ \K2 \\ \K3 \end{pmatrix}
= \begin{pmatrix}
e^{-\frac53r}\,\sin\theta\,\cos\varphi \\
e^{-\frac23r}\,\sin\theta\,\sin\varphi \\
e^{\frac43r}\,\cos\theta
\end{pmatrix}\,.
\]
Recall that points along the orbits~\eqref{eqn:proj.orbit.supint} correspond to metrics that are linked by coordinate transformations, and thus we may choose representatives for these metrics on the unit sphere, by letting $t=-r$,
\begin{equation}\label{eqn:normal.forms.supint}
a[\theta,\varphi] = \sin\theta\,\cos\varphi\,\a1
+\sin\theta\,\sin\varphi\,\a2
+\cos\theta\,\a3\,.
\end{equation}
Now, it suffices to consider coordinate transformations between such representatives only.
We complete this in two steps. First, apply~\eqref{eqn:dom3.transformation} to an element $a=\sum_{i=1}^3 \K{i}\a{i}$ on the unit sphere, i.e.~\eqref{eqn:normal.forms.supint},
and obtain
$\tau^*(a)=\sum_{i=1}^3 \K{i}k^{-\frac{\alpha_i}{3}}\,\a{i}$.
Generically, this result will not be on the unit sphere, so before we continue, let us replace~$\tau^*(a)$ by its representative on the unit sphere.
In fact, using~\eqref{eqn:proj.orbit.supint},
\[
 \phi_t^*(\tau^*(a))=\sum_{i=1}^3 e^{-\frac{\alpha_i}{3}\,t}\,\K{i}k^{-\frac{\alpha_i}{3}}\,\a{i}\,,
\]
and this lies on the unit sphere iff
$e^{-\frac{\alpha_i}{3}\,t}\,k^{-\frac{\alpha_i}{3}} = 1$, $i\in\{1,2,3\}$.
Therefore, we impose $t=-\ln(k)$, in order that $\phi_t^*(\tau^*(a))$ lies on the unit sphere.
As a result, $\phi_{t}^*(\tau^*(a))=a$.
This confirms that, indeed, the normal forms~\eqref{eqn:normal.forms.supint} cannot be mapped into one another by any coordinate transformation.

\subsection{The mutually non-diffeomorphic normal forms}\label{sec:proof.normal.forms}
For each of the metrics of Theorem~\ref{thm:matveev} with degree of mobility~2, we determine the freedom of applying coordinate transformations in terms of restrictions to the parameters.
Making use of this freedom, we arrive at the normal forms of Theorem~\ref{thm:normal.forms}.
In view of Lemma~\ref{la:different.cases}, we need only consider coordinate transformations between metrics of the same type in Theorem~\ref{thm:matveev}. Partially, we have already addressed this freedom in Section~\ref{sec:preparations}.
Now, we investigate it in more detail.

\begin{remark}\label{rmk:general.reasoning.cases}
In the current section, we exclusively consider cases for which the degree of mobility is exactly~2, i.e.\ the inclusion $\hat{\solSp}\subseteq\solSp$ of the 2-dimensional invariant subspace actually is an equality.
Silently, we incorporate the parameter restrictions obtained in Theorems~\ref{thm:theorem.2.revised} and~\ref{thm:theorem.3.revised}, as well as in Section~\ref{sec:preparations}.
\end{remark}

\subsubsection{Normal forms (A.1) of Theorem~\ref{thm:normal.forms}}\label{sec:A.1}
Consider metrics A(I) of Theorem~\ref{thm:matveev}, for which $\lie_w:\solSp\to\solSp$ has two real (distinct) eigenvalues. By Formula~\eqref{eqn:general.metric} any metric of this type is of the form (recall that in what follows we always assume degree of mobility~2)
\begin{equation}\label{eqn:A.I.generic.metric}
 g[\xi,h,\varepsilon;\K1,\K2]
 = \frac{x^\xi-hy^\xi}{(\K1+\K2hy^\xi)(\K1+\K2x^\xi)^2}\,dx^2
  +\varepsilon\,\frac{x^\xi-hy^\xi}{(\K1+\K2hy^\xi)^2(\K1+\K2x^\xi)}\,dy^2\,.
\end{equation}
\begin{align*}
  \varepsilon=\pm1\,,\quad
  |h|\geq1\,,\quad
  \xi\in(0,4]\,,\qquad\qquad
  \text{but $\xi\ne1$ and:}\quad
  & \text{if $\xi=2$, then $h\ne-\varepsilon$}, \\
  & \text{if $\xi=3$, then $(h,\varepsilon)\ne(\pm1,-1)$}, \\
  & \text{if $\xi=4$, then $h\ne1$}\,.
\end{align*}

\begin{remark}
The general form~\eqref{eqn:A.I.generic.metric} of metrics of type A(I) is obtained using Formula~\eqref{eqn:general.metric}, cf.~\cite{matveev_2012}. The~$\g{i}$ are, by grace of Theorem~\ref{thm:matveev}, given in the form
\begin{equation}\label{eqn:Liouville.A.1.base.metrics}
\g1=(e^{\xi x}-he^{\xi y})(e^{2x}\,dx^2+\varepsilon\,e^{2y}\,dy^2)\,,\quad
\g2=\left(e^{-\xi x}-\frac{1}{h}\,e^{-\xi y}\right)
\left(e^{(2-\xi)\,x}\,dx^2+\varepsilon\,e^{(2-\xi)\,y}\,dy^2\right)\,,
\end{equation}
but we need to take into account also the restrictions put forth in Theorems~\ref{thm:theorem.2.revised} and~\ref{thm:theorem.3.revised} as well as in Proposition~\ref{prop:homotheties.and.eigenvectors} and Lemma~\ref{la:lambda.1.homothetic}.
Note also that the $g_{i}$ correspond, via~\eqref{eqn:Liouville.tensor.g}, to non-proportional eigenvectors of~$\lie_w$.
Moreover, in Equation~\eqref{eqn:Liouville.A.1.base.metrics} and case~A(I) of Theorem~\ref{thm:normal.forms} we use coordinates different from those in~\eqref{eqn:A.I.generic.metric}, and, for consistency, $x>0$ and $y>0$ must be required in~\eqref{eqn:A.I.generic.metric}. This is easily seen from the explicit transformation $x_\text{new}=e^{\xi\,x_\text{old}}$, $y_\text{new}=e^{\xi\,y_\text{old}}$.
The projective vector field~$w$, in the new coordinates and after rescaling, is given by $w=x\partial_x+y\partial_y$ and its flow is $\phi_t(x_0,y_0)=(x_0e^{t},y_0e^{t})$ where $x_0>0$, $y_0>0$.
\end{remark}

\paragraph{Benenti tensors and candidates for isometries.}
The Benenti tensor~\eqref{eqn:benenti.tensor} of the metrics
\begin{equation}\label{eqn:Liouville.A.1.base.metrics.g1.g2}
\g1=g[\xi,h,\varepsilon;1,0]=(x^{\xi}-hy^{\xi})(dx^2+\varepsilon\,dy^2)\,,
\quad
\g2=g[\xi,h,\varepsilon;0,1]=\left(\frac{1}{x^{\xi}}-\frac{1}{hy^{\xi}}\right)
\left(\frac{dx^2}{x^{\xi}}+\varepsilon\,\frac{dy^2}{y^{\xi}}\right)
\end{equation}
is, up to a constant conformal factor,
\begin{equation}\label{eqn:A.I.benenti.1}
L(\g2,\g1) = x^\xi\,dx\otimes\partial_x + hy^\xi\,dy\otimes\partial_y\,.
\end{equation}
Consider a second pair of metrics, in coordinates $(u,v)$, $\hat{g}_1=g[\eta,k,\zeta;1,0]$ and $\hat{g}_2=g[\eta,k,\zeta;0,1]$. Their Benenti tensor is
\begin{equation}\label{eqn:A.I.benenti.2}
 L(\hat{g}_2,\hat{g}_1) = u^\eta\,du\otimes\partial_u + kv^\eta\,dv\otimes\partial_v\,.
\end{equation}
As the signature of the metrics is given by $\varepsilon$ and $\zeta$, respectively, and isometries preserve the signature, we have that $\zeta=\varepsilon$.
Moreover, thanks to Lemma~\ref{la:lambda.different.classes}, we obtain $\eta=\xi$.
The Benenti tensors~\eqref{eqn:A.I.benenti.1} and~\eqref{eqn:A.I.benenti.2} are of the form discussed in Section~\ref{sec:strategy.benenti}, and so we can use the results obtained there, particularly Equation~\eqref{eqn:possible.transf.2r.diagonal.L}, which leads us to consider the following transformations:
\begin{equation}\label{eqn:A.I.candidate.transformations}
 (x,y)\to(Kh^{\nicefrac{1}{\xi}}x,Kk^{\nicefrac{1}{\xi}}y)\,\quad\text{and}\quad
 (x,y)\to(K(kh)^{\nicefrac{1}{\xi}}y,Kx)\,,\quad\text{where } K>0\,.
\end{equation}

\begin{remark}
The transformations~\eqref{eqn:A.I.candidate.transformations} are obtained under the assumption that isometries map $g_{i}\to\hat{g}_{i}$.
This assumption is, see Lemma~\ref{la:lambda.different.classes}, justified as long as $\lambda\ne-1$. If however $\lambda=-1$, i.e.\ $\xi=4$, the eigenvalues of~$\lie_w$ have the same absolute value, which makes the eigenspaces indistinguishable. We therefore need to take into account the transformations $g_{1}\to\hat{g}_{2}$ and $g_{2}\to\hat{g}_{1}$, too.
Formally this means we have to replace~\eqref{eqn:A.I.benenti.2} by its inverse, which gives two additional transformations
\begin{equation}\label{eqn:A.I.candidate.transformations.special}
 (x,y) = \left( \frac{K}{u},\frac{K}{(kh)^{\nicefrac14}\,v} \right)
 \quad\text{and, respectively,}\quad
 (x,y) = \left( \frac{K}{|k|^{\nicefrac14}\,v},\frac{K}{|h|^{\nicefrac14}\,u} \right)\,,\quad
 \text{with $K>0$}\,.
\end{equation}
\end{remark}

\noindent For ease of notation, we extend our previous convention:
An expression $K^{\xi}$ for real $K$ and with real, but unspecified exponent shall be understood as having an absolute value, $|K|^{\xi}$.

\paragraph{Eigenspaces and remaining freedom.}
Any diffeomorphism must preserve the eigenspaces of~$\lie_w$, and we now use this freedom to produce restrictions on the transformations~\eqref{eqn:A.I.candidate.transformations} and~\eqref{eqn:A.I.candidate.transformations.special}.
Let us discuss the first case of~\eqref{eqn:A.I.candidate.transformations},
\[
  g_1 \to K^{\xi+2}\,h^{1+\frac{2}{\xi}}\,(x^\xi-\sgn(h)|k|\,y^\xi)\,\left(dx^2+\varepsilon \left(\frac{k}{h}\right)^{\frac{2}{\xi}} dy^2\right)\,,
  \quad
  \text{compare~\eqref{eqn:Liouville.A.1.base.metrics.g1.g2}}\,.
\]
This must be a multiple of $g_1=(x^\xi-k\,y^\xi)\,(dx^2+\varepsilon dy^2)$, implying $|h|=|k|$ and then $h=k$. Redefining $K\to Kh^{\frac{1}{\xi}}$, and applying $(x,y)\to(Kx,Ky)$ to~\eqref{eqn:A.I.generic.metric}, we obtain
\[
 g \to K^{\xi+2}\,\left(\frac{x^\xi-hy^\xi}{(K_1+K_2h\,K^\xi y^\xi)(K_1+K_2\,K^\xi x^\xi)^2}\,dx^2
 +\varepsilon\frac{x^\xi-hy^\xi}{(K_1+K_2h\,K^\xi y^\xi)^2(K_1+K_2\,K^\xi x^\xi)}\,dy^2\right)\,,
\]
verifying that this transformation is form-preserving.
Choosing \smash{$K^{\xi}=\big|\frac{K_1}{K_2}\big|$}, and defining $\varrho=\sgn(\K1\K2)$ and $\kappa=K^{\xi+2}\,\K1^{-3}$, we thus have
\begin{equation}\label{eqn:A.I.result.1}
 g \to \kappa\,\left(\frac{x^\xi-hy^\xi}{(1+\varrho h\,y^\xi)(1+\varrho x^\xi)^2}\,dx^2
+\varepsilon\frac{x^\xi-hy^\xi}{(1+\varrho h\,y^\xi)^2(1+\varrho x^\xi)}\,dy^2\right)\,.
\end{equation}
Applying the second transformation in~\eqref{eqn:A.I.candidate.transformations} to~$\g1$ yields
\[
 g_1 \to K^{\xi+2}\,h\varepsilon\,\left(x^\xi-\frac{|hk|}{h}y^\xi\right)\,(dx^2+\varepsilon (kh)^{\frac{2}{\xi}} dy^2)\,,
\]
from which we conclude $|kh|=1$ and then $k=\frac1h$. Since $|h|\geq1$, we therefore conclude this transformation may only be applied if $k=h\in\{\pm1\}$.\footnote{Note that this restriction arises from the fact that we have already made use of the freedom to swap $x\leftrightarrow y$ in Section~\ref{sec:optimizing.h.C}. In fact, this freedom was used to replace $h\to\frac{1}{h}$ in some cases (identifying two projective classes). For the special case $|h|=1$, this is an identity and swapping $x\leftrightarrow y$ can therefore be used to identify metrics within the same projective class.}
Applying the resulting transformation $(x,y)\to(Ky,Kx)$ to~\eqref{eqn:A.I.generic.metric}, we find
\[
 g \to h\,\left(\frac{x^\xi-k y^\xi}{(K_1+K_2h\,K^\xi x^\xi)(K_1+K_2\,K^\xi y^\xi)^2}\,dx^2
			+\varepsilon\frac{x^\xi-k y^\xi}{(K_1+K_2h\,K^\xi x^\xi)^2(K_1+K_2\,K^\xi y^\xi)}\,dy^2\right)\,,
\]
and letting $K^\xi=\left|\frac{K_1}{K_2h}\right|$ we obtain, with the definition $\kappa=\frac{K^{\xi+2}}{\K1^3}$,
\begin{equation}\label{eqn:A.I.result.2}
g \to \kappa\,\frac{h\varepsilon}{K_1^3}\,\left(\frac{x^\xi-ky^\xi}{(1+\varrho kx^\xi)(1+\varrho y^\xi)^2}\,dx^2
       +\varepsilon\frac{x^\xi-ky^\xi}{(1+\varrho\sgn(h)x^\xi)^2(1+\varrho\sgn(k)ky^\xi)}\,dy^2\right)\,.
\end{equation}
Finally, redefining $\varrho\to\varrho\sgn(h)$, the resulting metric is as in~\eqref{eqn:A.I.result.1}, if we replace $h\to\frac1h$ and $\kappa\to\kappa\,h\varepsilon$.
How can we use this freedom to optimize the parameters? If $h<0$, we obtain a metric similar in form but with reversed sign of $\varrho$. Thus we may assume $\varrho=1$ in the normal form. If $h>0$ and $\varepsilon=1$, we obtain a metric of the same form as the initial metric, with the only difference that $\kappa$ needs to be replaced by $-\kappa$. This enables us to choose $\kappa>0$ in the normal form. If $\varepsilon=-1$ and $h=1$, we obtain exactly the initial metric, leaving us no means to reduce the parametric freedom further.

\begin{remark}[The special case $\lambda=-1$]
If $\lambda=-1$, i.e.\ $\xi=4$, we also need to take into account the transformations~\eqref{eqn:A.I.candidate.transformations.special}. The first, $(x,y) = \left( \frac{K}{u},\frac{K}{(kh)^{\nicefrac14}\,v} \right)$, leads to
\[
 \g1 \to -K^6\,\left(\frac{1}{u^4}-\frac{h}{|hk|v^4}\right)\left(\frac{du^2}{u^4}+\varepsilon\,\frac{dv^2}{\sqrt{kh}\,v^4}\right)\,,
\]
which implies $k=h>0$. Doing the analogous computation for $\g2$, we achieve $h=1$.
The second transformation in~\eqref{eqn:A.I.candidate.transformations.special}, $(x,y) = \left( \frac{K}{v},\left|\frac{k}{h}\right|^{\nicefrac14}\frac{K}{u} \right)$, leads to
\[
 \g1 \to \varepsilon K^6\,\frac{h|k|}{|h|}\,\left(\frac{1}{u^4}-\frac{|h|}{|k|h\,v^4}\right)\left(\frac{du^2}{u^4}+\varepsilon\,\frac{dv^2}{v^4}\right)\,,
\]
from which follows $h=k=1$.
Summarizing, we have found that if $\xi=4$, the additional freedom of transformations~\eqref{eqn:A.I.candidate.transformations.special} might only be applied if $h=1$. But this configuration has been explicitly ruled out by Theorem~\ref{thm:theorem.3.revised}, as such metrics admit a non-trivial Killing vector field and are not of the kind that we study in this paper. We can therefore discard~\eqref{eqn:A.I.candidate.transformations.special} completely.
\end{remark}

\paragraph{There are no coordinate transformations between the normal forms.}
The above computations suffice to prove that all metrics of type A(I) can be brought into the form~\eqref{eqn:A.I.result.1} or~\eqref{eqn:A.I.result.2}, respectively, and that these normal forms are mutually non-diffeomorphic.
This follows by construction since we checked all possible coordinate transformations that preserve the form of the metric~\eqref{eqn:A.I.generic.metric}, i.e.\ we ensured that all available freedom is made use of.
However, it is instructive to present another approach here that verifies our result, providing further corroboration, and insight (c.f.\ Section~\ref{sec:poles.isometries}).
To this end, let us consider the square of the length of the projective vector field as in Equation~\eqref{eqn:isometries.length.preservation}.
Writing down~\eqref{eqn:isometries.length.preservation} explicitly for~\eqref{eqn:A.I.result.1} resp.~\eqref{eqn:A.I.result.2}, we have
\begin{multline}\label{eqn:equality.length.A.I}
  \kappa\,\left(\frac{x_0^\xi-hy_0^\xi}{(1+\varrho hy_0^\xi t)(1+\varrho x_0^\xi t)^2}\,x_0^2
  +\varepsilon\,\frac{x_0^\xi-hy_0^\xi}{(1+\varrho hy_0^\xi t)^2(1+\varrho x_0^\xi t)}\,y_0^2\right) \\
  =\kappa'\,\left(\frac{x_0'^\xi-h'y_0'^\xi}{(1+\varrho'h'y_0'^\xi t)(1+\varrho'x_0'^\xi t)^2}\,x_0'^2
  +\varepsilon\,\frac{x_0'^\xi-h'y_0'^\xi}{(1+\varrho'h'y_0'^\xi t)^2(1+\varrho'x_0'^\xi t)}\,y_0'^2\right)\,.
\end{multline}
Recalling Section~\ref{sec:poles.isometries}, let us study poles of the functions in $t$ on the left and the right hand side of Equation~\eqref{eqn:equality.length.A.I}. We find below that this is enough for our purposes.
Possible poles are the zeros of the denominators, provided $x_0^\xi\ne hy_0^\xi$ and $x_0'^\xi\ne hy_0'^\xi$, respectively.
Determining these possible poles and equating them pairwisely, we find:
\begin{center}
 \begin{tabular}{cccc}
	(a) &
	(b) &
	(c) &
	(d) \\
	$\frac{h'y_0'^\xi}{hy_0^\xi}=\varrho'\varrho$ &
	$\frac{x_0'^\xi}{hy_0^\xi}=\varrho'\varrho$ &
	$\frac{h'y_0'^\xi}{hy_0^\xi}=\varrho'\varrho$ &
	$\frac{x_0'^\xi}{hy_0^\xi}=\varrho'\varrho$ \\
	$\frac{x_0'^\xi}{x_0^\xi}=\varrho'\varrho$ &
	$\frac{h'y_0'^\xi}{x_0^\xi}=\varrho'\varrho$ &
	$\frac{h'y_0'^\xi}{x_0^\xi}=\varrho'\varrho$ &
	$\frac{x_0'^\xi}{x_0^\xi}=\varrho'\varrho$ \\
	$\Downarrow$ &
	$\Downarrow$ &
	$\Downarrow$ &
	$\Downarrow$ \\
	$\varrho'=\varrho$ &
	$\sgn(h)=\varrho'\varrho$ &
	$\sgn(h')=\varrho'\varrho$ &
	$\varrho'=\varrho$ \\
	$\sgn(h')=\sgn(h)$ &
	$\sgn(h')=\varrho'\varrho$ &
	$\sgn(h)=1$ &
	$\sgn(h)=1$
 \end{tabular}
\end{center}
Let us consider the individual cases separately:
In case~(a), eliminating  $x_0',y_0'$ as well as $\varrho'$ in~\eqref{eqn:equality.length.A.I}, we obtain the requirement $|h'|=|h|$, and thus $h'=h$, and, reinserting this into~\eqref{eqn:equality.length.A.I}, the expressions in the brackets are identical, from which follows $\kappa'=\kappa$. Thus, we conclude that there is no coordinate transformation between two different metrics of the form~\eqref{eqn:A.I.result.1}.
In case~(b), Equation~\eqref{eqn:equality.length.A.I} becomes
\begin{multline*}
  \kappa\,\left(\frac{x_0^\xi-hy_0^\xi}{(1+\varrho hy_0^\xi t)(1+\varrho x_0^\xi t)^2}\,x_0^2
  +\varepsilon\,\frac{x_0^\xi-hy_0^\xi}{(1+\varrho hy_0^\xi t)^2(1+\varrho x_0^\xi t)}\,y_0^2\right) \\
  =\kappa'\sgn(h)\,\left(\frac{hy_0^\xi-x_0^\xi}{(1+\varrho x_0^\xi t)(1+\varrho hy_0^\xi t)^2}\,h^{\nicefrac{2}{\xi}}y_0^2
  +\varepsilon\,\frac{hy_0^\xi-x_0^\xi}{(1+\varrho x_0^\xi t)^2(1+\varrho hy_0^\xi t)}\,(h')^{-\nicefrac{2}{\xi}}\,x_0^2\right)\,.
\end{multline*}
We conclude first that $|h'h|=1$, and then we infer from $|h|\geq1$ that $h'=h=\varrho'\varrho$ hold.
But this implies that $\kappa'=-\sgn(h)\varepsilon\,\kappa$, and thus we obtain the following table.
\begin{center}
\textbf{Transformation behaviour of parameters $\kappa$ and $\varrho$}\smallskip

\begin{TAB}(r,0.5cm,0.5cm)[3pt]{|c|c|c|}{|c|c|c|}
		    & $h=-1$		& $h=+1$ \\
  $\varepsilon=-1$  & $\kappa'=-\kappa$, $\varrho'=-\varrho$
		    & $\kappa'=\kappa$,\quad $\varrho'=\varrho$ \\
  $\varepsilon=+1$  & $\kappa'=\kappa$,\quad $\varrho'=-\varrho$
		    & $\kappa'=-\kappa$, $\varrho'=\varrho$ \\
\end{TAB}
\end{center}
Now, since for $h=-1$ the normal form~\eqref{eqn:A.I.result.2} requires $\varrho=1$, the case~(b) cannot apply, meaning we cannot associate the poles in this way. Similarly, if $h=1$ and $\varepsilon=1$, we chose $\kappa>0$ and the case~(b) again cannot apply. We are therefore left with our finding in (a), namely that if two normal forms are linked by a coordinate transformation, we already have $(\kappa',\varrho')=(\kappa,\varrho)$, and the normal forms are identical.

For the cases~(c) and~(d), we can proceed analogously, but the finding remains the same, confirming that two normal forms must already be identical, if a coordinate transformation exists between them.
Therefore we have proven the following.
\begin{proposition}\label{prop:normal.forms.AI}
A metric of the type A(I) in Theorem~\ref{thm:matveev} that admits exactly one essential projective vector field can locally, by a suitable coordinate transformation, be mapped into one and only one of the following metrics.
	\begin{enumerate}[label={(\roman*)}]
		\item For $\kappa\in\Rnz$, $\varrho\in\pmo$, and either $|h|>1$, $\varrho\in\pmo$, or $h=1$ and $\varepsilon=-1$:
		\begin{equation*}
		g = \kappa\left(\frac{x^\xi-hy^\xi}{(1+\varrho hy^\xi)(1+\varrho x^\xi)^2} dx^2 +\frac{\varepsilon\,(x^\xi-hy^\xi)}{(1+\varrho hy^\xi)^2(1+\varrho x^\xi)} dy^2\right),
		\end{equation*}
		\item \emph{Special case $h=-1$, i.e.\ choose $\varrho=1$ in~\eqref{eqn:A.I.result.2}.}
		For $\varepsilon\in\pmo$ and $\kappa\in\Rnz$, the normal forms are:
		\begin{equation*}
		g = \kappa\left(\frac{x^\xi+y^\xi}{(1-y^\xi)(1+x^\xi)^2} dx^2 +\frac{\varepsilon\,(x^\xi+y^\xi)}{(1-y^\xi)^2(1+x^\xi)} dy^2\right),
		\end{equation*}
		\item \emph{Special case $h=+1$, $\varepsilon=+1$ (choose $\kappa>0$).}
		For $\varrho\in\pmo$, the normal forms are:
		\begin{equation*}
		g = \kappa\left(\frac{x^\xi-y^\xi}{(1+\varrho y^\xi)(1+\varrho x^\xi)^2} dx^2 +\frac{(x^\xi-y^\xi)}{(1+\varrho y^\xi)^2(1+\varrho x^\xi)} dy^2\right),
		\end{equation*}
	\end{enumerate}
The parameter $\xi$ can assume values $\xi\in(0,1)\cup(1,4]$. In case $\xi=2$ we require $h\ne-\varepsilon$, in case $\xi=3$ we require $(h,\varepsilon)\ne(\pm1,-1)$ and in case $\xi=4$ we require $h\ne1$.
\end{proposition}

\subsubsection{Normal forms (A.2) of Theorem~\ref{thm:normal.forms}}\label{sec:A.2}
Let us now take the case A(II) in Theorem~\ref{thm:matveev}, i.e.\ we assume $\lie_w:\solSp\to\solSp$ has exactly one $1$-dimensional eigenspace, cf.\ Remark~\ref{rmk:eigenspaces.are.1d}.
A metric from this class takes the form, with $|h|\geq1$ and $\K1\ne0$,
\begin{equation}\label{eqn:metric.AII.generic}
 g[h;\K1,\K2] = \frac{(y-x)\,e^{-3x}}{(\K1x+\K2)^2(\K1y+\K2)}\,dx^2
  +\frac{h\,(y-x)\,e^{-3y}}{(\K1x+\K2)(\K1y+\K2)^2}\,dy^2\,.
\end{equation}

\begin{remark}
The metric~\eqref{eqn:metric.AII.generic} is obtained through Formula~\eqref{eqn:general.metric} from the following two metrics:
\begin{subequations}\label{eqn:basis.metrics.IIA}
\begin{align*}
\g1 = g[h;1,0] &=\frac{1}{x} \left(\frac{1}{x}-\frac{1}{y}\right) e^{-3x} dx^2+\frac{h}{y} \left(\frac{1}{x}-\frac1y\right) e^{-3y} dy^2 \\
\g2 = g[h;0,1] &=(y-x) e^{-3x} dx^2+(y-x) h e^{-3y} dy^2.
\end{align*}
\end{subequations}
Note that~$\g2$ corresponds to an eigenvector of~$\lie_w$. Thus we assume~$\K1\ne0$ in view of Proposition~\ref{prop:homotheties.and.eigenvectors}.
Explicitly, the projective vector field is \eqref{eqn:proj.vector.field.A}.
The flow of~$w$ is $\phi_t(x_0,y_0)=(x_0+t,y_0+t)$.
\end{remark}

\noindent We proceed as in Section~\ref{sec:A.1}, but with Equation~\eqref{eqn:benenti.1r} instead of~\eqref{eqn:benenti.2r}, i.e.
\[
  L=-L(\g2,\K1\g1+\K2\g2) = \begin{pmatrix}
  K_1x-K_2 & 0 \\ 0 & K_1y-K_2
  \end{pmatrix}\,.
\]
Let us study a change of coordinates $(x,y)\to(u,v)$ that preserves $L$ up to~\eqref{eqn:benenti.1r}. The possible transformations follow from~\eqref{eqn:possible.transf.1r.diagonal.L}, i.e.
\begin{equation}\label{eqn:trafos}
  u=Kx+C,\ \ v=Ky+C\qquad\text{or}\qquad
  u=Ky+C,\ \ v=Kx+C\,.
\end{equation}
These transformations have to preserve the $1$-dimensional eigenspace.
Therefore, apply~\eqref{eqn:trafos} to the metric~$\g2$,
\[
  g_2 \to K^3e^{-3C}\,(y-x)\,(e^{-3Kx}dx^2+he^{-3Ky}dy^2)\,,
  \ \ \text{respectively,}\ \
  -K^3e^{-3C}\,(y-x)\,(he^{-3Kx}dx^2+e^{-3Ky}dy^2)\,.
\]
By a straightforward comparison with $\g2$, we find $K=1$.
In order to arrive at normal forms, let us consider a generic metric \eqref{eqn:metric.AII.generic} with $\K1\ne0$.
Using $(x,y) \to (x+C, y+C)$, we can put~\eqref{eqn:metric.AII.generic} into the same form with $K_2=0$, such that we arrive at
\begin{equation}\label{eqn:A.II.normal.form}
  g[h;\kappa] = \kappa\,\left(\frac{(y-x)\,e^{-3x}}{x^2y}\,dx^2
				+ \frac{h\,(y-x)\,e^{-3y}}{xy^2}\,dy^2\right)\,,
\end{equation}
with $\kappa=\frac{e^{3\frac{K_2}{K_1}}}{K_1^3}$.
%
Finally, consider \eqref{eqn:metric.AII.generic} with the transformation
$(x,y) \to (y+C, x+C)$.
The only case when this reproduces a metric of the same form, is when $h=1$, implying that we reproduce~\eqref{eqn:A.II.normal.form} up to its sign,
and thus we are free the choose the sign of $\K1$.
Therefore, \eqref{eqn:A.II.normal.form} is indeed the required normal form after provision of $\kappa\in\Rnz$, if $h\ne1$, or $\kappa>0$, if $h=1$.
Summarizing, we have proven the following proposition.
\begin{proposition}\label{prop:normal.forms.AII}
A metric of the type A(II) in Theorem~\ref{thm:matveev} that admits exactly one, essential projective vector field can locally, by a suitable coordinate transformation, be mapped into one and only one of the following metrics.
\begin{equation}\label{eqn:normal.form.IIA}
  g=\kappa\,\left(\frac{(y-x)e^{-3x}}{x^2y}dx^2 + \frac{h (y-x)e^{-3y}}{xy^2}dy^2\right)
\end{equation}
with $\kappa\in\Rnz$. If $h=1$, we require $k>0$.
\end{proposition}

\begin{remark}
Note that the normal forms~\eqref{eqn:normal.form.IIA} are constant multiples of metric~$\g1$ of~\eqref{eqn:basis.metrics.IIA}. In Figure~\ref{fig:orbit.II} on page~\pageref{fig:orbit.II}, these metrics lie on the axis of abscissa. In the special case $h=1$, the normal forms~\eqref{eqn:normal.form.IIA} lie in the right half of the horizontal axis only.
\end{remark}

\subsubsection{Normal forms (A.3) of Theorem~\ref{thm:normal.forms}}\label{sec:A.3}
Metrics of case A(III) in Theorem~\ref{thm:matveev} are never eigenvectors of $\lie_w$, but we may still use a modification of the method used in Sections~\ref{sec:A.1} and~\ref{sec:A.2}.
Instead of the Benenti tensor for eigenvectors of $\lie_w$ (which do not exist), we use a family of Benenti tensors as discussed in Section~\ref{sec:techniques.strategies}.
Namely, we use the Benenti tensor for the two generating metrics $\g1,\g2$ given in Theorem~\ref{thm:matveev}.
According to Lemma~\ref{la:lambda.different.classes} we require $\lambda\geq0$, and from Section~\ref{sec:optimizing.h.C} we infer $h\leq1$. Moreover, if $\lambda=0$ we require $|h|\ne1$.
From~\eqref{eqn:general.metric} we obtain that any metric~$g\in\projcl(\g1,\g2)$ in the projective class of the metrics~$\g1,\g2$, see Theorem~\ref{thm:matveev}, is of the form
\begin{equation}\label{eqn:A.III.general.metric}
\begin{split}
 g[\lambda,h;\K1,\K2]=(\tan(y)-\tan(x))\,
 &\left(
   \frac{e^{-3\lambda x}}{\cos(x) (\K2\tan(y)-\K1) (\K2\tan(x)-\K1)^2}\,dx^2
   \right. \\
 &\left.\qquad
   +\frac{h e^{-3\lambda y}}{\cos(y) (\K2\tan(y)-\K1)^2 (\K2\tan(x)-\K1)}\,dy^2
   \right)\,.
\end{split}
\end{equation}
The projective vector field, in these coordinates, is \eqref{eqn:proj.vector.field.A}, and its flow $\phi_t(x_0,y_0)=(x_0+t,y_0+t)$.
Metric~\eqref{eqn:A.III.general.metric} actually admits a simple representative under geodesic equivalence.
\begin{lemma}\label{la:metrics.III.representatives}
 Let $g$ be a metric of type A(III) of Theorem~\ref{thm:matveev} but with the additional restrictions of Theorems~\ref{thm:theorem.2.revised} and~\ref{thm:theorem.3.revised}.
 There exists a coordinate transformation $\tau$ that maps~$g$ onto a multiple of~$\g1$,
 \begin{equation}\label{eqn:representative.A.III}
  \tau^*(g)
  = e^{\beta}\,\g1
  = e^{\beta}\,(\tan(x)-\tan(y))\,\left(\frac{e^{-3\lambda x}}{\cos(x)}\,dx^2 + \frac{h\,e^{-3\lambda y}}{\cos(y)}\,dy^2\right)
 \end{equation}
 where $\beta\in\left(0,2\pi\lambda\right]$ if $\lambda>0$ and $\beta\in(-\infty,\infty)$ if $\lambda=0$.
\end{lemma}
\begin{proof}
 Let $\phi_t$ be the flow of the projective vector field. Then $\tau^*(g)=\phi_t^*(g)$ for a suitably chosen~$t$. Indeed, considering the parametrization discussed in Remark~\ref{rmk:III.parameter.values}, $t$ should be chosen (if $\lambda>0$) such that $\tau^*(g)$ is the intersection closest to the origin of the orbit with the $\g1$-axis. For $\lambda=0$ the intersection of the orbit with the positive $\g1$-axis is unique, see also Figure~\ref{fig:orbit.III}.
\end{proof}

\noindent Lemma~\ref{la:metrics.III.representatives} tells us that we can restrict our investigation to coordinate changes that preserve the $\g1$-direction. Thus, we may consider the Benenti tensor ($\K1,\K2\in\R$)
\begin{equation}\label{eqn:benenti.A3.case}
  L = -L(\a1,\K1\a1+\K2\a2) =
  \begin{pmatrix} \K1\tan(x)+\K2 & \\ & \K1\tan(y)+\K2 \end{pmatrix}
\end{equation}
where $\a{i}$ corresponds to $\g{i}$ via~\eqref{eqn:general.metric}.
Proceeding as in Section~\ref{sec:strategy.benenti}, we arrive at the analog of Equation~\eqref{eqn:possible.transf.2r.diagonal.L}, i.e.\ the new coordinates $u(x,y)$, $v(x,y)$ satisfy
\begin{equation}\label{eqn:A.III.trafo}
  \tan(u)=K\tan(x)+C\,,\ \tan(v)=K\tan(y)+C\,,
  \quad\text{or}\quad
  \tan(u)=K\tan(y)+C\,,\ \tan(v)=K\tan(x)+C\,,
\end{equation}
with $K\ne0$ and $C\in\R$.
We require the transformation $(x,y)\to(u,v)$ to preserve $\g1$ up to multiplication by a non-zero constant, i.e.\ we require the existence of a transformation from representatives to representatives (note that if two metrics are connected by a coordinate transformation, there exists also a coordinate transformation between their representatives).
We obtain, applying the first transformation of~\eqref{eqn:A.III.trafo} to a copy of~$\g1$,
\begin{align*}
 \g1 &= (\tan(u)-\tan(v))\,\left(\frac{e^{-3\lambda u}}{\cos(u)}\,du^2 + \frac{h\,e^{-3\lambda v}}{\cos(v)}\,dv^2\right) \\
     &= K^2\,(\tan(x)-\tan(y))\,\left(\frac{e^{-3\lambda\arctan(K\tan(x)+C)}}{\cos(\arctan(K\tan(x)+C))}\,\frac{\sec^4(x)\,dx^2}{((K\tan(x)+C)^2+1)^2} \right. \\
     &\hspace{4cm} \left. + \frac{h\,e^{-3\lambda\arctan(K\tan(y)+C)}}{\cos(\arctan(K\tan(y)+C))}\,\frac{\sec(y)^4\,dy^2}{((K\tan(y)+C)^2+1)^2}\right) \\
     &= K^2\,(\tan(x)-\tan(y))\,\left(\frac{e^{-3\lambda\arctan(K\tan(x)+C)}}{\cos^4(x)\,((K\tan(x)+C)^2+1)^{\nicefrac32}}\,dx^2 \right. \\
     &\hspace{4cm} \left. + \frac{h\,e^{-3\lambda\arctan(K\tan(y)+C)}}{\cos^4(y)\,((K\tan(y)+C)^2+1)^{\nicefrac32}}\,dy^2\right) \\
     &\propto (\tan(x)-\tan(y))\,\left(\frac{e^{-3\lambda x}}{\cos(x)}\,dx^2 + \frac{h\,e^{-3\lambda y}}{\cos(y)}\,dy^2\right)
\end{align*}
This can only be satisfied if
\begin{align*}
 (K\tan(x)+C)^2+1=\frac{1}{\cos^2(x)}
 &\quad\Leftrightarrow\quad
 (K^2-1)\sin^2(x)+2KC\cos(x)\sin(x)+C^2\cos^2(x)=0
 \\
 &\quad\Leftrightarrow\quad
 K^2=1\,, C=0\,.
\end{align*}
Let us consider the cases $\lambda=0$ and $\lambda>0$ separately, as suggested by Remark~\ref{rmk:III.parameter.values}.
Assume first $\lambda>0$. Then the situation is as shown on the right panel of Figure~\ref{fig:nf.C.III}.
Resubstituting $K=\pm1$, $C=0$ into $\g1$, we obtain $K=1$ and $C=0$, since otherwise the exponents do not transform correctly.
Now let us assume $\lambda=0$. In this case, $K=-1$ would change the sign of a constant conformal factor to $\g1$. However, the normal forms require a positive factor, which implies $K=1$.
We have therefore confirmed that the coordinates transform as follows:
\begin{equation}\label{eqn:A.III.transformation}
  u = x+N\pi\,,\
  v = y+M\pi\,,\quad\text{or}\quad
  u = y+N\pi\,,\
  v = x+M\pi\,,\quad\text{for some}\ N,M\in\Z\,.
\end{equation}
Let us start with the case $\lambda=0$. We can use the second transformation to transform~\eqref{eqn:A.III.general.metric} into itself, but with $\theta\in\halfangl$ after the transformation. But, by a simple shift of $x,y$ by $\theta$, we remove $\theta$ from~\eqref{eqn:A.III.general.metric}, so the metric can essentially not be changed.
Next, let us consider $\lambda>0$ with the first transformation in~\eqref{eqn:A.III.transformation},
\[
 \g1\to (-1)^N\,e^\beta e^{-3\lambda\,N\pi}\,(\tan(x)-\tan(y))\,\left(\frac{e^{-3\lambda x}}{\cos(x)}\,dx^2 + (-1)^{M-N} he^{-3\lambda(M-N)\pi}\,\frac{\,e^{-3\lambda y}}{\cos(y)}\,dy^2\right)\,.
\]
We have $\beta\to\beta-3\lambda N\pi$, and we must require $N$ to be even.
Since $0<\beta-2N\lambda\pi\leq 2\lambda\pi$, however, already $N=0$.
Thus, $h\to (-1)^M he^{-3\lambda\,M\pi}$.
The transformed metric must satisfy $|he^{-3\lambda\,(M-N)\pi}|\leq e^{-3\lambda\pi}$.
On the other hand, we have
\[
  |he^{-3\lambda\,M\pi}|
  = |h| e^{-3\lambda\,M\pi}
  \leq e^{-3\lambda\pi\,(1+M)}\,,
\]
which allows us to restrict to $e^{-3\lambda\pi}<h\leq1$.
The second transformation in~\eqref{eqn:A.III.transformation} needs only to be considered when $h=1$. In this case, however,
\[
 \g1\to (-1)^N\,e^\beta e^{-3\lambda\,N\pi}\,(\tan(x)-\tan(y))\,\left(\frac{e^{-3\lambda x}}{\cos(x)}\,dx^2 + (-1)^{M-N} e^{-3\lambda(M-N)\pi}\,\frac{\,e^{-3\lambda y}}{\cos(y)}\,dy^2\right)\,,
\]
and thus $M=N$. Moreover, as in the first case, we obtain $N=0$.
From the above considerations we infer restrictions on the parameters. Incorporating these restrictions into~\eqref{eqn:A.III.general.metric}, we obtain the following proposition.

\begin{proposition}\label{prop:normal.forms.AIII}
A metric of the type A(III) in Theorem~\ref{thm:matveev} that admits exactly one essential projective vector field can, locally by a suitable coordinate transformation, be mapped onto one and only one of the following metrics.
\begin{enumerate}[label={(\roman*)}]
	\item for $\lambda=0$:\qquad
	$g=\kappa\,\sin(x-y)\,\left(
	\frac{dx^2}{\sin(y)\,\sin^2(x)}
	+\frac{h\,dy^2}{\sin^2(y)\,\sin(x)}
	\right)$\,,

	\item for $\lambda>0$:\qquad
	$g=\sin(x-y)\,\left(
	 \frac{e^{-3\lambda\,x}}{\sin(y-\theta)\,\sin^2(x-\theta)}\,dx^2
	 +\frac{h\,e^{-3\lambda\,y}}{\sin^2(y-\theta)\,\sin(x-\theta)}\,dy^2
	\right)$\,.
\end{enumerate}
If $\lambda=0$, $|h|<1$ and $\kappa>0$.
If $\lambda>0$, $h$ assumes values $e^{-3\lambda\pi}<h\leq1$, and $\theta\in\angl$.
\end{proposition}
\begin{proof}
 The normal forms are~\eqref{eqn:representative.A.III} with the additional restrictions on $h$.
 However, it seems desirable to choose new coordinates.
 Specifically, using standard trigonometrical identities and the reparametrization~\eqref{eqn:III.reparametrization}, we obtain
 \begin{equation}
  g = \sin(y-x)\,\left(
		  \frac{e^{-3\lambda\,(x+\theta+\beta)}\,dx^2}{\sin(y-\theta)\,\sin^2(x-\theta)}
		  +\frac{h\,e^{-3\lambda\,(y+\theta+\beta)}\,dy^2}{\sin^2(y-\theta)\,\sin(x-\theta)}
		\right)\,.
 \end{equation}
After a suitable translational change of the coordinates, we then arrive at the asserted normal forms. Note that in case $\lambda>0$, we also redefine $\theta_\text{old}\to\theta_\text{new}=2\theta_\text{old}+\beta$, and thus arrive at the requirement $\theta_\text{new}\in\angl$.
\end{proof}

\subsubsection{Normal forms (B.4) of Theorem~\ref{thm:normal.forms}}\label{sec:B.4}
A metric from class B(I) in Theorem~\ref{thm:matveev} takes the form
\begin{equation}\label{eqn:B.I.general.metric}
  g[\xi,C;\K1,\K2]
  = \frac{Cz^\xi-\cc{C}\cc{z}^\xi}{(\K1+\K2\cc{C}\cc{z}^\xi)(\K1+\K2Cz^\xi)^2}\,dz^2
  -\frac{Cz^\xi-\cc{C}\cc{z}^\xi}{(\K1+\K2\cc{C}\cc{z}^\xi)^2(\K1+\K2Cz^\xi)}\,d\cc{z}^2
\end{equation}
where $C\in\mathds{C}$, $C=e^{i\varphi}$ with $\varphi\in[0,\pi)$, and $\xi=2\,\frac{\lambda-1}{2\lambda+1}$ as in~\eqref{eqn:xi.via.lambda}. We require $\xi\in(0,4]$ with the restrictions $\xi\ne1$ and: if $\xi=2$ then $\varphi\in(0,\pi)$; if $\xi=3$ then $\varphi\in(0,\pi)$ with $\varphi\ne\tfrac{\pi}{2}$.

\begin{remark}
Metric~\eqref{eqn:B.I.general.metric} is obtained, using Formula~\eqref{eqn:general.metric}, from the pair $\g1,\g2$ of Theorem~\ref{thm:matveev}.
The metrics~$\g1$ and~$\g2$ correspond, via~\eqref{eqn:Liouville.tensor.g}, to non-proportional eigenvectors $\a1,\a2$ of~$\lie_w$, respectively.
Furthermore, we have
\begin{equation*}
 \lie_w\vect{\a1\\ \a2}=\begin{pmatrix} \lambda & 0\\ 0 & 1 \end{pmatrix}\vect{\a1\\ \a2}
 \qquad\text{and}\qquad
 w=-\frac{2\lambda+1}{2}\,(z\partial_z+\cc{z}\partial_{\cc{z}})\,.
\end{equation*}
\end{remark}

\noindent The Benenti tensor~\eqref{eqn:benenti.2r} for $\g1=g[\xi,C;1,0]$ and $\g2=g[\xi,C;0,1]$ reads
\[
  L(\g2,\g1) =
  \begin{pmatrix}
  -Cz^\xi & 0 \\ 0 & -\cc{C}\cc{z}^\xi
  \end{pmatrix}\,.
\]
We are looking for coordinate transformations $(z,\cc{z})\to(w,\cc{w})$ such that $L$ assumes the same form, up to a constant conformal factor.
Indeed, we find, analogously to Equation~\eqref{eqn:possible.transf.2r.diagonal.L}, the possible form-preserving transformations $\tau:(z,\cc{z})\to(w,\cc{w})$, with
$z=Aw$ or $z=A\cc{w}$, where $A>0$ is a positive real number and $\cc{z}$ denotes the complex conjugate of $z$.
Using these transformations, we arrive, in a manner analogous to that in Section~\ref{sec:A.1}, at the following normal forms

\begin{proposition}\label{prop:normal.forms.BI}
Metrics of type B(I) of Theorem~\ref{thm:matveev} that admits exactly one, essential projective vector field can, locally by a suitable coordinate transformation, be mapped onto one and only one of the following metrics,
\[
  g[\xi,C;\kappa] = \kappa\,\left(
  \frac{Cz^\xi -\cc{C}\cc{z}^\xi}{(1+Cz^\xi)(1+\cc{C}^\xi)^2}\,dz^2
  +\frac{Cz^\xi -\cc{C}\cc{z}^\xi}{(1+Cz^\xi)^2(1+\cc{C}^\xi)}\,d\cc{z}^2
  \right)\,.
\]
where $\xi\in(0,1)\cup(1,4]$, $\kappa\in\Rnz$ and $C=e^{i\varphi}$, $\varphi\in\halfangl$. If $\xi=2$, $C\ne\pm1$; if $\xi=3$ then $C^2\ne\pm1$.
\end{proposition}

\begin{remark}\label{rmk:conjugation.IB.reparam.class}
The observation in Section~\ref{sec:optimizing.h.C} is crucial: Conjugating the coordinates, $z\leftrightarrow\cc{z}$, the metric $g$ is transformed such that $(\kappa,\varphi)\to(\kappa,-\varphi)$.
It allows us to always transform $\varphi$ such that it lies in $\halfangl$.
\end{remark}

\subsubsection{Normal forms (B.5) of Theorem~\ref{thm:normal.forms}}\label{sec:B.5}
A metric of the type in B(II) takes the form
\begin{equation}\label{eqn:B.II.general.metric}
 g[C;\K1,\K2] =
 \frac{C\,(\cc{z}-z)\,e^{-3z}}{(\K2-\K1\cc{z})(\K2-\K1z)^2}\,dz^2
 -\frac{\cc{C}\,(\cc{z}-z)\,e^{-3\cc{z}}}{(\K2-\K1\cc{z})^2(\K2-\K1z)}\,d\cc{z}^2\,,
\end{equation}
with $C=e^{i\varphi}$, $\varphi\in[0,\pi)$; we also require $\K1\ne0$ (metrics with $\K1=0$ admit a homothetic vector field).

\begin{remark}\label{rmk:case.B.II}
Metric~\eqref{eqn:B.II.general.metric} is obtained, via Formula~\eqref{eqn:general.metric}, from $\g1,\g2$ given in Theorem~\ref{thm:matveev}.
The projective vector field is $w=\partial_z+\partial_{\cc{z}}$, and thus the projective flow reads
$\phi_t(z_0,\cc{z}_0) = (z_0+t,\cc{z}_0+t)$.
Next, let~$C=e^{i\varphi}$.
\end{remark}

\noindent The metrics $\g1=g[C;0,\K2]$ are eigenvectors of $\lie_w$ and the Benenti tensor~\eqref{eqn:benenti.tensor} for $\g2=g[C;0,1]$ reads
\[
  L = -L(a_2,\K1\a1+\K2\a2) = (\K1z-\K2)\,dz\otimes\partial_z + (\K1\cc{z}-\K2)\,d\cc{z}\otimes\partial_{\cc{z}}\,,
\]
where we denote by $\a1$ and $\a2$ the solutions corresponding to $\g1,\g2$, respectively, via Equation~\eqref{eqn:Liouville.tensor.g}.
Analogously to Section~\ref{sec:A.2}, we find that any possible coordinate transformation is of the form
\[
  w = z+k\,,\ \ \cc{w}=\cc{z}+k
  \qquad\text{or}\qquad
  w = \cc{z}+k\,,\ \ \cc{w}=z+k\,.
\]
The computations would just repeat the reasoning in Section~\ref{sec:A.2}, so we instead just give the result.
\begin{proposition}\label{prop:normal.forms.BII}
A metric of the type B(II) in Theorem~\ref{thm:matveev} that admits exactly one, essential projective vector field can, locally by a suitable coordinate transformation, be mapped onto one and only one of the following metrics,
 \[
   g=\kappa\,(\cc{z}-z)\,
 	\left(
 	  \frac{e^{i\varphi}\,e^{-3z}}{z^2\cc{z}}\,dz^2
 	  -\frac{e^{-i\varphi}\,e^{-3\cc{z}}}{z\cc{z}^2}\,d\cc{z}^2
 	\right),
 \]
 with $\varphi\in[0,\pi)$ and $\kappa\in\Rnz$.
\end{proposition}

\subsubsection{Normal forms (B.6) of Theorem~\ref{thm:normal.forms}}\label{sec:B.6}
A metric of type B(III) in Theorem~\ref{thm:matveev} reads
\begin{equation}\label{eqn:B.III.general.metric.1}
\begin{split}
 g=g[\lambda,C;\K1,\K2]
 =(\tan(\cc{z})-\tan(z))\,
   &\left(
    \frac{C e^{-3\lambda z}}{\cos(z) (\K2\tan(\cc{z})-\K1) (\K2\tan(z)-\K1)^2}\,dz^2
   \right. \\
   &\left.\qquad
    - \frac{\cc{C} e^{-3\lambda\cc{z}}}{\cos(\cc{z}) (\K2\tan(\cc{z})-\K1)^2 (\K2\tan(z)-\K1)}\,d\cc{z}^2
   \right)
\end{split}
\end{equation}
where $C=e^{i\varphi}$ with $\varphi\in[0,\pi)1$; if $\lambda=0$, $\varphi\in(0,\pi)$.
Using the parametrization~\eqref{eqn:III.reparametrization}, Metric~\eqref{eqn:B.III.general.metric.1} becomes
\begin{equation}\label{eqn:B.III.general.metric.2}
  g = \frac{\sin(\cc{z}-z)}{K^3}\,\left(
		 \frac{C\,e^{-3\lambda\,(z+\theta)}\,dz^2}{\sin(\cc{z}-\theta)\,\sin^2(z-\theta)}
-\frac{\cc{C}\,e^{-3\lambda\,(\cc{z}+\theta)}\,d\cc{z}^2}{\sin^2(\cc{z}-\theta)\,\sin(z-\theta)}
	\right)\,.
\end{equation}

\begin{remark}
Metrics of type B(III) are obtained, via Formula~\eqref{eqn:general.metric}, from the pair (neither is an eigenvector)
\begin{align*}
 \g1 &= (\tan(z)-\tan(\cc{z}))\,\left(
		 C\,\frac{e^{-3\lambda z}}{\cos(z)}\,dz^2
		 -\cc{C}\frac{e^{-3\lambda\cc{z}}}{\cos(\cc{z})}\,d\cc{z}^2
	\right), \\
 \g2 &= (\cot(z)-\cot(\cc{z}))\,\left(
		C\,\frac{e^{-3\lambda z}}{\sin(z)}\,dz^2
		 -\cc{C}\frac{e^{-3\lambda\cc{z}}}{\sin(\cc{z})}\,d\cc{z}^2
	\right).
\end{align*}
The projective vector field, fixed by condition~\eqref{eqn:fix.proj.vector.field} and the case~(III) of~\eqref{eqn:normal-matrices}, assumes the form
\eqref{eqn:proj.vector.field.B},
and thus its flow reads
$\phi_t(z_0,\cc{z}_0)=(z_0+t,\cc{z}_0+t)$.
\end{remark}

\noindent We continue in two steps, treating the cases of vanishing and non-vanishing~$\lambda$ separately. Such a separation is also suggested by Remark~\ref{rmk:III.parameter.values}.
We begin with the case of vanishing $\lambda$.

\begin{lemma}\label{la:B.III.lambda.zero}
If $\lambda=0$ in~\eqref{eqn:B.III.general.metric.2}, then $g$ is isometric to a metric of the form
\begin{equation}\label{eqn:metric.B.III.final.lambda.zero}
g[C;\kappa] = \kappa\,\sin(z-\cc{z})\,\left(
\frac{C\,dz^2}{\sin(\cc{z})\,\sin^2(z)}
-\frac{\cc{C}\,d\cc{z}^2}{\sin^2(\cc{z})\,\sin(z)}
\right)\,,
\end{equation}
where $\kappa>0$.
There is no coordinate transformation between two different copies $g=g[C;\kappa]$ and $g'=g[C';\kappa']$ of~\eqref{eqn:metric.B.III.final.lambda.zero}.
\end{lemma}
\begin{proof}
Apply the change of coordinates $(x,y)\to(x+\theta,y+\theta)$ to metric~\eqref{eqn:B.III.general.metric.2} and introduce $\kappa=K^{-3}>0$. This yields~\eqref{eqn:metric.B.III.final.lambda.zero}.
Next, let us check that these normal forms are mutually non-diffeomorphic.
This is follows from inspection of Equation~\eqref{eqn:isometries.length.preservation}, evaluated for the current context,
\begin{multline}\label{eqn:B6.lambda.zero.nonisom}
  \kappa\,\sin(z_0-\cc{z}_0)\,\left(
  \frac{C}{\sin(\cc{z}_0+t)\,\sin^2(z_0+t)}
  -\frac{\cc{C}}{\sin^2(\cc{z}_0+t)\,\sin(z_0+t)}
  \right) \\
  =
  \kappa'\,\sin(z_0'-\cc{z}_0')\,\left(
  \frac{C'}{\sin(\cc{z}_0'+t)\,\sin^2(z_0'+t)}
  -\frac{\cc{C}'}{\sin^2(\cc{z}_0'+t)\,\sin(z_0'+t)}
  \right)\,.
\end{multline}
Examining the poles in $t$, we obtain two cases: (a) $z_0'=z_0+N\pi$, $\cc{z}_0'=\cc{z}_0+N\pi$, and (b) $z_0'=\cc{z}_0+N\pi$, $\cc{z}_0'=z_0+N\pi$ (with $N\in\mathds{Z}$).

Substituting back into~\eqref{eqn:B6.lambda.zero.nonisom}, we get an equation on $\kappa,\kappa',C,C'$, only involving the functions $\sin(z_0+t),\sin(\cc{z}_0)$.
Note that these functions are linearly independent; their Wronskian, $i\sinh(2y)$, does not vanish identically. Therefore, $C'=C$ or $C'=\cc{C}$, but revisiting the restrictions on $C$ in Theorem~\ref{thm:matveev} we conclude $C'=C$.
Next, reinserting again into~\eqref{eqn:B6.lambda.zero.nonisom}, we obtain that $\kappa'=(-1)^N\,\kappa$. Due to $\kappa,\kappa'>0$, this implies that $N$ is even and $\kappa'=\kappa$.
\end{proof}

\noindent Next, let us consider the case when $\lambda>0$, i.e., when $\lambda$ is non-zero.

\begin{lemma}\label{la:B.III.lambda.nonzero}
If $\lambda>0$ in~\eqref{eqn:B.III.general.metric.2}, then $g$ is isometric to a metric of the form
\begin{equation}\label{eqn:metric.B.III.final.lambda.nonzero}
 g[C;\theta]= \sin(z-\cc{z})\,\left(
		 \frac{C\,e^{-3\lambda z}}{\sin^2(z+\theta)\,\sin(\cc{z}+\theta)}\,dz^2
		 -\frac{\cc{C}\,e^{-3\lambda\cc{z}}}{\sin(z+\theta)\,\sin^2(\cc{z}+\theta)}\,d\cc{z}^2
		\right)
\end{equation}
with $C=e^{i\varphi}$, $\varphi\in\halfangl$, and $\theta\in\angl$.
There is no coordinate transformation between two different copies $g=g[C;\theta]$ and $g'=g[C';\theta']$ of~\eqref{eqn:metric.B.III.final.lambda.nonzero}.
\end{lemma}
\begin{proof}
Letting $\alpha=\frac{\ln(K)}{\lambda}+\theta$, we may rewrite metric~\eqref{eqn:B.III.general.metric.2} as
\begin{equation*}\label{eqn:metric.B.III.general.lambda.nonzero}
 g = \sin(z-\cc{z})\,\left(
		 \frac{Ce^{-3\lambda\,(z+\alpha)}}{\sin(\cc{z}-\theta)\,\sin^2(z-\theta)}\,dz^2
		 -\frac{\cc{C}\,e^{-3\lambda\,(\cc{z}+\alpha)}}{\sin^2(\cc{z}-\theta)\,\sin(z-\theta)}\,d\cc{z}^2
	\right)\,,
\end{equation*}
where $C=e^{i\varphi}$, $\varphi\in\halfangl$.
By a change of coordinates $z\to z+\alpha$ and a redefinition of $\theta$, implying $\theta\in\angl$, we arrive at the normal form~\eqref{eqn:metric.B.III.final.lambda.nonzero}.
For the second part of the statement consider again the requirement~\eqref{eqn:isometries.length.preservation},
\begin{equation}\label{eqn:lengths.equality.lemma.nf.B.III}
\begin{split}
& \sin(z_0-\cc{z}_0)\,\left(
		 \frac{Ce^{-3\lambda\,(z_0+t)}}{\sin(\cc{z}_0+t-\theta)\,\sin^2(z_0+t-\theta)}
		 -\frac{\cc{C}\,e^{-3\lambda\,(\cc{z}_0+t)}}{\sin^2(\cc{z}_0+t-\theta)\,\sin(z_0+t-\theta)}
	\right) \\
&= \sin(z_0'-\cc{z}_0')\,\left(
		 \frac{C'e^{-3\lambda\,(z_0'+t)}}{\sin(\cc{z}_0'+t-\theta')\,\sin^2(z_0'+t-\theta')}
		 -\frac{\cc{C}'\,e^{-3\lambda\,(\cc{z}_0'+t)}}{\sin^2(\cc{z}_0'+t-\theta')\,\sin(z_0'+t-\theta')}
	\right)
\end{split}
\end{equation}
In order that both sides can be equal, poles of the left hand side must correspond to poles on the right hand side. Thus, we obtain the possibilities:
\begin{center}
\begin{tabular}{cccc}
(a) & (b) & (c) & (d) \\
	$z_0'=z_0+H+N\pi$
	& $z_0'=z_0+H+N\pi$
	& $z_0'=\cc{z}_0+H+N\pi$
	& $z_0'=\cc{z}_0+H+N\pi$ \\
	$\cc{z}_0'=\cc{z}_0+H+N\pi$
	& $\cc{z}_0'=z_0+H+N\pi$
	& $\cc{z}_0'=\cc{z}_0+H+N\pi$
	& $\cc{z}_0'=z_0+H+N\pi$
\end{tabular}
\end{center}
with $N\in\mathds{Z}$ and $H=\arctan\frac{\K2'}{\K1'}-\arctan\frac{\K2}{\K1}=\theta'-\theta$. Again, in the same way as we did before, we can discard cases (b) and (c).
Let us consider case (a):
Plugging the relations $z_0'=z_0+H+N\pi$ and $\cc{z}_0'=z_0+H+N\pi$ into~\eqref{eqn:lengths.equality.lemma.nf.B.III}, we find
\begin{equation*}
e^{-3\lambda\,(H+N\pi)} = (-1)^N
\qquad\quad
\Rightarrow
\qquad\quad
N\in2\mathds{Z}\quad\text{and}\quad\theta' = \theta\,,
\end{equation*}
where we view $\theta',\theta$ as angles.
In case (d), we follow the same procedure and obtain
\begin{equation}\label{eqn:aux.La.17}
  C\,(-1)^N\,e^{-3\lambda\,(H+N\pi)} = \cc{C},
\end{equation}
The real part of~\eqref{eqn:aux.La.17} entails the requirement
$e^{-3\lambda\,(H+N\pi)} = (-1)^{N}$,
which is only possible if $N$ is an even integer.
On the other hand, taking the absolute value of both sides of~\eqref{eqn:aux.La.17} implies
$e^{-3\lambda\,(H+N\pi)} = 1$,
which entails $H+N\pi$.
Since $N$ is even and $\theta,\theta'\in\angl$, we are forced to conclude $N=0$ and $\theta'=\theta$.
Resubstituting into~\eqref{eqn:aux.La.17} yields the restriction $\cc{C}=C$, and thus $C=1$ for case (d) to happen. This completes the proof.
\end{proof}

\noindent We conclude, putting together Lemma~\ref{la:B.III.lambda.zero} and Lemma~\ref{la:B.III.lambda.nonzero}:
\begin{proposition}\label{prop:normal.forms.BIII}
A metric of the type B(III) in Theorem~\ref{thm:matveev} that admits exactly one essential projective vector field can, locally by a suitable coordinate transformation, be mapped onto one and only one of the following metrics,
\begin{enumerate}[label={(\roman*)}]
	\item for $\lambda=0$:\qquad
	$g = \kappa\,\sin(z-\cc{z})\,\left(
	  \frac{C\,dz^2}{\sin^2(z)\,\sin(\cc{z})}
	  -\frac{\cc{C}\,d\cc{z}^2}{\sin(z)\,\sin^2(\cc{z})}
	\right)$\,,

	\item for $\lambda>0$:\qquad
	$g = \sin(z-\cc{z})\,\left(
	\frac{C\,e^{-3\lambda z}}{\sin^2(z+\theta)\,\sin(\cc{z}+\theta)}\,dz^2
	 -\frac{\cc{C}\,e^{-3\lambda\cc{z}}}{\sin(z+\theta)\,\sin^2(\cc{z}+\theta)}\,d\cc{z}^2
	\right)$\,.
\end{enumerate}
Here, $C=e^{i\varphi}$.
If $\lambda=0$, $\varphi\in(0,\pi)$ and $\kappa>0$.
If $\lambda>0$, $\varphi\in\halfangl$ and $\theta\in\angl$.
\end{proposition}

\subsubsection{Normal forms (C.7) of Theorem~\ref{thm:normal.forms}}\label{sec:C.7}
Metrics C(Ia) in Theorem~\ref{thm:matveev} have degree of mobility~3, see~\cite{matveev_2012} and Theorem~\ref{thm:theorem.2.revised}. We therefore turn immediately to case C(Ib). Metrics of this class are of the form ($y>0$)
\begin{equation}\label{eqn:C.Ib.general metric}
 g[\xi;\K1,\K2] =
 2\,\frac{x+y^{\frac{1}{\xi}}}{(\K1-\K2y)^3}\,dxdy
 +\K2\,\frac{(x+y^{\frac{1}{\xi}})^2}{(\K1-\K2y)^4}\,dy^2\,.
\end{equation}

\begin{remark}
 Metrics~\eqref{eqn:C.Ib.general metric} are obtained, via Formula~\eqref{eqn:general.metric}, from
\begin{subequations}\label{eqn:C.Ib.eigenvectors}
\begin{align}
	\g1 &=2(Y(y)+x)dxdy \label{eqn:jordan.block.metric.ICa.g1} \\
	\g2 &=-2\frac{Y(y)+x}{y^3}dxdy + \frac{(Y(y)+x)^2}{y^4}dy^2\,,
\end{align}
\end{subequations}
which lie in the same eigenspace of $\lie_w$.
Here, $y>0$ and $Y(y)=Cy^{\frac{1}{\xi}}$ with $C\in\mathds{R}$, $\xi$ as in~\eqref{eqn:xi.via.lambda}.
Let us quickly remark the following:
If $C=0$, the metric has constant curvature and we can ignore this case. For $C\ne0$, we can w.l.o.g.\ restrict to $C=1$, by grace of a coordinate transformation $x\to Cx$ and subsequent rescaling of the metrics~\eqref{eqn:C.Ib.eigenvectors}.
In the following we therefore continue only with $C=1$.
We recall the definition~\eqref{eqn:xi.via.lambda} of $\xi$ in terms of the eigenvalue~$\lambda$: $\xi=2\,\frac{\lambda-1}{2\lambda+1}$.
The normal forms (C.9) of Theorem~\ref{thm:normal.forms} have degree of mobility~2 (see Theorem~\ref{thm:matveev}), i.e.\ $\dim(\projcl(\g1,\g2))=2$ according to Definition~\ref{def:degree.mobility}. This requires $\lambda\ne\frac52$, i.e.\ $\xi\neq\frac12$, for the present section, since for $\xi=\frac12$ the degree of mobility is 3 (this case is discussed by Theorem~\ref{thm:theorem.2.revised}).
The projective vector field is \eqref{eqn:proj.vector.field.CIb}. We may multiply it through with $-(\lambda+\nicefrac12)$ to obtain
\begin{equation}\label{eqn:proj.vf}
w=-\left(\lambda+\frac{1}{2}\right)x\partial_x - (\lambda-1)y\partial_y\,.
\end{equation}
By integration we infer the projective flow as
\begin{equation}\label{eqn:ICa.proj.flow}
  \phi_t(x_0,y_0)=(x_0\,e^{-(\lambda+\frac12)t},\,y_0\,e^{-(\lambda-1)t})\,.
\end{equation}
\end{remark}
\medskip

\noindent The Benenti tensor~\eqref{eqn:benenti.2r} for $\g1=g[\xi;1,0]$ and $\g2=g[\xi;0,1]$ is, up to a conformal rescaling,
\[
  L=L(\g2,\g1)
  =\begin{pmatrix}
   y & x+y^\eta \\ 0 & y
  \end{pmatrix}\,,
\]
where we introduce $\eta=\frac{1}{\xi}$ for convenience.
Transformations $(x,y)\to(u,v)$ that preserve $L$ up to a constant conformal factor satisfy a set of equations from which, through a straightforward computation (analogous to those in Section~\ref{sec:strategy.benenti}), we can conclude that
\begin{equation}\label{eqn:C.I.trafos}
  (x,y) = (c_1u+c_2,kv)\,.
\end{equation}
Now, the metric
\[
  g_2 = -2\,\frac{y^\eta+x}{y^3}\,dxdy +\frac{(y^\eta+x)^2}{y^4}\,dy^2
\]
is an eigenvector of $\lie_w$.
Using \eqref{eqn:C.I.trafos},
\[
 g_2 \to -2\,\frac{K^\eta v^\eta+c_1u+c_2}{K^3v^3}\,Kc_1\,dudv +\frac{(K^\eta v^\eta+c_1u+c_2)^2}{K^4v^4}\,K^2\,dv^2\,,
\]
we infer
\begin{equation}\label{eqn:C.I.trafo}
  c_2 = 0\,,\quad
  c_1 = K^\eta\qquad
  \text{and thus}\quad
  (x,y) = (K^\eta\,u,Kv)\,.
\end{equation}
Next, let us consider a generic metric~\eqref{eqn:C.Ib.general metric} from the projective class
and apply~\eqref{eqn:C.I.trafo} with $K=\left|\frac{K_1}{K_2}\right|$. Thus,
\begin{equation}\label{eqn:C.7.normal.forms.1}
 g \to
 g[\eta;k,\varrho]
 = k\left( 2\,\frac{y^\eta+x}{(\varrho-y)^3}\,dxdy +\frac{(y^\eta+x)^2}{(\varrho-y)^4}\,dy^2 \right)
\end{equation}
with $k=\frac{K^\eta}{K_2^3}$ and $\varrho=\sgn\left(\tfrac{\K1}{\K2}\right)$.

\begin{proposition}\label{prop:normal.forms.CIb}
A metric of the type C(Ib) in Theorem~\ref{thm:matveev} that admits exactly one essential projective vector field can, locally by a suitable coordinate transformation, be mapped onto one and only one on the following metrics,
\begin{equation*}
g=\kappa\left(-2\frac{y^{\frac{1}{\xi}}+x}{(y-\varrho)^3}dxdy
+ \frac{
	\left(y^{\frac{1}{\xi}}+x\right)^2
   }{
    (y-\varrho)^4
   }dy^2\right)\,,
\end{equation*}
where $\varrho\in\pmo$, $\kappa\in\Rnz$ and $\xi\in(0,\frac12)\cup(\frac12,1)\cup(1,4]$.
\end{proposition}
\begin{proof}
 Recalling $\eta=\frac{1}{\xi}$ and letting $\kappa=-k$, the normal forms are~\eqref{eqn:C.7.normal.forms.1}.
 It remains to check that there exists no coordinate transformation between two different copies $g[\eta;k,\varrho]$ and $g[\eta;k',\varrho']$ of~\eqref{eqn:C.7.normal.forms.1}.
 As it has been explained in Section~\ref{sec:A.1}, this implicitly follows from the above computations.
 Alternatively, although the computations are a bit more cumbersome, the reader will find it easy to check the statement in a way analogous to the previous cases, studying poles and zeros of the length of the projective vector field.
\end{proof}

\subsubsection{Normal forms (C.8) of Theorem~\ref{thm:normal.forms}}\label{sec:C.8}
Metrics of class C(II) of Theorem~\ref{thm:matveev} have the general form
\begin{equation}\label{eqn:C.II.general.metric}
g[\K1,\K2] =
2\,\frac{x+Y(y)}{(\K1-\K2y)^3}\,dxdy
+\frac{(x+Y(y))^2\,\K2}{(\K1-\K2y)^4}\,dy^2
\end{equation}
where
\begin{equation}\label{eqn:C.II.Y}
Y(y)= e^{\frac{3}{2y}}\,\frac{\sqrt{|y|}}{y-3}
      +\int^y e^{\frac{3}{2s}}\,\frac{\sqrt{|s|}}{(s-3)^2}\,ds\,.
\end{equation}
The metrics $g[0,1]$ corresponds, via~\eqref{eqn:Liouville.tensor.g}, to an eigenvector of $\lie_w$, but $g[1,0]$ does not.

\begin{remark}
 We can construct any metric $g$ in a projective class of type C(II) in Theorem~\ref{thm:matveev}, via Formula~\eqref{eqn:general.metric}, from the following two metrics:
\begin{equation}\label{eqn:basis.metrics.IIC}
 \g1 = g[1,0] = 2(Y+x)\,dxdy\,,\qquad
 \g2 =g[0,1] = -2\,\frac{Y+x}{y^3}\,dxdy+\frac{(Y+x)^2}{y^4}\,dy^2\,,
\end{equation}
where $Y=Y(y)$ is given by~\eqref{eqn:C.II.Y}.
Using integration by parts, one can rewrite it in the simplified form
\begin{equation}\label{eqn:C.II.final.Y}
 Y(y) = \frac12\int^y \frac{e^{\frac{3}{2s}}}{|s|^{\nicefrac32}}\,ds\,,
\end{equation}
where in the final normal form we drop the constant factor $\frac12$ using an obvious coordinate transformation of $x$.
Note that in Section~3.1.3.\ of \cite{matveev_2012}, the function $Y=Y_1'$ is obtained, in the form~\eqref{eqn:C.II.Y}, from a function $Y_1$ that satisfies the ODE
\begin{equation}\label{eqn:IIC.ODE.Y1}
 y^2\,Y_1''-\frac12 (y-3) Y_1'+\frac12 Y_1 =0\,.
\end{equation}
\end{remark}

\noindent Although in theory possible, doing computations similar to those in Sections~\ref{sec:A.2} or~\ref{sec:B.5} appears to be too complicated and cumbersome for the case under consideration here.
We therefore proceed with a slightly different reasoning.
For the case C(II), the action on $\solSp$ of the local flow $\phi_t$ of the projective vector field~$w$ is given by
\[
  \phi_t^*(\K1\a1+\K2\a2)
  = \K1\,e^t\,\a1 + (\K1 t +\K2)\,e^t\,\a2
\]
and thus by choosing $t=-\frac{\K2}{\K1}$ we can obtain a representative on the $\a1$-axis.
First let us compute the Benenti tensor for the metrics $(\g2,\g1)$. It reads
\begin{equation*}
L = L(\g2,\g1) = \begin{pmatrix} y & x+Y(y) \\ 0 & y \end{pmatrix}\,.
\end{equation*}
Note that the Benenti tensor $L$ is preserved up to a constant factor under diffeomorphisms preserving the $\a1-$direction.
Analyzing the transformation behavior of $L$ as we have done in the other cases, we obtain that necessarily
\begin{equation*}
v = Ky\qquad\text{with }K\in\Rnz\,.
\end{equation*}
Let us assert
\[
  \g1 \to K\,(u(x,y)+Y(Ky))\,du\,dy \stackrel{!}{=} (x+Y(y))\,dx\,dy\,.
\]
Thus, using this invariance of the linear spans of $\a1$ and $\a2$, we first find the possible coordinate transformations
\[
 u = Cx+D\,,\qquad v = Ky\,,
\]
and in fact we have
$D+Y(Ky) = C\,Y(y)$.
Differentiating this, we obtain
\begin{equation*}
 K\,\frac{e^{\frac{3}{2Ky}}}{|Ky|^{\frac32}}
 = C\,\frac{e^{\frac{3}{2y}}}{|y|^{\frac32}}
 \quad\Rightarrow\quad
 e^{\frac{3}{2y}\,\left(\frac{1}{K}-1\right)}
 = C\,\frac{K}{|K|^{\frac12}}\,,
\end{equation*}
which confirms $C=K=1$, and thus $D=0$, i.e.\ any coordinate transformation of the considered type is necessarily the identity. We have thus confirmed that the representatives on $\a1$ are mutually non-diffeomorphic, which implies that, if two metrics are linked by a coordinate transformation, they lie on the same $\phi_t$-orbit in $\solSp$.
By Formula~\eqref{eqn:general.metric}, a metric from the projective class $\projcl(\g1,\g2)$ can be written as
\begin{equation}\label{eqn:metric.IIC.general.initial}
 g = g[\K1,\K2] = 2\,\frac{Y(y)+x}{(\K1-\K2\,y)^3}\,dxdy
		 +\frac{\K2\,(Y(y)+x)^2}{(\K1-\K2\,y)^4}\,dy^2\,,
\end{equation}
using the same coordinates as in~\eqref{eqn:basis.metrics.IIC}.
Metrics with $\K1=0$ are eigenvectors of $\lie_w$ and admit a homothetic vector field. We therefore require $\K1\ne0$ in the remainder of this section.
We reparametrize the metrics~\eqref{eqn:metric.IIC.general.initial}, for $\K1\ne0$, by replacing $(\K1,\K2)\to(K,\alpha)$ according to the definitions
\begin{equation*}
 K = \frac{e^{\nicefrac{\K2}{\K1}}}{\K1}
 \qquad\text{and}\qquad
 \alpha = \frac{\K2}{\K1}\,.
\end{equation*}
Thus, the metric~\eqref{eqn:metric.IIC.general.initial} is obtained, in the same local coordinates, as
\begin{equation}\label{eqn:metric.IIC.reparam}
 g = K^3\,e^{-3\alpha}\,\,
	  \left( 2\,\frac{Y(y)+x}{(1-\alpha\,y)^3}\,dxdy
		+\frac{\alpha\,(Y(y)+x)^2}{(1-\alpha\,y)^4}\,dy^2 \right)\,,
\end{equation}
with $\alpha\in\mathds{R}$ and $K\in\Rnz$.
The parameter $K$ characterizes orbits of the projective flow, as we shall see below.

\begin{lemma}\label{la:normal.form.IIC}
There is a local coordinate transformation that maps metric~\eqref{eqn:metric.IIC.reparam} onto a constant multiple of metric~$\g1$ of~\eqref{eqn:basis.metrics.IIC}, i.e.\
\begin{equation}\label{eqn:metric.IIC.transformation}
  K^3\,e^{-3\alpha}\,\,
  \left( 2\,\frac{Y(y)+x}{(1-\alpha\,y)^3}\,dxdy
  +\frac{\alpha\,(Y(y)+x)^2}{(1-\alpha\,y)^4}\,dy^2 \right)
  = 2K^3\,(Y(v)+u)\,dudv\,.
\end{equation}
in terms of new coordinates $u=u(x,y)$, $v=v(y)$.
\end{lemma}

\begin{proof}
Metrics on the same projective orbit are linked by coordinate transformations given by the flow of the projective vector field, i.e.\ solutions to
$\dot{x} = \frac12\,(y-3)\,x +\frac12\,Y_1(y)$,
$\dot{y} = y^2$.
The new coordinates are thus given by $u=x(t_0;x,y)$, $v=y(t_0;y)$.
Now consider the metric on the left hand side of~\eqref{eqn:metric.IIC.transformation},
\begin{equation}\label{eqn:metric.lhs.CII}
 g = K^3\,e^{-3\alpha}\,\,
 \left( 2\,\frac{Y(y)+x}{(1-\alpha\,y)^3}\,dxdy
 +\frac{\alpha\,(Y(y)+x)^2}{(1-\alpha\,y)^4}\,dy^2 \right)\,.
\end{equation}
Applying the mentioned coordinate transformation, $g$ is transformed into a metric of the same form because, for the metrics $g$ and another metric $\til{g}$ isometric to it, $\til{g}=\phi_t^*(g)$ (for a fixed $t$ where $\phi_t$ is the projective flow; compare~\eqref{eqn:III.reparametrization}). Thus we have, with $a=\psi^{-1}(g)$ and $\til{a}=\psi^{-1}(\til{g})$, that
\[
  \til{a} = \psi^{-1}(\til{g}) = \psi^{-1}(\phi_t^*(g))
  = \phi_t^*(\psi^{-1}(g)) = \phi_t^*(a)\,.
\]
The metric $g$ is determined by two parameters,
$\alpha = -\frac{\K2}{\K1}$ and $K = \frac{1}{\K1}\,e^{\nicefrac{\K2}{\K1}}$,
represented in terms of the initially introduced parameters $\K1,\K2$.
\begin{figure}
  \begin{center}
	\scalebox{.4}{\input{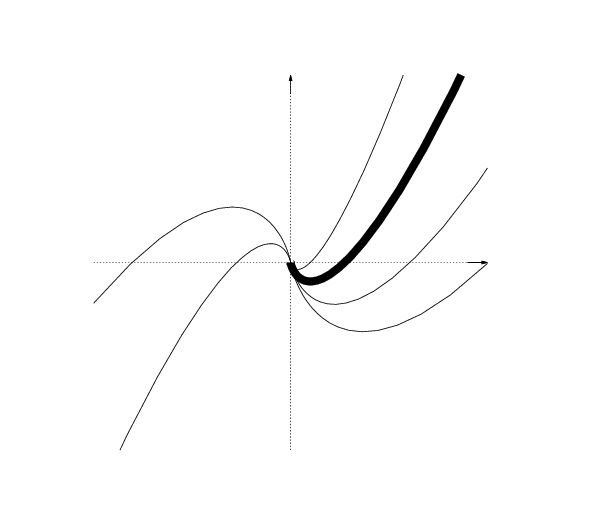}}
	\hspace{0.8cm}
	\scalebox{.4}{\input{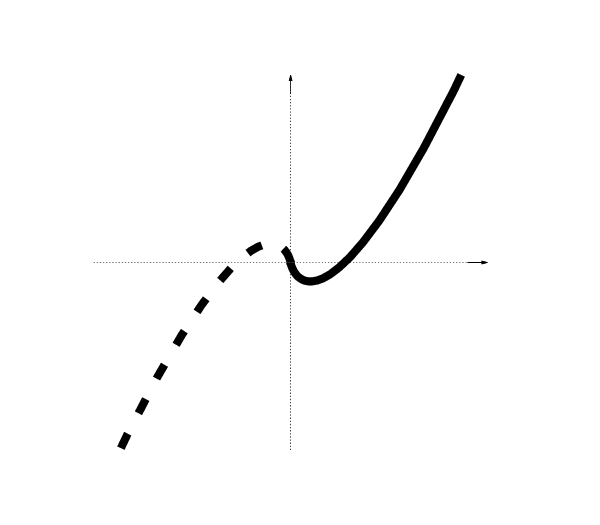}}
	\hspace{0.8cm}
	\scalebox{.4}{\input{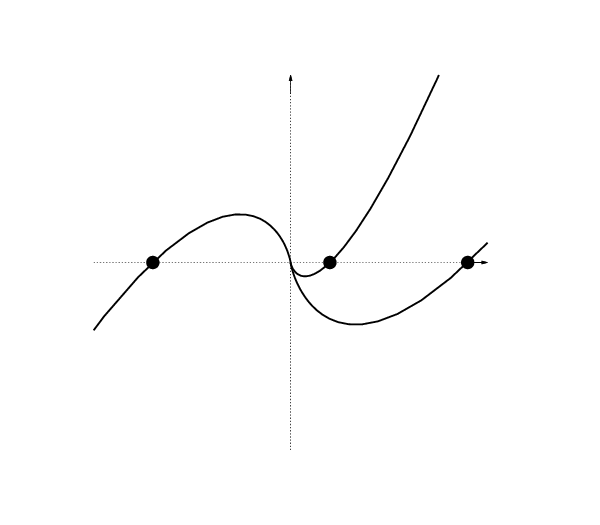}}
  \end{center}
  \vspace{-1cm}
  \caption{Graph of the projective orbits for a metric $g=\psi(\K1\,\a1+\K2\,\a2)$ of type II. The panel on the left depicts several generic orbits, the thick curve being one example. The graph in the middle illustrates the special case $h=1$. The right panel shows how representatives under the projective flow are chosen: Metrics on the same projective orbit are linked by coordinate transformations and we chose the metric on the intersection of an orbit with the axis of abscissa (i.e.\ proportional to $\g1$).}\label{fig:orbit.II}
\end{figure}
Looking at Figure~\ref{fig:orbit.II}, the canonical choice for a normal form is the intersection of the (thick=projective orbit) curve with the axis of abscissa.
The value $t_0$ hence is determined by
\[
  \K2' = (\K1\,t_0+\K2)\,e^{t_0} = 0
  \qquad\Rightarrow\qquad
  t_0=-\frac{\K2}{\K1}=-\alpha\,.
\]
The resulting metric, isometric to~$g$, is thus a multiple of $\g1$, and the multiplicative factor is determined by
\[
  \K1' = \K1\,e^{t_0} = \frac{e^\alpha}{K}\,e^{-\alpha} = \frac{1}{K}\,
\]
and thus, in new coordinates $(u,v)$, the metric \eqref{eqn:metric.lhs.CII} assumes the form
\begin{equation}\label{eqn:metric.rhs.CII}
 \til{g} = g[\kappa] = \kappa\,(Y(v)+u)\,dudv
\end{equation}
with $\kappa:=2K^3$.
\end{proof}

\noindent Figure~\ref{fig:orbit.II} illustrates geometrically how the normal forms are chosen on the orbits of the projective vector field.

\begin{lemma}\label{la:nonisometry.IIC}
Two different copies $g=g[\kappa]$ and $\til{g}=g[\til{\kappa}]$ of~\eqref{eqn:metric.rhs.CII} are non-isometric.
\end{lemma}

\begin{proof}
	One only needs to prove that the condition $\tilde{\kappa}=\kappa$ is necessary. Thus assume $g$ and $\tilde{g}$ are connected by a coordinate transformation. This means that there exists a coordinate transformation $(x,y)\mapsto(u,v)$, $u=u(x,y), v=v(x,y)$ mapping one metric onto the other. There are only two possibilities: $u=u(x), v=v(y)$ or $u=u(y), v=v(x)$.
	Assume first $u=u(x), v=v(y)$, so that the following equation holds:
	\[
	\kappa\,(Y(y)+x)\,dxdy = \tilde{\kappa}\,\left(Y\big(v(y)\big)+u(x)\right)\,u_x\,v_y\,dxdy\,.
	\]
	The above equation implies
	\begin{equation}\label{eq.lemma9}
	\kappa\,(Y(y)+x) = \tilde{\kappa}\,\left(Y\big(v(y)\big)+u(x)\right)\,u_x\,v_y\,.
	\end{equation}
	By differentiating Equation~\eqref{eq.lemma9} twice w.r.t.\ $x$, we obtain
	\begin{equation}\label{eq.kkk}
    0 = \til{\kappa}u_y\,\left( Y\big(v(y)\big)u_{xxx} +3u_xu_{xx} +uu_{xxx} \right)\,.
	\end{equation}
	Since $v(y)$, being a coordinate transformation, is invertible,
we can infer that $Y'(v(y))u_{xxx}=0$ (by differentiating the right hand side of \eqref{eq.kkk}). Thus, $u_{xxx}=0$, since $Y'(v(y))$ is not identically zero, so that
	\[
	  u(x) = \frac{e_1x^2}{2}+c_1x+d_1\,.
	\]
	Substituting this back into \eqref{eq.lemma9} yields
	\[
	  \kappa\,(Y(y)+x) = \til{\kappa}\,\left( Y\big(v(y)\big)+\frac{e_1x^2}{2}+c_1x+d_1\right) (e_1x+c_1)\,v_y\,,
	\]
	which is polynomial in $x$. Analyzing each coefficient w.r.t.\ $x$ separately, one finds from the cubic term $e_1=0$, and from the linear term that $v_y$ is a non-zero constant, i.e.\ $v=c_2y+d_2$, $c_2\in\mathds{R}\setminus\{0\},d_2\in\mathds{R}$. Taking this into account, by differentiating Equation~\eqref{eq.lemma9} w.r.t. $y$, we obtain
	\begin{equation}\label{eq.kkkk}
	\kappa Y'=\tilde{\kappa}\,c_2^2\, Y_v\, u_x
	\end{equation}
	implying that $u_x$ is a non-zero constant, i.e.\ $u=c_1x+d_1$, $c_1\in\mathds{R}\setminus\{0\},d_1\in\mathds{R}$, so that equation \eqref{eq.kkk} becomes
	\begin{equation}\label{eq.k.ktilde.c1.c2}
	\kappa = \tilde{\kappa}c_1^2c_2
	\end{equation}
	whereas Equation~\eqref{eq.kkkk} assumes the form
	\begin{equation}\label{eq.hope.pre.final}
	\frac{\kappa}{2}\frac{e^{\frac{3}{2y}}}{|y|^{\frac{3}{2}}} = \frac{\tilde{\kappa}c_1c_2^2}{2} \frac{e^{\frac{3}{2(c_2y+d_2)}}}{|c_2y+d_2|^{\frac{3}{2}}}\,,
	\end{equation}
	which, on account of \eqref{eq.k.ktilde.c1.c2}, reduces to
	\begin{equation}\label{eq.hope.final}
	c_1\frac{e^{\frac{3}{2y}}}{|y|^{\frac{3}{2}}} = c_2 \frac{e^{\frac{3}{2(c_2y+d_2)}}}{|c_2y+d_2|^{\frac{3}{2}}}\,.
	\end{equation}
	Since Equations~\eqref{eq.hope.pre.final} and~\eqref{eq.hope.final} hold for any $y$, we have that $d_2=0$, $c_2=1$ and finally $c_1=1$, and therefore $\kappa=\tilde{\kappa}$ in view of \eqref{eq.k.ktilde.c1.c2}.
	\medskip

	Now assume $u=u(y)$ and $v=v(x)$. The analogue of Equation~\eqref{eq.lemma9} becomes
	\begin{equation}\label{eqn:IIC.nonisometry.start}
	 \kappa(Y(y)+x) = \til{\kappa}\,\left(Y\big(v(x)\big)+u(y)\right)\,u_y v_x\,.
	\end{equation}
	Taking two derivatives w.r.t.\ $x$, we arrive at
	\begin{equation}\label{eqn:IIC.nonisometry.Dxx}
	  \til{\kappa}u_y\,\left(
	  Y''(v(x))\,v_x^3 +2v_xv_{xx}\,Y'(v(x)) +Y'(v(x))v_{xx} +Y(v(x))v_{xxx}+u(y)\,v_{xxx}
	  \right) = 0\,.
	\end{equation}
	Dividing by $\til{\kappa}u_y$ (recall that $u_y$ is non-zero) and taking a derivative w.r.t.\ $y$ of this equation, we thus find $v_{xxx}=0$, which implies that $v(x)$ is of the form $v(x)=\frac{e_1x^2}{2}+c_1x+d_1$, $e_1,c_1,d_1\in\mathds{R}$.
	Since $v_{xxx}=0$, from \eqref{eqn:IIC.nonisometry.Dxx} we obtain
	\begin{equation}\label{eqn:IIC.nonisometry.ODE.v}
	  Y''(v(x)) v_x^3 +(2v_x+1)\,v_{xx}\,Y'(v(x)) = 0\,.
	\end{equation}
	Differentiating~\eqref{eqn:IIC.ODE.Y1} w.r.t.\ $y$, we obtain
	\begin{equation}\label{eqn:IIC.nonisometry.ODE.Y}
	 y^2 Y''(y) + \frac32\,(y+1)\,Y'(y) = 0\,.
	\end{equation}
	Eliminating $Y''$ from~\eqref{eqn:IIC.nonisometry.ODE.v} by grace of~\eqref{eqn:IIC.nonisometry.ODE.Y}, and replacing $v(x)$, a polynomial equation in $x$ is obtained. The system of equations on $e_1$, $c_1$ and $d_1$ obtained from its coefficients w.r.t.\ $x$ can straightforwardly be solved and yields the solution $e_1=c_1=0$, $d_1=-1$. This, however, is a contradiction since it would imply $v(x)=-1$, but $v(x)$, being part of a coordinate transformation, cannot be constant.
\end{proof}

\noindent Summing up, we arrive at the following proposition.

\begin{proposition}\label{prop:normal.forms.CII}
A metric of the type C(II) in Theorem~\ref{thm:matveev} that admits exactly one essential projective vector field can, locally by a suitable coordinate transformation, be mapped onto one and only one of the following metrics, where $\kappa\in\Rnz$,
\[
  g = \kappa\,(Y(y)+x)\,dxdy\,,
\]
with $Y(y)$ given by Equation~\eqref{eqn:C.II.final.Y}.
\end{proposition}

\subsubsection{Normal forms (C.9) of Theorem~\ref{thm:normal.forms}}\label{sec:C.9}
We can treat case C(III) similar to the case C(II). A metric of class C(III) has the form~\eqref{eqn:C.II.general.metric}, except that we have to replace $Y(y)$ by $Y(y;\lambda)=Y_\lambda(y)$.
Thus we obtain, by help of Formula~\eqref{eqn:general.metric}, a metric in $\projcl(\g1,\g2)$ in the form
\begin{equation}\label{eqn:metric.C.III.general.initial}
g = g[\lambda;\K1,\K2] = 2\,\frac{Y(y)+x}{(\K1-\K2\,y)^3}\,dxdy
+\frac{\K2\,(Y(y)+x)^2}{(\K1-\K2\,y)^4}\,dy^2\,.
\end{equation}
where
\begin{equation*}
Y_\lambda(y) = e^{-\frac32\lambda\arctan(y)}\,
\frac{\sqrt[4]{y^2+1}}{y-3\lambda}
+\int^y e^{-\frac32\lambda\arctan(\xi)}\,
\frac{\sqrt[4]{\xi^2+1}}{(\xi-3\lambda)^2}\,d\xi
= \int^y \frac{e^{-\frac32\lambda\arctan(\xi)}}{|\xi^2+1|^{\nicefrac34}}\,d\xi\,.
\end{equation*}
Note that $\lie_w$ does not admit real eigenvalues in the present case, so no further restrictions on $\K1$,$\K2$ apply.

\begin{remark}
 Metric~\eqref{eqn:metric.C.III.general.initial} is obtained, using Formula~\eqref{eqn:general.metric}, from the two metrics~$\g1$ and~$\g2$ given for case~C(III) in Theorem~\ref{thm:matveev}.
The function $Y=Y_\lambda(y)=Y_1'(y)$ is obtained from the ODE
$(y^2+1)\,Y_1''-\frac12 (y-3\lambda) Y_1'+\frac12 Y_1 =0$,
see~\cite{matveev_2012}.
W.l.o.g., the solution is given by
\begin{equation}\label{eqn:Y1.case.C.III}
Y_1(y)
= (y-3\lambda)\,\int^y e^{-\frac32\lambda\arctan(s)} \frac{\sqrt[4]{s^2+1}}{(\xi-3\lambda)^2}\,ds\,.
\end{equation}
\end{remark}
\noindent Let us again use the parametrization~\eqref{eqn:III.reparametrization}, such that the metric~\eqref{eqn:metric.C.III.general.initial}, in the same local coordinates, reads
\begin{equation}\label{eqn:metric.C.III.reparam}
g = K^{-3}\,e^{-3\lambda\theta}\,\,
	\left( 2\,\frac{Y(y)+x}{(\sin\theta-y\,\cos\theta)^3}\,dxdy
	 +\cos(\theta)\frac{(Y(y)+x)^2}{(\sin\theta-y\,\cos\theta)^4}\,dy^2 \right)\,,
\end{equation}
with $\theta\in\angl$ and $K\in\Rnz$.
We prove that any two isometric metrics lie on the same orbit under the action of the projective flow. Thus these orbits are completely characterized by the constant~$K$.

\begin{lemma}\label{la:C.III.lambda.zero}
There is a local coordinate transformation that maps metric~\eqref{eqn:metric.C.III.reparam} onto a constant multiple of the metric~$\g1$, i.e.\
\begin{equation*}
K^{-3}\,e^{-3\lambda\theta}\,\,
\left( 2\,\frac{Y(y)+x}{(\sin\theta-y\,\cos\theta)^3}\,dxdy
+\cos(\theta)\,\,\frac{(Y(y)+x)^2}{(\sin\theta-y\,\cos\theta)^4}\,dy^2 \right)
= 2\,K^3\,(Y(v)+u)\,dudv\,.
\end{equation*}
\end{lemma}

\begin{proof}
The proof is analogous to that of Proposition~\ref{prop:normal.forms.CII}.
For the normal forms we make the following choices:

\paragraph{Case $\lambda=0$:}
We have $K=\const$ and thus the projective orbits are concentric circles around the origin, c.f.\ Figures~\ref{fig:orbit.III} and~\ref{fig:nf.C.III}. For the normal forms we can therefore choose freely (by grace of a rotation) the value of $\theta$ in \eqref{eqn:metric.C.III.reparam}, and if we choose $\theta=\tfrac{\pi}{2}$, then $\cos(\theta)=0$. The resulting metric is a constant multiple of $\g1$, with the factor being a positive real number.

\paragraph{Case $\lambda\ne0$:}
The orbital invariant is still $K$, but the orbits are spirals in $\solSp$, see the middle or right panel of Figure~\ref{fig:nf.C.III}.
Since the radial distance of any of these spirals is strictly growing with increasing $\theta$, each orbit crosses the unit circle exactly once.
Following the spiraling orbit of the projective flow, the distance from the origin is given by $\zeta(t)=K e^{\lambda\theta}\,e^{\lambda\,t}$.
We require that $\zeta(t)=1$, i.e.\ we choose the intersection of this orbit with the unit sphere, which is reached at time $t=-\frac{\ln K}{\lambda}-\theta$.
However, it is more convenient to choose the normal forms again on the axis given by multiples of $\g1$.
We use the parametrization~\eqref{eqn:III.reparametrization}, but recall Remark~\ref{rmk:III.parameter.values}, i.e.\ we let $K=e^{\lambda\alpha}$ where $\alpha\in\angl$.
Choosing $\theta=\tfrac{\pi}{2}+2\pi N$ ($N\in\mathds{Z}$) for the normal form, we obtain a metric of the form~\eqref{eqn:metric.C.III.general.initial} with
\[
  \K1 = K e^{\lambda (\frac{\pi}{2}+2\pi N)}
  \qquad\text{and}\qquad
  \K2 = 0\,,
\]
and choosing $N$ appropriately permits us further to have
$\K1 = e^{\lambda\beta}$ with $\beta\in\angl$.
\end{proof}

\noindent Figure~\ref{fig:nf.C.III} illustrates how we can choose representatives for the orbits of the flow of~$w$.
Summing up, we have found a representative in the form
\begin{equation}\label{eqn:normal.form.C.9}
 g = g[\lambda;\kappa] = \kappa\,(Y(y)+x)\,dxdy
\end{equation}
with $\kappa\in\Rnz$ if $\lambda>0$ and $\kappa=e^{\lambda\alpha}$ if $\lambda=0$, $\alpha\in\angl$. The following lemma shows that this form is already optimal.

\begin{figure}
\begin{center}
 \scalebox{.45}{\input{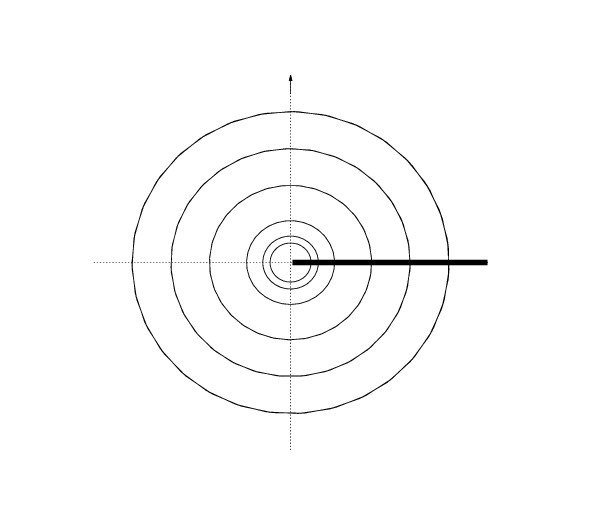}}
 \hspace{1cm}
 \scalebox{.45}{\input{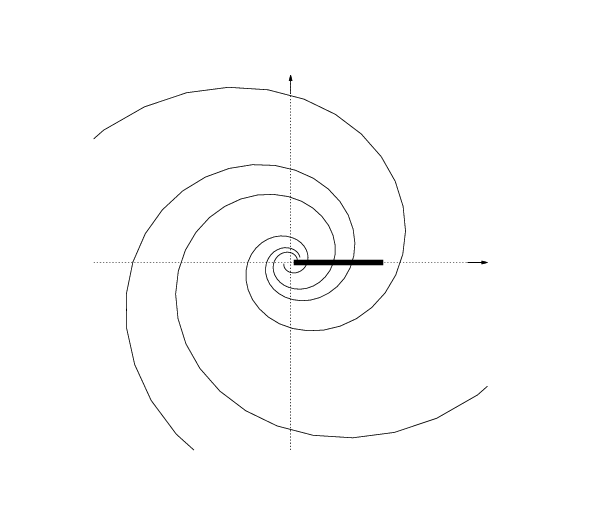}}
 \hspace{1cm}
 \scalebox{.45}{\input{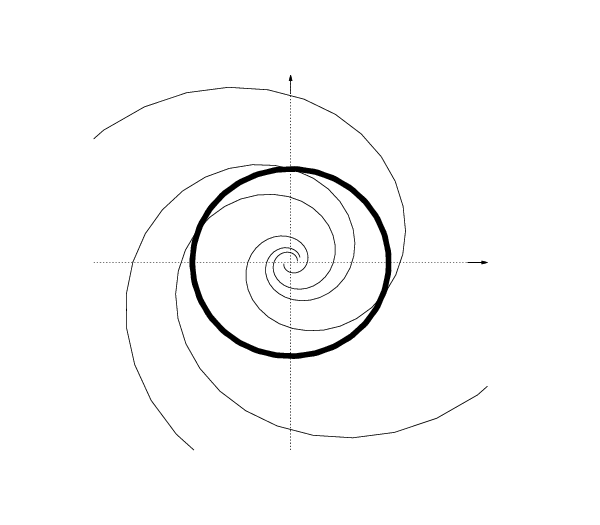}}
\end{center}
\vspace{-1cm}
\caption{Illustration of the choice of representatives for metrics of type III in Theorem~\ref{thm:matveev}. The thin lines show orbits of the projective vector field; at any point on the thick lines passes only one orbit and such a point is the representative for this orbit. In the left panel, the situation for $\lambda=0$ is shown (the orbits are circles; we choose the representatives to be the intersection of each circle with the positive $\K1$-axis). The middle panel illustrates $\lambda>0$ with the representatives as in Lemma~\ref{la:C.III.lambda.zero}. Finally, the right panel shows an alternative choice of representatives for $\lambda>0$.}\label{fig:nf.C.III}
\end{figure}

\begin{lemma}\label{la:C.III.lambda.nonzero}
Two copies of metric~\eqref{eqn:normal.form.C.9}, say $g[\lambda;\kappa]$ and $g[\til{\lambda};\til{\kappa}]$ are isomorphic if and only if $\til{\lambda}=\lambda$ and $\til{\kappa}=\kappa$.
\end{lemma}
\begin{proof}
Recall that $\til{\lambda}=\lambda$ follows from Lemma~\ref{la:lambda.different.classes}. So, assume $\til{\lambda}=\lambda$.
The proof for $\til{\kappa}=\kappa$ is analogous to that of Lemma~\ref{la:nonisometry.IIC}. Instead of \eqref{eq.hope.pre.final} we find
\[
  \kappa\,\frac{e^{\frac32 \lambda \arctan(y)}}{|y^2+1|^{\nicefrac34}}
  = \til{\kappa}c_1c_2^2\,\frac{e^{\frac32 \lambda \arctan(c_2y+d_2)}}{|(c_2y+d_2)^2+1|^{\nicefrac34}}\,,
\]
and instead of \eqref{eq.hope.final} we find
\[
  c_1\,\frac{e^{\frac32 \lambda \arctan(y)}}{|y^2+1|^{\nicefrac34}}
= c_2\,\frac{e^{\frac32 \lambda \arctan(c_2y+d_2)}}{|(c_2y+d_2)^2+1|^{\nicefrac34}}\,.
\]
Thus, it follows that $d_2=0$, $c_2=1$ and finally $c_1=1$. Therefore, $\til{\kappa}=\kappa$ as claimed.
\end{proof}

\noindent Putting together Lemma~\ref{la:C.III.lambda.zero} and Lemma~\ref{la:C.III.lambda.nonzero}, we arrive at

\begin{proposition}\label{prop:normal.forms.CIII}
A metric of the type C(III) in Theorem~\ref{thm:matveev} that admits exactly one essential projective vector field can, locally by a suitable coordinate transformation, be mapped onto one and only one of the following metrics,
\begin{align*}
 &\text{For $\lambda=0$:}
 & g &= \kappa\,(Y(\lambda=0,y)+x)\,dxdy
 && \text{with $\kappa>0$} \\
 &\text{For $\lambda>0$:}
 & g &= e^{\lambda\,\alpha}\,(Y(\lambda,y)+x)\,dxdy
 &&\text{with $\alpha\in\angl$}\,.
\end{align*}
\end{proposition}

\noindent The collection of Propositions~\ref{prop:normal.forms.AI},
\ref{prop:normal.forms.AII}, \ref{prop:normal.forms.AIII},
\ref{prop:normal.forms.BI}, \ref{prop:normal.forms.BII}, \ref{prop:normal.forms.BIII}
\ref{prop:normal.forms.CIb}, \ref{prop:normal.forms.CII}, \ref{prop:normal.forms.CIII} and
\ref{prop:normal.forms.CIa}
form the statement and the proof of Theorem~\ref{thm:normal.forms}.

\begin{remark}
 Let us quickly comment on how to find Proposition~\ref{prop:normal.forms.CIII} using Benenti tensors.
 Let $\a1,\a2$ correspond to $\g1,\g2$ via~\eqref{eqn:Liouville.tensor.g}.
 Their Benenti tensor $L = L(\a1,\K1\a1+\K2\a2)$ is
 \[
  L=\begin{pmatrix} \K2y+\K1 & \K2\,(x+Y_\lambda(y))+\K1 \\  & \K2y+\K1 \end{pmatrix}\,,
  \quad\text{with}\quad
  Y_\lambda(y) = \int^y\frac{e^{-\frac32\lambda\arctan(s)}}{|s^2+1|^{\nicefrac34}}\,ds\,.
 \]
 As in Sections~\ref{sec:C.7} and~\ref{sec:C.8}, we find the possible form-preserving transformation $(x,y)\to(u,v)$,
 $x = Ku+C$, $y = Kv+s$,
 with the constants $K\ne0$, and $C,s\in\R$.
 Any metric in the projective class $\projcl(\g1,\g2)$ permits representatives such that
 $g \sim e^\alpha\,\g1$ with
 $\alpha\in(0,\tfrac{2\pi}{\lambda}]$,
 where $\frac{2\pi}{\lambda}$ is to be understood as $\infty$ if $\lambda=0$.
 Thus, since any isometry of the desired kind maps representatives onto representatives, we must require
 \[
  Kx+C+Y_\lambda(Ky+s)=e^{\beta}\,(x+Y_\lambda(y))\,,
  \quad\text{with}\quad
  \beta\in\big(0,\tfrac{2\pi}{\lambda}\big]\,.
 \]
 Hence, $K = e^\beta$.
 Continuing the computations further is technically involved, which is why we proved Proposition~\ref{prop:normal.forms.CIII} in a different manner.
\end{remark}

\subsection{Proof of Corollary~\ref{cor:projective.classes}}\label{sec:proof.projective.classes}
As a by-product of Theorem~\ref{thm:normal.forms} we obtain Corollary~\ref{cor:projective.classes}.
This is a classification of the projective classes, and a sharper result than the one in~\cite{matveev_2012}.
In fact, a careful examination of the proof undertaken in the previous sections reveals that it also proves Corollary~\ref{cor:projective.classes}.
To substantiate this, we note first that we have seen, implicitly in Section~\ref{sec:proof.normal.forms}, that there are two kinds of parameters involved in Theorem~\ref{thm:normal.forms}. The first kind of parameters pertains to the projective class, namely $C,h,\varepsilon$ and $\lambda$ or, respectively, $\xi$. The second kind are such parameters that specify a metric within a certain projective class, namely $\kappa,\varrho,\theta$. These parameters are obtained from $\K1$ and $\K2$ after specifying the first-kind parameters (i.e., after fixing the space~$\solSp$ of solutions to~\eqref{eqn:linear.system}). The case of degree of mobility~3 gives rise to only one projective class, as proven in Proposition~\ref{prop:dom3.proj.equiv}.

Let us denote, abstractly, $g[c_1,\dots,c_r;d_1,\dots,d_m]$ whereby we mean that the metric~$g$ (from Theorem~\ref{thm:normal.forms}) is characterized by parameters $c_1,\dots,c_r$ that specify its projective class, and parameters $d_1,\dots,d_m$ that ``locate'' the specific metric (or rather its class under coordinate transformations) within its projective class. Now, if two different sets of parameters $(c_1',\dots,d_m')$ and $(c_1,\dots,d_m)$ would characterize projectively equivalent metrics, their projective classes would be identical.
If $c_i'=c_i$ for all $i$, we have proven that $d_j'\not=d_j$ for some $j$ (i.e.\ the metrics are merely projectively equivalent), or otherwise the metrics are already isometric, i.e.\ linked by a coordinate transformation.
Therefore, there must be at least one $1\leq k\leq r$ such that $c_k'\not=c_k$, but $(c_1',\dots,c_r')$ and $(c_1,\dots,c_r)$ would still specify the same projective class. Thus, there would exist a pair of metrics $g_a=g[c_1',\dots,c_r';d_1^a,\dots,d_m^a]$, $g_b=g[c_1,\dots,c_r;d_1^b,\dots,d_m^b]$ from these two different representations of the same projective class, such that $g_a$ and $g_b$ are connected by a coordinate transformation.
But this contradicts Theorem~\ref{thm:normal.forms}, which we have just proven.

\section*{Acknowledgments}
The authors wish to thank Vladimir Matveev for numerous suggestions, helpful discussions and comments.
They gratefully acknowledge support from the project \emph{FIR-2013 Geometria delle equazioni differenziali} of \emph{Istituto Nazionale di Alta Matematica} (INdAM).
Gianni Manno acknowledges, as well, support by ``Starting grant per giovani ricercatori'' of Politecnico di Torino (code \emph{53\_RSG16MANGIO}), the PRIN project 2017
``Real and Complex Manifolds: Topology, Geometry and holomorphic dynamics'',
the project ``Connessioni proiettive, equazioni di Monge-Amp\`ere e sistemi integrabili'' (INdAM)
and ``MIUR grant Dipartimenti di Eccellenza 2018-2022
(E11G18000350001)''.
Gianni Manno is a member of GNSAGA of INdAM.
Andreas Vollmer acknowledges his previous membership with GNSAGA. This paper was partially written when Andreas Vollmer was a research fellow of INdAM.
He is now a research fellow of the German Research Foundation (Deutsche Forschungsgemeinschaft DFG, project number 353063958).

\printbibliography

\noindent We acknowledge the use of Sagemath and Maple\textsuperscript{TM} in computations referred to in this paper. Maple is a trademark of Waterloo Maple Inc.
Sagemath is a free and open-source mathematics software system licensed under the GNU General Public License.

\end{document}